\documentclass[11pt]{amsart}

\usepackage{amsmath}
\usepackage{amsfonts}
\usepackage{amssymb}
\usepackage{amsthm}
\usepackage{mathtools}
\usepackage{url}
\usepackage{comment}
\usepackage{bbold}
\usepackage{appendix}
\usepackage{upgreek}
\usepackage{mathrsfs}

\usepackage{extsizes}
\usepackage[margin=3cm]{geometry}

\usepackage{enumitem}
\setenumerate[0]{label=(\roman*)}

\usepackage{hyperref}
 \hypersetup{
     colorlinks=true,
     linkcolor=red,
     filecolor=blue,
     citecolor = blue,      
     urlcolor=cyan,
     }

\DeclareMathOperator*{\EE}{\mathbb{E}}

\newcommand{\RR}{\mathbb{R}}
\newcommand{\CC}{\mathbb{C}}
\newcommand{\NN}{\mathbb{N}}
\newcommand{\ZZ}{\mathbb{Z}}
\newcommand{\QQ}{\mathbb{Q}}

\newcommand{\TT}{\mathbb{T}}

\newcommand{\A}{\mathcal{A}}
\newcommand{\C}{\mathcal{C}}

\newcommand{\F}{\mathcal{F}}
\newcommand{\G}{\mathcal{G}}

\renewcommand{\L}{\mathcal{L}}
\newcommand{\Z}{\mathcal{Z}}

\newcommand{\I}{\mathcal{I}}

\newcommand{\W}{\mathcal{W}}
\newcommand{\T}{\mathcal{T}}
\newcommand{\HK}{\mathcal{HK}}

\renewcommand{\P}{\mathcal{P}}
\newcommand{\X}{\mathcal{X}}
\newcommand{\K}{\mathcal{K}}
\newcommand{\Y}{\mathcal{Y}}

\newcommand{\Lip}{{\rm{Lip}}}
\newcommand{\poly}{{\rm{poly}}}

\newcommand{\Span}{{\rm{Span}}}

\newcommand{\norm}[1]{\left\Vert #1\right\Vert}
\newcommand{\nnorm}[1]{\lvert\!|\!| #1|\!|\!\rvert}

\theoremstyle{theorem}
\newtheorem{theorem}{Theorem}[section]
\newtheorem{definition}[theorem]{Definition}
\newtheorem{proposition}[theorem]{Proposition}

\newtheorem{corollary}[theorem]{Corollary}
\newtheorem{conjecture}[theorem]{Conjecture}
\newtheorem{lemma}[theorem]{Lemma}

\theoremstyle{definition}

\title{On several notions of complexity of polynomial progressions}
\author{Borys Kuca}
\date{}

\begin{document}

\maketitle

\begin{abstract}

For a polynomial progression $$(x,\; x+P_1(y),\; \ldots,\; x+P_{t}(y)),$$ we define four notions of complexity: Host-Kra complexity, Weyl complexity, true complexity and algebraic complexity. The first two describe the smallest characteristic factor of the progression, the third one refers to the smallest-degree Gowers norm controlling the progression, and the fourth one concerns algebraic relations between terms of the progressions. We conjecture that these four notions are equivalent, which would give a purely algebraic criterion for determining the smallest Host-Kra factor or the smallest Gowers norm controlling a given progression. We prove this conjecture for all progressions whose terms only satisfy homogeneous algebraic relations and linear combinations thereof. This family of polynomial progressions includes, but is not limited to, arithmetic progressions, progressions with linearly independent polynomials $P_1, \ldots, P_t$ and progressions whose terms satisfy no quadratic relations. For progressions that satisfy only linear relations, such as $$(x,\; x+y^2,\; x+2y^2,\; x+y^3,\; x+2y^3),$$ we derive several combinatorial and dynamical corollaries: (1) an estimate for the count of such progressions in subsets of $\ZZ/N\ZZ$ or totally ergodic dynamical systems; (2) a lower bound for multiple recurrence; (3) and a popular common difference result in $\ZZ/N\ZZ$. Lastly, we show that Weyl complexity and algebraic complexity always agree, which gives a straightforward algebraic description of Weyl complexity.

%we derive an asymptotic estimate for the count of certain polynomial progressions in subsets of finite fields and an analogous result in ergodic theory. Additionally, we show that Weyl complexity and algebraic complexity always agree, which gives a straightforward algebraic description of Weyl complexity.
\end{abstract}

%%%
%%%
%%% Introduction
%%%
%%%

\section{Introduction}

%Throughout the paper, the letter $t$ denotes a (fixed) positive integer. 
A polynomial $P\in\RR[y]$ is \emph{integral} if $P(\ZZ)\subset \ZZ$ and $P(0)=0$. For $t\in\NN_+$, an \emph{integral polynomial progression} of length $t+1$ is a tuple $\vec{P}\in\RR[x,y]^{t+1}$ given by
\begin{align*}
    \vec{P}(x,y) = (x, \; x+P_1(y), \; \ldots, \; x+P_{t}(y))
\end{align*}
for distinct integral polynomials $P_1, \ldots, P_{t}$. We moreover say that a set $A\subset\NN$ contains $\vec{P}(x,y)$ for some $x,y\in\NN$ if $\vec{P}(x,y)\in A^{t+1}$. 
A major result on integral polynomial progressions is the polynomial Szemer\'{e}di theorem by Bergelson and Leibman, which extends the famous theorem of Szemer\'{e}di on arithmetic progressions.

\begin{theorem}[Polynomial Szemer\'{e}di theorem]\cite{bergelson_leibman_1996}\label{polynomial Szemeredi theorem}
Let $t\in\NN_+$ and $\vec{P}\in\RR[x,y]^{t+1}$ be an integral polynomial progression, and suppose that $A\subseteq\NN_+$ is dense\footnote{Meaning that $\limsup\limits_{N\to\infty}\frac{|A\cap[N]|}{N}>0$, where $[N]=\{1, \ldots, N\}$.}. Then $A$ contains $\vec{P}(x,y)$ for some $x,y\in\NN_+$. 
\end{theorem}

Theorem \ref{polynomial Szemeredi theorem} can be deduced from the following ergodic theoretic statement using the Furstenberg correspondence principle.

\begin{theorem}\label{theorem on multiple ergodic averages are positive}\cite{bergelson_leibman_1996, host_kra_2005a}
Let $(X,\X,\mu, T)$ be an invertible measure-preserving dynamical system, $t\in\NN_+$ and $\vec{P}\in\RR[x,y]^{t+1}$ be an integral polynomial progression. If $\mu(A)>0$ for $A\in\X$, then
\begin{align}\label{multiple ergodic averages are positive}
    \lim_{N\to\infty}\EE_{n\in [N]}\mu(A\cap T^{P_1(n)}A \cap \cdots \cap T^{P_{t}(n)}A) > 0,
\end{align}
where $[N]=\{1, \ldots, N\}$ and $\EE_{x\in X} = \frac{1}{|X|}\sum_{x\in X}$ for any set $X$.
\end{theorem}

To prove Theorem \ref{polynomial Szemeredi theorem}, one thus needs to understand limits of multiple ergodic averages of the form
\begin{align}\label{multiple ergodic averages}
    \EE_{n\in [N]}T^{P_1(n)} f_1 \cdots T^{P_{t}(n)} f_{t}
\end{align}
for $f_1, \ldots, f_{t}\in L^\infty(\mu)$. By a remarkable result of Host and Kra \cite{host_kra_2005a, host_kra_2005b}, there exists a family of {factors}\footnote{The definitions of factors, Weyl systems, nilsystems, and other concepts from ergodic theory and higher order Fourier analysis used in the introduction will be provided in subsequent sections.} 
%\begin{align}\label{Host-Kra factors}
 %   \ldots \to \Z_k \to Z_{k-1} \to \ldots \to \Z_1 \to \Z_0,
%\end{align}
$(\Z_s)_{s\in\NN}$, 
called henceforth \emph{Host-Kra factors}, with the property that weak or $L^2$ limits of expressions of the form (\ref{multiple ergodic averages}) remain unchanged if we project any of the functions $f_i$ onto one of the factors $\Z_s$ for some $s$ dependent on $\vec{P}$ and $i$. 
%If a factor $\Y$ of $\X$ has the property that the $L^2$-limit of  If this is the case then we say that the factor $\Z_k$ is \emph{characteristic} for $\vec{P}$ at $i$, as in the following definition.
\begin{definition}[Characteristic factors]\label{characteristic factors}
Let $(X,\X,\mu,T)$ be an invertible measure-preserving dynamical system, $t\in\NN_+$ and ${\vec{P}\in\RR[x,y]^{t+1}}$ be an integral polynomial progression.

Suppose that $1\leqslant i\leqslant t$. A factor $\Y$ of $\X$ is \emph{characteristic for the $L^2$-convergence of $\vec{P}$ at $i$} if for all choices of $f_1, \ldots, f_{t}\in L^\infty(\mu)$, the $L^2$-limit of (\ref{multiple ergodic averages}) is 0 whenever $\EE(f_i|\Y)=0$.

Similarly, suppose that $0\leqslant i\leqslant t$. A factor $\Y$ of $\X$ is \emph{characteristic for the weak convergence of $\vec{P}$ at $i$} if for all choices of $f_0, \ldots, f_{t}\in L^\infty(\mu)$, the weak limit of (\ref{multiple ergodic averages}), i.e. the expression
\begin{align}\label{weak limit}
    \lim_{N\to\infty}\EE_{n\in [N]}\int_X f_0 \cdot T^{P_1(n)} f_1 \cdots T^{P_{t}(n)} f_{t} d\mu,
\end{align}
is 0 whenever $\EE(f_i|\Y)=0$.
\end{definition}

\begin{theorem}[\cite{host_kra_2005a, leibman_2005a}]\label{Host-Kra factors are characteristic}
Let $t\in\NN_+$. For each integral polynomial progression $\vec{P}\in\RR[x,y]^{t+1}$, there is $s\in\NN$ such that for all invertible ergodic systems $(X,\X,\mu,T)$, the factor $\Z_s$ is characteristic for the $L^2$ convergence of $\vec{P}$ at $i$ for all $0\leqslant i\leqslant t$.
\end{theorem}

The utility of Host-Kra factors comes from the fact that they are inverse limits of \emph{nilsystems} \cite{host_kra_2005b}, and so understanding (\ref{multiple ergodic averages}) for arbitrary systems comes down to proving certain equidistribution results on spaces called nilmanifolds that possess rich algebraic structure.
%Some of Host-Kra factors take up familiar forms: the factor $\Z_0$ is the $\sigma$-algebra of $T$-invariant sets while the factor $\Z_1$ is the \emph{Kronecker factor}. For $s>1$, however, the factors $\Z_s$ get ever more complicated as $s$ increases. 
Importantly, $\Z_s$ is a factor of $\Z_{s+1}$ for each $s\in\NN$, hence it is natural to inquire about the smallest value of $s$ for which the factor $\Z_s$ is characteristic for $\vec{P}$ at $i$. %This leads to the following definition.

%called henceforth \emph{Host-Kra factors}, with the property that expressions of the form (\ref{multiple ergodic averages}) remain unchanged if we project any of the functions $f_i$ onto one of the factors $\Z_k$, where $k$ depends on $\vec{P}$ and $i$.  If this is the case then we say that the factor $\Z_k$ is \emph{characteristic} for $\vec{P}$ at $i$. The utility of the factors $\Z_k$ comes from the fact that they are inverse limits of \emph{nilsystems}, and so understanding (\ref{multiple ergodic averages}) comes down to proving certain equidistribution results on nilmanifolds. 

\begin{definition}[Host-Kra complexity]\label{Host-Kra complexity}
Let $t\in\NN_+$ and $\vec{P}\in\RR[x,y]^{t+1}$ be an integral polynomial progression. Fix $0\leqslant i\leqslant t$. The progression $\vec{P}$ has \emph{Host-Kra complexity} $s$ at $i$, denoted $\HK_i(\vec{P})$, if $s$ is the smallest natural number such that the factor $\Z_s$ is characteristic for the weak convergence of $\vec{P}$ at $i$ for all invertible totally ergodic dynamical systems $(X,\X,\mu, T)$. We say $\vec{P}$ has \emph{Host-Kra complexity} $s$ if $\max_i \HK_i(\vec{P}) = s$.
%the rational Kronecker factor $\K_{rat}$ is characteristic for the $L^2$-convergence of $\vec{P}$ at $i$. Otherwise it has complexity $k>0$ if $k$ is the smallest natural number such that the factor $\Z_k$ is characteristic for the $L^2$-convergence of $\vec{P}$ at $i$

%$L^2$-limit of (\ref{multiple ergodic averages}) is 0 when $\EE(f_i|\K_{rat})=0$, where $\EE(f_i|\K_{rat})$ is the projection of $f_i$ onto the rational Kronecker factor $\K_{rat}$. Otherwise we say that it has complexity $k>0$ if $k$ is the smallest natural number such that th
\end{definition}
Investigating complexity has been of particular interest for a class of dynamical systems called Weyl systems, leading to another notion of complexity, a variant of which is given below.

\begin{definition}[Weyl complexity]\label{Weyl complexity}
Let $t\in\NN_+$ and $\vec{P}\in\RR[x,y]^{t+1}$ be an integral polynomial progression. Fix $0\leqslant i\leqslant t$. The progression $\vec{P}$ has \emph{Weyl complexity} $s$ at $i$, denoted $\W_i(\vec{P})$, if $s$ is the smallest natural number such that the factor $\Z_s$ is characteristic for the weak convergence of $\vec{P}$ at $i$ for all Weyl systems $(X,\X,\mu, T)$.
We say $\vec{P}$ has \emph{Weyl complexity} $s$ if $\max_i \W_i(\vec{P})=s$.
\end{definition}
In previous works \cite{bergelson_leibman_lesigne_2007, leibman_2009, frantzikinakis_2008, frantzikinakis_2016}, the aforementioned notions of complexity have been defined for a polynomial family $\P=\{P_1, \ldots, P_t\}$ rather than for a progression $\vec{P}$. However, we want to extend the definitions of complexity to ``index 0", i.e. the $x$ term in $\vec{P}$, which is why we prefer to define it for $\vec{P}$ rather than $\P$. Similarly, complexity has previously been defined for $L^2$ convergence rather than weak convergence. However, the existence of $L^2$ limit (Theorem \ref{Host-Kra factors are characteristic}) and basic functional analysis imply that weak and $L^2$ limits are identical.

Host-Kra factors are deeply related to a family of seminorms called \emph{Gowers-Host-Kra seminorms}. For $s\in\NN_+$ and $f\in L^\infty(\mu)$, the Gowers-Host-Kra seminorm of $f$ of degree $s$ is denoted by $\nnorm{f}_s$ and satisfies the property
\begin{align}\label{Gowers-Host-Kra seminorms vs. Host-Kra factors}
    \nnorm{f}_{s+1} = 0 \iff \EE(f|\Z_s) = 0
\end{align}
as well as the monotonicity property
\begin{align}\label{monotonicity property}
    \nnorm{f}_1 \leqslant \nnorm{f}_2 \leqslant \nnorm{f}_3 \leqslant \ldots
\end{align}
%The properties (\ref{Gowers-Host-Kra seminorms vs. Host-Kra factors}) and (\ref{monotonicity property}) together imply that if $\EE(f|\Z_{k+1})=0$, then $\EE(f|\ZZ_k) = 0$; this is another reason why seeking the smallest value of $k$ for which $\Z_k$ is characteristic for a progression $\vec{P}$ at index $i$ is reasonable.

Gowers-Host-Kra seminorms have natural finitary analogues. For the transformation $Tx = x+1$ on $X=\ZZ/N\ZZ$ with $N$ prime and the uniform probability measure $\mu$, the weak limit (\ref{weak limit}) becomes
\begin{align}\label{asymptotic count}
    \EE_{x,y\in\ZZ/N\ZZ}f_0(x) f_1(x+P_1(y)) \cdots f_{t}(x+P_{t}(y)).
\end{align}
The Gowers-Host-Kra seminorm of any $f:\ZZ/N\ZZ\to\CC$ is a norm (for $s>1$) called the \emph{Gowers norm} and denoted by $U^s$, and it takes the form
\begin{align}\label{Gowers norm}
    \norm{f}_{U^s}&=\left(\EE_{x, h_1, ..., h_s\in\ZZ/N\ZZ}\prod_{w\in\{0,1\}^s} \C^{|w|}f(x+w_1 h_1 + \ldots + w_s h_s)\right)^\frac{1}{2^s},
\end{align}
where $\C: z\mapsto \overline{z}$ is the conjugation operator and $|w|=w_1+ \cdots+w_s$. As a result, $\norm{f}_{U^s} = 0$ for some $s>1$ if and only if $\norm{f}_{U^2} = 0$ if and only if $f = 0$, and so inquiring about the smallest characteristic factor of this system in the sense of Definition \ref{characteristic factors} makes little sense. We can however ask which Gowers norm ``controls" $\vec{P}$ in a more finitary way, and this leads to another notion of complexity.

\begin{definition}[True complexity]\label{true complexity}
Let $t\in\NN_+$ and $\vec{P}\in\RR[x,y]^{t+1}$ be an integral polynomial progression. Fix $0\leqslant i\leqslant t$. The progression $\vec{P}$ has \emph{true complexity} $s$ at $i$, denoted $\T_i(\vec{P})$, if $s$ is the smallest natural number with the following property: for every $\epsilon>0$, there exist $\delta>0$ and $N_0\in\NN$ such that for all primes $N>N_0$ and all functions $f_0, \ldots, f_{t}:\ZZ/N\ZZ\to\CC$ satisfying $\max_i\norm{f_i}_\infty\leqslant 1$, we have
\begin{align*}
    \left|\EE_{x,y\in\ZZ/N\ZZ}f_0(x)f_1(x+P_1(y)) \cdots f_{t}(x+P_{t}(y))\right| < \epsilon
\end{align*}
whenever $\norm{f_i}_{U^{s+1}}<\delta$. We say $\vec{P}$ has \emph{true complexity} $s$ if $\max_i \T_i(\vec{P})=s$.
\end{definition}

We have so far defined three notions of complexity, that of Host-Kra, Weyl and true complexity. They are all defined in terms of ergodic theory or higher order Fourier analysis and have to do with ``controlling" expressions like (\ref{multiple ergodic averages}) and (\ref{asymptotic count}) by characteristic factors, Gowers-Host-Kra seminorms and Gowers norms. We shall now introduce one more notion, defined purely in terms of algebraic properties of polynomial progressions, and conjecture that all four concepts of complexity are in fact the same.

\begin{definition}[Algebraic relations and algebraic complexity]\label{algebraic complexity}
Let $t\in\NN_+$ and $\vec{P}\in\RR[x,y]^{t+1}$ be an integral polynomial progression. An \emph{algebraic relation} of degree $(j_0, \ldots, j_{t})$ satisfied by $\vec{P}$ is a tuple $(Q_0, \ldots, Q_{t})\in\RR[u]^{t+1}$ such that 
\begin{align}\label{algebraic relation}
    Q_0(x)+ Q_1(x+P_1(y)) + \ldots + Q_{t}(P_{t}(y)) = 0,
\end{align}
where $\deg Q_i = j_i$ for each $0\leqslant i\leqslant t$. The progression $\vec{P}$ has \emph{algebraic complexity} $s$ at $i$ for some $0\leqslant i\leqslant t$, denoted $\A_i(\vec{P})$, if $s$ is the smallest natural number such that for any algebraic relation $(Q_0, \ldots, Q_{t})$ satisfied by $\vec{P}$, the degree of $Q_i$ is at most $s$. It has \emph{algebraic complexity} $s$ if $\max_i \A_i(\vec{P}) = s$.
\end{definition}

\begin{conjecture}[Four notions of complexity are the same]\label{main conjecture}
Let $t\in\NN_+$ and $\vec{P}\in\RR[x,y]^{t+1}$ be an integral polynomial progression. Fix $0\leqslant i\leqslant t$. Then
\begin{align*}
    \HK_i(\vec{P}) = \W_i(\vec{P}) = \T_i(\vec{P}) = \A_i(\vec{P})\leqslant t-1.
\end{align*}
\end{conjecture}
The heuristic for Conjecture \ref{main conjecture} is as follows: evaluating expressions like (\ref{weak limit}) and (\ref{asymptotic count}) comes down to understanding the distribution of certain polynomial sequences on nilmanifolds, and the only obstructions to equidistribution come from algebraic relations of the form (\ref{algebraic relation}).

Several substatements of Conjecture \ref{main conjecture}, such as the equivalence of Weyl and Host-Kra complexity and the upper bound on complexities, have previously been conjectured in \cite{bergelson_leibman_lesigne_2007, leibman_2009, frantzikinakis_2008, frantzikinakis_2016}. Similarly, the equivalence of true and algebraic complexity has been studied and proved for linear forms \cite{gowers_wolf_2010, gowers_wolf_2011a, gowers_wolf_2011b, gowers_wolf_2011c} as well as certain subclasses of polynomial progressions \cite{peluse_2019, kuca_2020a, kuca_2020b}. However, we have not seen the full statement of Conjecture \ref{main conjecture} anywhere in the literature. In particular, we have not found a conjecture relating Host-Kra and Weyl complexity to algebraic complexity, even though the aforementioned papers researching the topic mention that algebraic relations form a source of obstructions preventing a progression from having a characteristic small-degree Host-Kra factor.

Before we state our main result, we have to distinguish between two large families of progressions.
\begin{definition}[Homogeneous and inhomogeneous relations and progressions]\label{homogeneity}
Let $t\in\NN_+$ and ${\vec{P}\in\RR[x,y]^{t+1}}$ be an integral polynomial progression. An \emph{algebraic relation} $(Q_0, \ldots, Q_{t})\in\RR[u]^{t+1}$ is \emph{homogeneous} of degree $d$ if it is of the form
\begin{align*}
    (Q_0(u), \ldots, Q_{t}(u)) = (a_0 u^d, \ldots, a_{t} u^{d})
\end{align*}
for some $a_0, \ldots, a_{t}\in\RR$ (some but not all of which may be zero), and \emph{inhomogeneous} otherwise. The progression $\vec{P}$ is \emph{homogeneous} if all the algebraic relations that it satisfies are linear combinations of its homogeneous algebraic relations, and it is called \emph{inhomogeneous} otherwise.
\end{definition}
An example of a homogeneous progression is $(x, \; x+y,\; x+2y,\; x+y^3)$, which only satisfies a homogeneous relation
\begin{align}\label{relation in 3APs}
    x - 2(x+y) + (x+2y) = 0.
\end{align}
Other examples include arithmetic progressions, progressions with $P_1, \ldots, P_t$ being linearly independent such as $(x,\; x+y,\; x+y^2)$, or progressions whose terms satisfy no quadratic relations, such as $(x,\; x+y^2,\; x+2y^2,\; x+y^3,\; x+2y^3)$.
By contrast, the progression $(x, \; x+y,\; x+2y,\; x+y^2)$ is inhomogeneous because it satisfies both (\ref{relation in 3APs}) and the inhomogeneous relation
\begin{align}\label{algebraic relation for x, x+y, x+2y, x+y^2}
    x^2 + 2x - 2(x+y)^2 + (x+2y)^2 - 2(x+y^2) = 0
\end{align}
that cannot be broken down into a sum of homogeneous relations. These two progressions will accompany us as running examples throughout the paper.

Our main result is the following.
\begin{theorem}[Conjecture \ref{main conjecture} holds for homogeneous progressions]\label{Main theorem}
Let $t\in\NN_+ $. If $\vec{P}\in\RR[x,y]^{t+1}$ is a homogeneous polynomial progression, then it satisfies Conjecture \ref{main conjecture}.
\end{theorem}

Having defined Host-Kra complexity using totally ergodic systems, we would like to extend our results to ergodic systems. We have however encountered an algebraic obstacle in doing so that prevents us from performing this generalisation for all homogeneous progressions. We introduce a subfamily of homogeneous polynomial progressions for which this extension is possible, borrowing the terminology of Frantzikinakis from \cite{frantzikinakis_2008}.
\begin{definition}[Eligible progressions]\label{eligible progressions}
A homogeneous polynomial progression $\vec{P}\in\RR[x,y]^{t+1}$ is \emph{eligible} if for every $r\in \NN_+$ and every $0\leq j\leq r-1$, the family
\begin{align*}
\vec{\tilde{P}}(x,y) = (x,\; x+\tilde{P}_{1,j}(y),\; ...,\; x+\tilde{P}_{t,j}(y)),
\end{align*}
where $\tilde{P}_{i,j}(y) = \frac{P_i(r(y-1)+j) - P_i(j)}{r}$, is homogeneous, and $\A_i(\vec{P}) = \A_i(\vec{\tilde{P}})$ for every $0\leq i \leq t$. 
\end{definition}
The condition in Definition \ref{eligible progressions} may seem artificial at first glance, but this turns out to be the condition that we need to pass from totally ergodic to ergodic systems. While we believe that all homogeneous progressions satisfy this condition, we have not been able to prove this.

We now state the corollary that gives us the smallest characteristic Host-Kra factor for eligible progressions on ergodic systems. The main difference is that if a system has complexity 0, then the $\Z_0$ factor has to be replaced by the rational Kronecker factor $\K_{rat}$.

\begin{corollary}\label{factors for ergodic systems}
Let $t\in\NN_+$ and $\vec{P}\in\RR[x,y]^{t+1}$ be an eligible homogeneous polynomial progression, and suppose that $\A_i(\vec{P}) = s$ for some $0\leqslant i\leqslant t$ and $s\in\NN$. For all invertible ergodic dynamical systems $(X, \X, \mu, T)$, the factor $\Z_s$ is characteristic for the weak and $L^2$ convergence of $\vec{P}$ at $i$ if $s>0$, and $\K_{rat}$ is characteristic for the weak and $L^2$ convergence of $\vec{P}$ at $i$ if $s=0$.
\end{corollary}

Since all polynomial progressions of algebraic complexity at most 1 are homogeneous and eligible, the following corollary follows.

\begin{corollary}\label{factors for ergodic systems for progressions of complexity 1}
Let $t\in\NN_+$ and $\vec{P}\in\RR[x,y]^{t+1}$ be polynomial progression of algebraic complexity at most 1. For all invertible ergodic dynamical systems $(X, \X, \mu, T)$, the factor $\Z_1$ is characteristic for the weak and $L^2$ convergence of $\vec{P}$ at $i$ if $\A_i(\vec{P}) = 1$, and $\K_{rat}$ is characteristic for the weak and $L^2$ convergence of $\vec{P}$ at $i$ if $\A_i(\vec{P}) = 0$.
\end{corollary}

Theorem \ref{Main theorem} as well as Corollaries \ref{factors for ergodic systems} and \ref{factors for ergodic systems for progressions of complexity 1} can be viewed as extensions of \cite{host_kra_2005a, host_kra_2005b, frantzikinakis_kra_2005, frantzikinakis_kra_2006, frantzikinakis_2008, bergelson_leibman_lesigne_2007, leibman_2009}, which find characteristic factors for linear configurations, linearly independent polynomials, progressions of length 4, examine Weyl complexity for arbitrary integral polynomial progression, and give an upper bound for Host-Kra complexity for general integral progressions. Theorem \ref{Main theorem} also partly extends \cite{gowers_wolf_2010, gowers_wolf_2011a, gowers_wolf_2011b, gowers_wolf_2011c, green_tao_2010a, altman_2021, manners_2018, manners_2021, peluse_2019, kuca_2020a, kuca_2020b}, which among other things determine true complexity for certain families of linear forms and integral polynomial progressions.

In particular, we extend our earlier work from \cite{kuca_2020b}. In that paper, we prove equidistribution results on nilmanifolds for progressions of the form $(x, \; x+Q(y),\; x+R(y),\; x+Q(y)+R(y))$ with $\deg Q < \deg R$, or $(x,\; x+Q(y),\; x+2Q(y),\; x+R(y),\; x+2R(y))$ with $\deg Q < (\deg R)/2$, both of which are homogeneous. These equidistribution results follow from inducting on the filtration of a certain nilmanifold associated with the progression; the induction scheme involved is quite sensitive to the progression in question. Here, we achieve a much more general equidistribution result (part (i) of Theorem \ref{dichotomy}) by obtaining a solid understanding of the algebra behind homogeneous progressions and introducing a more flexible induction scheme.

From the fact that all progressions of algebraic complexity 1 are homogeneous and eligible, we deduce the following counting result.
\begin{corollary}\label{counting result}
Let $t\in\NN_+$ and $\vec{P}\in\RR[x,y]^{t+1}$ be an integral polynomial progression of algebraic complexity at most 1. Suppose that $Q_1, \ldots, Q_d\in\RR[y]$ are integral polynomials such that $P_i(y) = \sum_{j=1}^d a_{ij} Q_j(y)$ for $a_{ij}\in\ZZ$ for each $0\leqslant i\leqslant t$ and $1\leqslant j\leqslant d$. Let $L_i(y_1, \ldots y_d) = \sum_{j=1}^d a_{ij} y_j$. Then the following is true.
\begin{enumerate}
\item For any $f_0, \ldots, f_t:\ZZ/N\ZZ\to\CC$ with $\max_i \norm{f_i}_\infty\leqslant 1$, we have

\begin{align*}
\EE_{x,y\in\ZZ/N\ZZ} \prod_{i=0}^t f_i(x+P_i(y)) = \EE_{x,y_1, ..., y_d\in\ZZ/N\ZZ} \prod_{i=0}^t f_i(x+L_i(y_1, \ldots, y_d)) + o(1),
\end{align*}
where the error term $o(1)$ is taken as $N\to\infty$ over primes and does not depend on the choice of $f_0, \ldots, f_t$.

\item For any invertible totally ergodic dynamical system $(X, \X, \mu, T)$ and $f_0, \ldots, f_t\in L^\infty(\mu)$, we have

\begin{align*}
\lim_{N\to\infty}\EE_{n\in [N]} \int_X \prod_{i=0}^t T^{P_i(n)} f_i d\mu= \lim_{N\to\infty}\EE_{n_1, ..., n_d\in [N]} \int_X \prod_{i=0}^t T^{L_i(n_1, \ldots, n_d)} f_i d\mu.
\end{align*}
\end{enumerate}
\end{corollary}

We shall illustrate Corollary \ref{counting result} for the specific example of $$\vec{P}(x,y) = (x,\; x+y^2,\; x+2 y^2,\; x + y^3,\; x + 2 y^3).$$ Taking $Q_1(y) = y^2$ and $Q_2(y) = y^3$ as in the statement of Corollary \ref{counting result}, we let $\vec{L}(x,y_1, y_2) = (x,\; x+y_1,\; x+2 y_1,\; x+y_2,\; x + 2 y_2)$. For any $A\subset\ZZ/N\ZZ$, we then have
\begin{align*}
&|\{(x,y)\in(\ZZ/N\ZZ)^2: (x,\; x+y^2,\; x+2 y^2,\; x + y^3,\; x + 2 y^3)\in A^5\}|\\
= &|\{(x, y_1, y_2)\in (\ZZ/N\ZZ)^3: (x,\; x+y_1,\; x+2 y_1,\; x+y_2,\; x + 2 y_2)\in A^5 \}|/N + o(N^2)
%&= |\{(x_1,\; x_2,\; x_3,\; x_4,\; x_5)\in A^5: x_1 + x_3 = 2x_2,\; x_1 + x_5 + 2x_4\}| + o(N^2)
%|\{(x,\; x+y^2,\; x+2 y^2,\; x + y^3, x+ 2 y^3)\in A^5\}| &= |\{(x,\; x+y_1,\; x+2 y_1,\; x+y_2,\; x + 2 y_2)\in A^5 \}|/N + o(N^2)\\
%&= |\{(x_1,\; x_2,\; x_3,\; x_4,\; x_5)\in A^5: x_1 + x_3 = 2x_2,\; x_1 + x_5 + 2x_4\}| + o(N^2)
\end{align*}
upon setting $f_0 = \ldots = f_t = 1_A$. If $(X,\X, \mu, T)$ is a totally ergodic system and $A\in\X$, then we similarly obtain that 
\begin{align*}
&\lim_{N\to\infty} \EE_{n\in[N]}\mu(A\cap T^{n^2}A\cap T^{2n^2}A\cap T^{n^3}A\cap T^{2n^3}A)\\
&=\lim_{N\to\infty}  \EE_{n,m\in [N]}\mu(A\cap T^{n}A\cap T^{2n}A\cap T^{m}A\cap T^{2m}A).
\end{align*}

%The reason why we have been able to prove Conjecture \ref{main conjecture} for homogeneous progressions is that these progressions satisfy certain linear algebraic properties. These properties allow us to explicitly describe closures of certain polynomial orbits on nilmanifolds, from which we then deduce Theorem \ref{Main theorem}. Inhomogeneous progressions unfortunately fail to satisfy the same linear algebraic properties.

For progressions of algebraic complexity 1, we also prove the following result, which generalises Theorem C of \cite{frantzikinakis_2008}, Theorem 1.12 of \cite{green_tao_2010a}, and results from \cite{bergelson_host_kra_2005}. In additive combinatorics, problems of this type are known as finding popular common differences; in ergodic theory, one speaks of establishing lower bounds for multiple recurrence.
\begin{theorem}\label{popular common differences}
Let $t\in\NN_+$ and $\vec{P}\in\RR[x,y]^{t+1}$ be an integral polynomial progression of algebraic complexity at most 1, with the following property: there exist linearly independent integral polynomials $Q_1, \ldots, Q_k$ such that 
\begin{align}\label{same group of polynomials}
    \{a_1 Q_1 + \ldots + a_k Q_k:\; a_1, \ldots, a_k\in\ZZ\} = \{b_1 P_1 + \ldots + b_t P_t:\; b_1, \ldots, b_t\in\ZZ\}.
\end{align}
Then the following is true.
\begin{enumerate}
\item Let $(X, \X, \mu, T)$ be an ergodic invertible measure preserving system and $A\in\X$. Suppose that $\mu(A)>0$. Then for every $\epsilon>0$, the set
\begin{align*}
\{ n\in\NN: \mu(A\cap T^{P_1(n)} A \cap \cdots \cap  T^{P_t(n)}A) > \mu(A)^{t+1} - \epsilon\}
\end{align*}
is syndetic, i.e. it has bounded gaps.
\item Suppose that $A\subset \NN$ has upper density $\alpha > 0$. Then for every $\epsilon > 0$, the set
\begin{align*}
\{ n\in\NN: \mu(A\cap (A + P_1(n)) \cap \cdots \cap  (A + P_t(n))) > \alpha^{t+1} - \epsilon\}
\end{align*}
is syndetic.
\item For any $\alpha, \epsilon > 0$ and prime $N$, and any subset $A\subset\ZZ/N\ZZ$ of size $|A|\geqslant \alpha N$, we have
\begin{align*}
|\{ n\in\ZZ/N\ZZ: |A\cap (A + P_1(n)) \cap \cdots \cap  (A + P_t(n))| > (\alpha^{t+1} - \epsilon)N\}|\gg_{\alpha, \epsilon} N.
\end{align*}
\end{enumerate}
\end{theorem}

The definition of homogeneity (Definition \ref{homogeneity}) is equivalent to a certain linear algebraic property that will be described in details in Section \ref{section on homogeneity}; this property makes it possible to explicitly describe closures of orbits of nilsequences evaluated at terms of homogeneous polynomial progressions, from which we deduce Theorem \ref{Main theorem}. Homogeneous polynomial progressions are moreover the largest family of integral polynomial progressions for which such an explicit description is possible, and even the simplest examples of inhomogeneous progressions lead to complications absent in the homogeneous case. The following result makes this precise. As with all other results in this section, all the concepts in Theorem \ref{dichotomy} are explained in subsequent sections.
\begin{theorem}[Dichotomy between homogeneous and inhomogeneous progressions]\label{dichotomy}
Let $t\in\NN_+$ and $\vec{P}\in\RR[x,y]^{t+1}$ be an integral polynomial progression. Suppose that $G$ is a connected, simply-connected, nilpotent Lie group with a rational filtration $G_\bullet$ and $\Gamma$ is a cocompact lattice. There exists a subnilmanifold $G^P/\Gamma^P$ of $G^{t+1}/\Gamma^{t+1}$ with the following property.
\begin{enumerate}
    \item If $\vec{P}$ is homogeneous, then for every irrational polynomial sequence $g:\ZZ\to G$ adapted to $G_\bullet$, the sequence 
    \begin{align*}
        g^P(x,y) = (g(x),\; g(x+P_1(y)),\; \ldots,\; g(x+P_t(y)))
    \end{align*}
    is equidistributed on $G^P/\Gamma^P$.
    \item If $\vec{P}$ is inhomogeneous, then for every irrational polynomial sequence $g\in\poly(\ZZ, G_\bullet)$, the closure of $g^P$ is a union of finitely many translates of a subnilmanifold of $G^P/\Gamma^P$. For every $\vec{P}$, we can moreover find a filtered nilmanifold $G/\Gamma$ and an irrational polynomial sequence $g:\ZZ\to G$ such that $g^P$ is equidistributed on a proper subnilmanifold of $G^P/\Gamma^P$. 
    %If $\vec{P}$ is inhomogeneous, then for every irrational polynomial sequence $g:\ZZ\to G$ adapted to $G_\bullet$, the sequence $g^P$ is equidistributed on a union of finitely many translates of a subnilmanifold of $G^P/\Gamma^P$. For every $\vec{P}$, we can moreover find a filtered nilmanifold $G/\Gamma$ and an irrational polynomial sequence $g:\ZZ\to G$ such that $g^P$ is equidistributed on a proper subnilmanifold of $G^P/\Gamma^P$. 
    
\end{enumerate}
\end{theorem}

While we have not been able to prove full Conjecture \ref{main conjecture} for inhomogeneous progressions, we are able to say a bit more about the relationship between various notions of complexity in the general case.

\begin{theorem}\label{relations between notions of complexity}
Let $t\in\NN_+$ and $\vec{P}\in\RR[x,y]^{t+1}$ be an integral polynomial progression. Fix $0\leqslant i\leqslant t$. Then
\begin{align*}
   \W_i(\vec{P}) = \A_i(\vec{P})\leqslant \min(\T_i(\vec{P}),  \HK_i(\vec{P})).
\end{align*}
\end{theorem}
Of the various statements made in Theorem \ref{relations between notions of complexity}, the fact that Host-Kra complexity bounds Weyl complexity is a simple consequence of definitions and shall be explained in Section \ref{section on Weyl complexity}. Similarly, the fact that algebraic complexity is bounded from above by true complexity has been shown in Theorem 1.13 of \cite{kuca_2020b}. It is the equivalence of Weyl and algebraic complexities that is a new statement here.

\subsection*{Outline of the paper}
We start the paper by introducing basic ergodic theoretic definitions and results concerning nilsystems in Section \ref{section on ergodic theory}, and we explain why analyzing expressions like (\ref{weak limit}) comes down to answering equidistribution questions on nilmanifolds. We then show in Section \ref{section on reducing to connected groups} that in studying equidistribution on nilmanifolds, we can restrict ourselves to nilmanifolds that are quotients of connected groups at the expense of replacing a linear sequence by a polynomial one. 

Section \ref{section on homogeneity} develops a notation and basic theory for certain vector spaces associated with polynomial progressions, and it explains key differences between homogeneous and inhomogeneous progressions. In particular, it contains the proof of the upper bound on algebraic complexity for homogeneous progressions from Theorem \ref{Main theorem}. Definitions introduced in this section allow us to state  the infinitary version of an equidistribution result for homogeneous polynomial progressions on nilmanifolds (Theorem \ref{infinitary equidistribution on nilmanifolds}) in Section \ref{section on Host-Kra complexity}, from which we deduce that for homogeneous progressions, Host-Kra complexity is bounded from above by algebraic complexity (Corollary \ref{Host-Kra complexity equals algebraic complexity}). We further use Theorem \ref{infinitary equidistribution on nilmanifolds} to deduce Corollaries \ref{factors for ergodic systems}, \ref{factors for ergodic systems for progressions of complexity 1} and \ref{counting result}(ii). 

In Section \ref{section on finitary nilmanifold theory}, we introduce finitary analogues of tools from Section \ref{section on ergodic theory}. These are needed in Section \ref{section on true complexity}, in which we show that proving the equivalence of true and algebraic complexity for homogeneous progression comes down to proving Theorem \ref{finitary equidistribution on nilmanifolds}, a finitary version of Theorem \ref{infinitary equidistribution on nilmanifolds}. We also explain in Section \ref{section on true complexity} how to prove Corollary \ref{counting result}(i). Theorem \ref{finitary equidistribution on nilmanifolds}, the main technical part of this paper, is derived in Section \ref{section on the full homogeneous case}. Unfortunately, Theorem \ref{finitary equidistribution on nilmanifolds} fails for inhomogeneous progressions, as explained in Section \ref{section on failure in the inhomogeneous case}. In Section \ref{section on inhomogeneous case}, we propose a method to handle inhomogeneous progressions. While we succeed in proving an analogue of Theorem \ref{infinitary equidistribution on nilmanifolds} for the inhomogeneous progression $(x, \; x+y,\; x+2y,\; x+y^2)$ in Proposition \ref{closure for inhomogeneous}, we have been unable to extend this construction to all inhomogeneous progressions. Subsequently, we show in \ref{section on Weyl complexity} that Weyl and algebraic complexity are always equal, which is the main statement of Theorem \ref{relations between notions of complexity}. We conclude the paper by proving Theorem \ref{popular common differences} in Section \ref{section on multiple recurrence}.

\subsection*{Acknowledgments}
We are indebted to Donald Robertson for his comments on earlier versions of the paper and fruitful conversations on the project while it was carried out, and to Faustin Adiceam and Julia Wolf for pointing out an error in the original statement and proof of Corollary 1.13. We would also like to thank Sean Prendiville for introducing us to the topic of complexity, Tuomas Sahlsten for hosting a reading group on the dynamical proof of Szemer\'{e}di theorem, and Jonathan Chapman for useful discussions on algebraic relations between terms of polynomial progressions. Finally, we would like to thank the anonymous referee for suggestions on how to improve the paper. 

%%%
%%%
%%% ERGODIC THEORY
%%%
%%%

\section{Infinitary nilmanifold theory}\label{section on ergodic theory}
\subsection{Basic definitions from ergodic theory}
Let $(X,\X,\mu,T)$ be an invertible measure-preserving dynamical system (henceforth, we shall simply call it a \emph{system}). The background in ergodic theory that we need can be found in \cite{host_kra_2005b, host_kra_2018}, among others; here, we only reiterate the most important definitions. 
\begin{definition}\label{factors}
A \emph{factor} of a system $(X, \X, \mu, T)$ can be defined in three equivalent ways:
\begin{enumerate}
\item it is a $T$-invariant sub-$\sigma$-algebra of $\X$;
\item it is a system $(Y,\Y,\nu,S)$ together with a \emph{factor map} $\pi: X'\to Y'$, i.e. a measurable map defined for a measurable $T$-invariant set $X'$ of full measure, satisfying $S\circ\pi = \pi\circ T$ on $X'$ and $\mu\circ\pi^{-1} = \nu$;
\item it is a $T$-invariant subalgebra of $L^\infty(\mu)$.
\end{enumerate}
\end{definition}
For $r\in\NN$, we let $\K_r$ be the factor spanned by all $T^r$-invariant functions in $L^\infty(\mu)$. In particular, $\K_1 = \I$ is the factor spanned by $T$-invariant functions, and the \emph{rational Kronecker factor} $\K_{rat} = \bigvee\limits_{r\in\NN}\K_r$ is the factor spanned by all the functions in in $L^\infty(\mu)$ that are $T^r$-invariant for some $r\in\NN$. A system is \emph{ergodic} if $\K_1=\I$ is the trivial factor spanned by constant functions, and it is \emph{totally ergodic} if $\K_{rat}$ is the trivial factor.

Of particular interest to us is a sequence of factors $(\Z_s)_{s\in\NN}$ defined in \cite{host_kra_2005b}, which we refer to as \emph{Host-Kra factors}. In accordance with Definition \ref{factors}, we shall sometimes think of $\Z_s$ as a sub-$\sigma$-algebra of $\X$, and at other times we will consider a factor map $\pi_s: X\to Z_s$ and a factor $(Z_s, \Z_s, \lambda, S)$ of $(X, \X, \mu, T)$. If we concurrently talk about Host-Kra factors of two distinct spaces $X$ and $Y$, we may write $Z_s(X)$ and $Z_s(Y)$ to mean Host-Kra factors of $X$ and $Y$ respectively. We do not explicitly use the definition of Host-Kra factors anywhere in the paper, and so we leave the interested reader to look it up in \cite{host_kra_2005b, host_kra_2018}. Instead, we rely on two properties of this family of factors that concern their utility and structure respectively. First, these factors are characteristic for the convergence of polynomial progressions, as proved in Theorem \ref{Host-Kra factors are characteristic}. Rephrasing Theorem \ref{Host-Kra factors are characteristic} in terms of Definition \ref{Host-Kra complexity}, we can say that each integral polynomial progression has a finite Host-Kra complexity. Second, each factor $\Z_s$ is an inverse limit\footnote{The system $(X,\X,\mu,T)$ is an \emph{inverse limit} of a sequence of factors $(X,\X_i,\mu,T)$ if $\X_i$ form an increasing sequence of factors of $\X$ such that $\X = \bigvee\limits_{i\in\NN}\X_i$ up to sets of measure zero.} of $s$-step nilsystems, which are objects of primary importance to us.

%%%
%%%
%%% Nilsystems
%%%
%%%

\subsection{Nilsystems}\label{section on nilmanifolds}
Let $G$ be a Lie group with connected component $G^0$ and identity $1$. A \emph{filtration} on $G$ of degree $s$ is a chain of subgroups
\begin{align*}
    G = G_0 = G_1 \geqslant G_2 \geqslant \ldots \geqslant G_s \geqslant G_{s+1} = G_{s+2} = \ldots = 1
\end{align*}
satisfying $[G_i, G_j]\leqslant G_{i+j}$ for each $i,j\in\NN$. We denote it as $G_\bullet = (G_i)_{i=0}^\infty$. A natural example of filtration is the lower central series, given by $G_{k+1} = [G, G_k]$ for each $k>1$, where the \emph{commutator} of two elements $a,b\in G$ is defined as $[a,b]=a^{-1}b^{-1}ab$, and $[A,B]$ is the subgroup of $G$ generated by all the commutators $[a,b]$ with $a\in A, b\in B$. The group $G$ is \emph{$s$-step nilpotent} if $G_{s+1} = 1$, where $G_{s+1}$ is the $s$-th element of the lower central series of $G$. The only 0-step nilpotent group is the trivial group, and 1-step nilpotent groups are precisely abelian groups. 

For the rest of the paper, we let $G$ be a nilpotent Lie group and $\Gamma\leqslant G$ be a cocompact lattice. We call the quotient $X=G/\Gamma$ a \emph{nilmanifold}. The group $G$ acts on $X$ by left translation, and for each $a\in G$, we call the map $T_a(g\Gamma) = (ag)\Gamma$ a \emph{nilrotation}. Setting $\G/\Gamma$ to be the Borel $\sigma$-algebra of $X$ and $\nu$ to be the Haar measure with respect to left translation, we call the system $(G/\Gamma, \G/\Gamma, \nu, T_a)$ a \emph{nilsystem}.

A subgroup $H\leqslant G$ is \emph{rational} if $H/(H\cap\Gamma)$ is closed in $G/\Gamma$. A filtration $G_\bullet$ is rational if $G_i$ is a rational subgroup for each $i\in\NN$. We shall assume throughout the paper that each filtration that we discuss is rational.

In the case when $(G/\Gamma, \G/\Gamma, \nu, T_a)$ is an ergodic nilsystem, which will always be our case anyway, we can make two simplifying assumptions about the group $G$. By passing to the universal cover, we assume that $G$ is simply connected. Replacing the nilsystem with several simpler nilsystems, we further assume that $G$ is spanned by $G^0$ and $a$. These assumptions, justified in Chapter 11 of \cite{host_kra_2018}, hold for the rest of the paper.

We also denote $\Gamma_i = G_i\cap \Gamma$ and $\Gamma^0 = G^0\cap \Gamma$. The rationality of $G_i$ in $G$ means that $\Gamma_i$ is cocompact in $G_i$. 

\begin{proposition}[Conditions for total ergodicity of nilsystems, Corollary 7 and 8 of \cite{host_kra_2018}]\label{totally ergodic nilsystems}
Let $(G/\Gamma, \G/\Gamma, \nu, T_a)$ be an ergodic nilsystem. There exists $r\in\NN_+$ such that ${T_a^j(G^0/\Gamma^0)}$ is totally ergodic with respect to $T_a^r$ for all $0\leqslant j<r$.

Moreover, the following are equivalent:
\begin{enumerate}
\item $T_a$ is totally ergodic;
\item $G/\Gamma$ is connected;
\item $G=G^0\Gamma$. 
\end{enumerate}
\end{proposition}

Nilsystems allow a particularly simple description of factors. If $G_\bullet$ is the lower central series filtration, then
\begin{align}\label{Host-Kra factors of nilsystems}
Z_s = \frac{G}{G_{s+1}\Gamma}
\end{align}
for all $s\in\NN_+$ (see Chapter 11 of \cite{host_kra_2018}). For $s=0$, we have $Z_0 = G/(G^0\Gamma)\cong(\ZZ/r\ZZ)$, where $r$ is the smallest positive integer for which $a^r\in G^0$. It follows from Proposition \ref{totally ergodic nilsystems} that $Z_0$ is trivial if and only if the nilsystem is totally ergodic.

Let $\vec{P}\in\RR[x,y]^{t+1}$ be an integral polynomial progression. By Theorem \ref{Host-Kra factors are characteristic}, there exists $s\in\NN$ such that for every ergodic system $(X, \X, \mu, T)$ and all choices of $f_0, \ldots, f_t\in L^\infty(\mu)$, we have
\begin{align}\label{projecting weak limit on the factor}
\nonumber	&\lim_{N\to\infty}\EE_{n\in [N]}\int_X f_0 \cdot T^{P_1(n)} f_1 \cdots  T^{P_{t}(n)} f_{t} d\mu\\
	 &= \lim_{N\to\infty}\EE_{n\in [N]}\int_{Z_s} \EE (f_0|\Z_s) \cdot S^{P_1(n)}\EE (f_1|\Z_s) \cdots  S^{P_{t}(n)} \EE (f_t|\Z_s) d\lambda
\end{align}

Using the fact that $Z_s$ is an inverse limit of ergodic $s$-step nilsystems, we can approximate the average (\ref{projecting weak limit on the factor}) arbitrarily well by projections onto ergodic nilsystems. Hence we are left with understanding averages of the form
\begin{align}\label{multiple averages on nilmanifolds}
\lim_{N\to\infty}\EE_{n\in[N]} \int_{G/\Gamma} \tilde{f}_0(b\Gamma)\cdot \tilde{f}_1(a^{P_1(n)}b\Gamma) \cdots  \tilde{f}_t(a^{P_t(n)} b\Gamma) d\nu(b\Gamma)
\end{align}
where $\tilde{f}_i$ is the projection of $f_i$ onto an ergodic $s$-step nilsystem $(G/\Gamma, \G/\Gamma, \nu, T_a)$ for all $0\leqslant i\leqslant t$. If $T$ is totally ergodic, then so is the nilrotation $T_a$.

%%%
%%%
%%%  Polynomial sequences
%%%
%%%

\subsection{Polynomial sequences}\label{section on polynomial sequences}
Let $G_\bullet$ be a filtration on $G$ of degree $s$. A \emph{polynomial sequence} $g:\ZZ\to G$ adapted to $G_\bullet$ is a sequence
\begin{align}\label{polynomial sequence}
    g(n) = \prod_{i=0}^s g_i^{{{n}\choose{i}}}.
\end{align}
with the property that $g_i\in G_i$ for each $i$. Such sequences form a group denoted as $\poly(\ZZ, G_\bullet)$ by Proposition 6.2 of \cite{green_tao_2012}. One may ask why we define polynomial sequence as (\ref{polynomial sequence}) rather than in the seemingly more natural form
\begin{align}\label{polynomial sequence, wrong form}
    g(n) = \prod_{i=0}^s g_i^{n^i}.
\end{align}
The reason is that if $g$ is written in the form (\ref{polynomial sequence}), then we have the following nice statement.
\begin{lemma}[Lemma 2.8 of \cite{candela_sisask_2012}]\label{coefficients in a subgroup}
Suppose that $g\in\poly(\ZZ, G_\bullet)$. The sequence $g(n) = \prod_{i=0}^s g_i^{{n}\choose{i}}$ takes values in $H\leqslant G$ if and only if $g_0, \ldots, g_s\in H$.
\end{lemma}
\begin{proof}
The converse direction is straightforward, and we prove the forward direction by induction on $0\leqslant k\leqslant s$. For $k=0$, we observe that $g_0 = g(0)\in H$. Suppose that the statement holds for some $1\leq k <s$, i.e. $g_0, \ldots, g_k\in H$. Then $g(k+1) =\left(\prod_{i=0}^k g_i^{{k+1}\choose{i}}\right) g_{k+1}$. Since $g(k+1), g_0, \ldots, g_k$ are all in $H$, it follows that $g_{k+1}\in H$. 
\end{proof} 
Lemma \ref{coefficients in a subgroup} is not true if $g$ is written in the form (\ref{polynomial sequence, wrong form}); for instance, $g(n) = {{n}\choose{2}} =\frac{1}{2}n^2 - \frac{1}{2}n$ takes values in $\ZZ$ even though $\frac{1}{2}, -\frac{1}{2}\notin\ZZ$.

In a similar manner, we define for any $D\in\NN_+$ the group $\poly(\ZZ^D, G_\bullet)$ of $D$-parameter polynomial sequences $g:\ZZ^D\to G$ adapted to $G_\bullet$, i.e. sequences of the form
\begin{align*}%\label{polynomial sequence}
    g(n_1, \ldots, n_D) = \prod_{i=0}^s\prod_{i_1+...+i_D = i} {g_{i_1, \ldots, i_D}}^{{{n_1}\choose{i_1}}\cdots{{n_D}\choose{i_D}}}
\end{align*}
for $g_{i_1, \ldots, i_D}\in G_{i_1+\ldots i_D}$.

%%%
%%%
%%%  Qualitative equidistribution
%%%
%%%

\subsection{Infinitary equidistribution theory on nilmanifolds}\label{section on infinitary equidistribution}
For the rest of Section \ref{section on ergodic theory}, we assume that $G$ is connected. For $D\in\NN_+$, a polynomial sequence $g\in\poly(\ZZ^D,G_\bullet)$ is \emph{equidistributed} on $G/\Gamma$ if
\begin{align*}
    \EE_{n\in[N]^D}F(g(n)\Gamma) \to \int_{G/\Gamma} F d\nu
\end{align*}
for any continuous function $F:G/\Gamma\to\CC$. The following notion is useful when discussing equidistribution.
\begin{definition}[Horizontal characters]\label{horizontal characters}
A \emph{horizontal character} on $G$ is a continuous group homomorphism $\eta: G\to\RR$ for which $\eta(\Gamma)\leqslant\ZZ$.
\end{definition}
In particular, each horizontal character vanishes on $[G,G]$.
 %A theorem by Leibman says that $g$ is equidistributed on $G/\Gamma$ if and only if its projection onto the horizontal torus $G/[G,G]\Gamma$ is equidistributed. This can be rephrased using the notion of a \emph{horizontal character}, which is a continuous group homomorphism $\eta:G\to\RR$ satisfying $\eta(\Gamma)\subset\ZZ$. 

Equidistribution on nilmanifolds was studied by Leibman, who provided a useful criterion for when a polynomial sequence is equidistributed on a nilmanifold. We only need the version of the statement in the case when $G$ is connected, as we will be able to reduce to this case.
\begin{theorem}[Leibman's equidistribution theorem, \cite{leibman_2005b}]\label{Leibman's equidistribution theorem}
Let $D\in\NN_+$ and $g\in\poly(\ZZ^D, G_\bullet)$. The following are equivalent:
\begin{enumerate}
    \item $g$ is equidistributed in $G/\Gamma$;
    \item the projection of $g$ onto $G/[G,G]$ is equidistributed in $G/[G,G]\Gamma$;
    \item if $\eta: G\to\RR$ is a horizontal character for which $\eta\circ g$ is constant, then $\eta$ is trivial.
\end{enumerate}
\end{theorem}

We shall also need a stronger notion of equidistribution, that of irrational sequences. 
\begin{definition}
Suppose that $G_\bullet$ is a filtration on $G$ and $i\in\NN_+$, and let
\begin{align*}
    G_i^\nabla = \langle G_{i+1}, [G_j, G_{i-j}], 1\leqslant j < i \rangle
\end{align*}
An \emph{$i$-th level character} is a continuous group homomorphism $\eta_i:G_i\to\RR$ that vanishes on $G_i^\nabla$ and satisfies $\eta_i(\Gamma_i)\in\ZZ$. An element $g_i$ of $G_i$ is \emph{irrational} if $\eta_i(g_i)\notin\ZZ$ for any nontrivial $i$-th level character $\eta_i$. A sequence $g(n) = \prod\limits_{i=0}^s g_i^{{n}\choose{i}}$ is \emph{irrational} if $g_i$ is irrational for all $i\in\NN_+$. 
\end{definition}

All irrational sequences are equidistributed, but not vice versa. For instance, let $g(n) = a_1 n + \ldots + a_s n^s$ be a real-valued polynomial. It is a polynomial sequence in $\RR$ adapted to the filtration $G_1 = \ldots = G_s = \RR$, $G_{s+1} = 0$. Thus, $g$ is irrational iff $a_s\notin\QQ$, and $g$ is equidistributed iff there exists $1\leqslant i\leqslant s$ with $a_i\notin\QQ$. It is clear in this case that irrational implies equidistributed, but not vice versa.

We want to emphasise that whether a sequence is irrational or not depends on what filtration we are using, whereas the notion of equidistribution does not depend on the filtration. 

%%%
%%%
%%% Reducing to connected groups.
%%%
%%%

\section{Reducing to the case of connected groups}\label{section on reducing to connected groups}

The expression (\ref{multiple averages on nilmanifolds}) indicates that to understand Host-Kra complexity of a polynomial progression $\vec{P}$, we have to understand the distribution of orbits
\begin{align}\label{polynomial orbit}
(b\Gamma,\; a^{P_1(n)}b\Gamma,\; \ldots,\; a^{P_t(n)}b\Gamma)
\end{align}
inside a connected nilmanifold $G^{t+1}/\Gamma^{t+1}$. 
%Throughout this section, we assume that $G$ is a simply connected, nilpotent Lie group generated by its connected component $G^0$ and $a$. 
The point of this section is to show that we can replace linear orbits $(a^n b\Gamma)_{n\in\NN}$ on $G/\Gamma$ by polynomial orbits $(g_b(n)\Gamma^0)_{n\in\NN}$ on $G^0/\Gamma^0$ for some irrational polynomial sequence $g_b:\ZZ\to G^0$ with respect to a certain naturally defined filtration $G_\bullet^0$ on $G_0$. This way, we want to reduce the question of finding the closure for (\ref{polynomial orbit}) inside $(G/\Gamma)^{t+1}$ to finding the closure for
\begin{align}\label{polynomial orbit 2}
(g_b(m)\Gamma^0,\; g_b(m+ P_1(n))\Gamma^0,\; \ldots,\; g_b(m+P_t(n))\Gamma^0)
\end{align}
inside $(G^0/\Gamma^0)^{t+1}$. The connectedness of $G^0$ allows us to use tools from Section \ref{section on infinitary equidistribution}.

\begin{lemma}\label{introducing a parameter to the integral}
Let $(G/\Gamma, \G/\Gamma, \nu, T_a)$ be a totally ergodic nilsystem and $F:(G/\Gamma)^{t+1}\to\RR$ be essentially bounded. Then
\begin{align*}
&\EE_{n\in[N]} \int_{G/\Gamma} F(b\Gamma, a^{P_1(n)}b\Gamma, \ldots, a^{P_t(n)}b\Gamma)d\nu(b\Gamma)\\
&=\EE_{m, n\in[N]} \int_{G/\Gamma} F(a^m b\Gamma, a^{m+P_1(n)}b\Gamma, \ldots, a^{m+P_t(n)}b\Gamma)d\nu(b\Gamma).
\end{align*}
\end{lemma}
\begin{proof}
Since $T_a$ is measure preserving, we have 
\begin{align*}
\int_{G/\Gamma} F(b\Gamma, a^{P_1(n)}b\Gamma, \ldots, a^{P_t(n)}b\Gamma)d\nu(b\Gamma)
= \int_{G/\Gamma} F(a^m b\Gamma, a^{m+P_1(n)}b\Gamma, \ldots, a^{m+P_t(n)}b\Gamma)d\nu(b\Gamma)
\end{align*}
for any $m,n\in\NN$. Consequently, 
\begin{align*}
&\int_{G/\Gamma} F(b\Gamma, a^{P_1(n)}b\Gamma, \ldots, a^{P_t(n)}b\Gamma)d\nu(b\Gamma)\\
&=\EE_{m\in[N]} \int_{G/\Gamma} F(a^m b\Gamma, a^{m+P_1(n)}b\Gamma, \ldots, a^{m+P_t(n)}b\Gamma)d\nu(b\Gamma),
\end{align*}
from which the lemma follows. 
\end{proof}

The main result of this section is the following. 
\begin{proposition}\label{replacing a linear sequence by a polynomial one}
Let $(G/\Gamma, \G/\Gamma, \nu, T_a)$ be a totally ergodic nilsystem and $b\in G^0$.  Suppose that $G_\bullet$ is the lower central series filtration on $G$ and $G^0_\bullet = G_\bullet\cap G^0$. Then there exists an irrational sequence $g_b\in\poly(\ZZ, G^0_\bullet)$ such that $g_b(n)\Gamma = a^n b\Gamma$.
\end{proposition}
%Thanks to Proposition \ref{replacing a linear sequence by a polynomial one}, we can replace a linear orbit $(a^n b\Gamma)_{n\in\NN}$ on a connected nilmanifold $G/\Gamma$ by a polynomial orbit $(g_b(n)\Gamma^0)_{n\in\NN}$ inside $G^0/\Gamma^0$. The connectedness of $G^0$ allows us to use tools from Section \ref{section on infinitary equidistribution}.

We observe that with this filtration on $G^0$, we have $G^0_k = G_k$ for $k\geqslant 2$. That follows from the fact that the groups $G_k$ are connected for $k\geqslant 2$ (Lemma 5 of \cite{host_kra_2018}), and hence are contained in $G^0$. 

We lose no generality in assuming that $b\in G^0$; Proposition \ref{totally ergodic nilsystems} and the connectedness of $G/\Gamma$ imply that for all $b\in G$ there exists $b'\in G^0$ such that $b\Gamma = b'\Gamma$.

\begin{proof}
The connectedness of $G/\Gamma$ implies that $G=G^0\Gamma$, and so there exist $\alpha\in G^0$ and $\gamma\in\Gamma$ such that $a = \alpha \gamma^{-1}$. Then
\begin{align*}
a^n b\Gamma = (\alpha\gamma^{-1})^n b\Gamma = (\alpha\gamma^{-1})^n b \gamma^{n}\Gamma.
\end{align*}
It follows from normality of $G^0$ and the fact that $\alpha$ and $b$ are elements of $G^0$ that the sequence $g_b(n) = (\alpha\gamma^{-1})^n b \gamma^{n}$ takes values in $G^0$. Since the sequences $h_1(n) = a^n b$ and $h_2(n) = \gamma^{n}$ are adapted to $G_\bullet$, and the set $\poly(\ZZ, G_\bullet)$ is a group, we deduce that $g_b = h_1 h_2$ is adapted to $G^0_\bullet = G_\bullet\cap G^0$. 

We want a more precise description of $g_b$, and for this we shall use some results from Sections 11-13 of \cite{leibman_2009}. Let $g=g_b$ for the identity $b=1$; that is, $g(n) = (\alpha \gamma^{-1})^n \gamma^n$. Leibman showed in Section 11.2 of \cite{leibman_2009} that
\begin{align}\label{abstract expansion of g}
g(n) &= \prod_{1\leqslant k_1 \leqslant s} (A^{k_1-1}\alpha)^{q_{k_1}(n)} \prod_{1\leqslant k_2 < k_1 < s}[A^{k_1-1}\alpha, A^{k_2-1}\alpha]^{q_{k_1, k_2}(n)}\\
\nonumber
&\prod_{1\leqslant k_3 < k_2 < k_1 < s}[[A^{k_1-1}\alpha, A^{k_2-1}\alpha], A^{k_3-1}\alpha]^{q_{k_1, k_2, k_3}(n)} \ldots,
\end{align}
where $Ax = [x,\gamma]$ and $q_{k_1, \ldots, k_r}$ are integral polynomials with $\deg q_{k_1, \ldots, k_r} \leqslant k_1 + \ldots + k_r$. More explicitly, we have
\begin{align}\label{explicit expansion of g}
g(n) = \alpha^n (A\alpha)^{{n}\choose{2}} (A^2\alpha)^{{n}\choose{3}} \cdots [A\alpha, \alpha]^{{n}\choose{3}}[A^2\alpha,\alpha]^{{n}\choose{4}} \cdots [A^2\alpha,A\alpha]^{4{{n+1}\choose{5}}}[A^3\alpha,A\alpha]^{5{{n+1}\choose{6}}}\cdots
\end{align}
The coefficients of $g$ can be analysed using a family of subgroups of $G^0$ introduced in Section 12 of \cite{leibman_2009}. For $k_1, \ldots, k_l\in\NN_+$, we let $G^0_{(k_1,\ldots,k_l)}$ be the subgroup of $G^0$ generated by all $l$-fold commutators\footnote{A 1-fold commutator is any element $h\in G$. For $l>1$, an $l$-fold commutator is an element of the form $[h_i, h_j]$, where $h_i$ is an $i$-fold commutator, $h_j$ is an $j$-fold commutator and $i+j = l$.} of elements of the form $A^{k_1-1}h_1$, ..., $A^{k_l-1}h_l$ for $h_1, \ldots, h_l\in G^0$.  We then define
\begin{align*}
G^0_{k,l} = \langle G^0_{(k_1, \ldots, k_i)}: i\geqslant l,\; k_1+\ldots+k_l\geqslant k\rangle
\end{align*}
for integers $1\leqslant l\leqslant k$ and set $G^0_{k,l} = G^0_{l,l}$ whenever $l>k$. 

The following lemma lists some basic properties of the groups $G^0_{k,l}$ that we shall use.
\begin{lemma}\label{basic properties of the filtration on G^0}
For any integers $1\leqslant l\leqslant k$,
\begin{enumerate}
\item $G^0_{k,l}$ is normal in $G$;
\item $[G^0_{k,l}, G^0_{i,j}]\leqslant G^0_{k+i, l+j}$ for any integers $1\leqslant i\leqslant j$;
\item $A^j G^0_{k,l}\leqslant G^0_{k+j,l}$ for any $j\in\NN$;
\item $G^0_{k+1,l}$ and $G^0_{k,l+1}$ are subgroups of $G^{k,l}$, and the quotient groups $G^0_{k,l}/G^0_{k+1,l}$ and $G^0_{k,l}/G^0_{k,l+1}$ are abelian;
\item for $k\geqslant 2$, $G_k = G_k^0 = G^0_{k,1} = \langle A^{k-1}G^0, G^0_{k,2}\rangle = \langle A G^0_{k-1}, G^0_{k,2}\rangle$;
\item $(G^0)^\nabla_k = \langle G^0_{k,2}, G^0_{k+1}\rangle$
\end{enumerate}
\end{lemma}

\begin{proof}
Properties (i)-(iv) are proved in Lemma 12.2 of \cite{leibman_2009}. For $k\geqslant 2$, the statement $G_k = G^0_k$ in (v) is true by definition, and the statement $G_k = G^0_{k,1}$ is proved in Lemma 12.3 of \cite{leibman_2009}. To finish the proof of (v), it remains to show that $G^0_{k,1} = \langle A^{k-1}G^0, G^0_{k,2}\rangle = \langle A G^0_{k-1}, G^0_{k,2}\rangle$ for $k\geqslant 2$. For $k = 2$, this is true by definition of $G^0_{k,1}$ and the fact that $G^0_{k,2}\geqslant G^0_{k,3}\geqslant \ldots$, which follows from part (iv). We assume that the statement is true for some $k\geqslant 2$. That $G_{k+1}^0$ contains $\langle A G^0_k, G^0_{k+1,2}\rangle$ follows from the fact that both $A G^0_k$ and $G^0_{k+1,2}$ are contained in the $(k+1)$-th element of the lower central series of $G$, which is precisely $G^0_{k+1}$. For the other direction, we observe that
\begin{align*}
G^0_{k+1} &=  [G_k, G] = [G^0_k, \langle G^0, \gamma \rangle] \leqslant \langle [G^0_k, G^0], [G^0_k,\gamma] \rangle \\
&\leqslant \langle [A^{k-1} G^0, G^0], [G^0_{k,2}, G^0], AG^0_k\rangle
\leqslant \langle G^0_{k+1,2}, AG_k^0\rangle.
\end{align*}
A similar argument shows that $G^0_{k+1} = \langle A^{k}G^0, G^0_{k+1,2}\rangle$.

Before we prove property (vi), we recall that $(G^0)^\nabla_k = \langle G_{k+1}, [G_j, G_{k-j}]: 1\leqslant j < k\rangle$. That (vi) holds for $k=1$ can be verified by inspection. For $k\geqslant 2$, we observe that $[A^{j-1} G^0, A^{k-j-1}G^0]\leqslant [G_j^0, G_{k-j}^0]$, and so
\begin{align*}
G^0_{k,2} \leqslant \langle [G^0_j, G^0_{k-j}]: 1\leqslant j < k\rangle;
\end{align*}
when coupled with property (v), this implies that $(G^0)^\nabla_k \geqslant \langle G^0_{k,2}, G^0_{k+1}\rangle$. For the converse, we have
\begin{align*}
[G_j^0, G^0_{k-j}] = [\langle A^{j-1}G^0, G^0_{j,2}\rangle, \langle A^{k-j-1}G^0, G^0_{k-j,2}\rangle] \leqslant \langle G^0_{k,2}, G^0_{k,3}, G^0_{k,4}\rangle \leqslant G^0_{k,2},
\end{align*}
for each $1\leqslant j < k$, from which it follows that $(G^0)^\nabla_k \leqslant \langle G^0_{k,2}, G^0_{k+1}\rangle$.
\end{proof}

Letting $g(n) = \prod_{i=1}^s g_{i}^{{n}\choose{i}}$, we observe from (\ref{abstract expansion of g}), (\ref{explicit expansion of g}) as well as parts (v) and (vi) of Lemma \ref{basic properties of the filtration on G^0} that 
\begin{align}
g_i = A^{i-1}\alpha \; \mod \; (G^0)_{i}^\nabla.
\end{align}

For an arbitrary $b\in G^0$, we have $g_b(n) = a^n b \gamma^n = b (\alpha_b \gamma^{-1})^n\gamma^n$, where $\alpha_b = \alpha[\alpha,b] Ab$, as observed in Section 11.3 of \cite{leibman_2009}. Letting $g_b(n) = \prod_{i=0}^s g_{b,i}^{{n}\choose{i}}$, it is therefore true that 
\begin{align}\label{simple formula for coefficients of g}
g_{b,i} = A^{i-1}\alpha_b = A^{i-1}\alpha \mod (G^0)^\nabla_i
\end{align}
for all $i\in\NN_+$.

For $i=1$, we have $g_{b,1} = \alpha$ mod $G^0_2$, and we claim that $g_{b,i}$ is irrational. The ergodicity of $a$ implies that for almost every $b$, the sequence $n\mapsto a^n b$ is equidistributed in $G/\Gamma$, and so the same is true for the sequence $g_b$ in $G^0/\Gamma^0$. Consequently, the projection $\pi(g_b): \ZZ\to G^0/(G^0_2\Gamma^0)$ is equidistributed as well. Since $\pi(g_b(n)) = \pi(b) + \pi(\alpha) n$, it follows that $\pi(\alpha)$ is an irrational element of $G^0/G^0_2$, and so $g_{b,1}$ is an irrational element of $G^0$.

Before proving that $g_{b,i}$ are irrational for $i>1$, we discuss some properties of the map $A: G\to G$. From the definition of the filtration $G^0_\bullet$ we observe that $AG^0_i\leqslant G^0_{i+1}$ for all $i\geqslant 1$ (this is also a consequence of parts (iv) and (v) of Lemma \ref{basic properties of the filtration on G^0}). Therefore the map $A_i := A|_{G^0_i}$ takes values in $G^0_{i+1}$, and moreover $A_i(\Gamma_i)\leqslant \Gamma_{i+1}$. We also observe that the projection $\overline{A}_i: G^0_i\to G^0_{i+1}/(G^0)^\nabla_{i+1}$ is a (continuous) group homomorphism because
\begin{align*}
A(xy) = [xy, \gamma] = [x,\gamma][[x,\gamma],y][y,\gamma] = Ax [Ax,y] Ay = Ax Ay \mod G^0_{2i+1, 2}
\end{align*}
for any $x,y\in G^0_i$ and $G^0_{2i+1,2}\leqslant G^0_{i+1,2}\leqslant (G^0)^\nabla_{i+1}$ by parts (iv) and (vi) of Lemma \ref{basic properties of the filtration on G^0}. From part (v) of Lemma \ref{basic properties of the filtration on G^0} it follows that $\overline{A}_i$ is surjective. Finally, we note using parts (iii) and (v) of Lemma \ref{basic properties of the filtration on G^0} that $A_i((G^0)^\nabla_i)\leqslant (G^0)^\nabla_{i+1}$. 

Suppose that $g_{b,i}$ is irrational but $g_{b,i+1}$ is not for some $1\leqslant i < s$. Then there exists a nontrivial $(i+1)$-th level character $\eta_{i+1}:G^0_{i+1}\to\RR$ such that $\eta_{i+1}(g_{b,i+1})\in\ZZ$. From (\ref{simple formula for coefficients of g}) and the fact that $\eta_{i+1}$ vanishes on $(G^0)^\nabla_{i+1}$, we deduce that $\eta_{i+1}(g_{b,i+1})=\eta_{i+1}(A^i \alpha)$.  We also let $\overline{\eta}_{i+1}:G^0_{i+1}/(G^0)^\nabla_{i+1}\to\RR$ be the induced map.

Let $\eta_i := \eta_{i+1}\circ A_i: G^0_i\to\RR$. It is an $i$-th level character as a consequence of four facts: the vanishing of $\eta_{i+1}$ on $(G^0)_{i+1}^\nabla$, the inclusion $(G^0_{i+1,2})\leqslant (G^0)_{i+1}^\nabla$ (both of which imply that $\eta_i = \overline{\eta}_{i+1}\circ \overline{A}_i$ is a continuous group homomorphism), the inclusion $A_i((G^0)^\nabla_i)\leqslant (G^0)^\nabla_{i+1}$, and the fact that $\eta_i(\Gamma_i)\leqslant\ZZ$. It moreover satisfies $$\eta_i(g_{b,i}) = \eta_i(A^{i-1}\alpha) = \eta_{i+1}(A^i \alpha) = \eta_{i+1}(g_{b,i+1}),$$ implying that  $\eta_i(g_{b,i})\in\ZZ$.
The nontriviality of $\eta_{i+1}$ implies that $\overline{\eta}_{i+1}$ and $\overline{A}_i$ are surjective maps onto nontrivial groups; hence $\eta_i$ is nontrivial. This contradicts the irrationality of $g_{b,i}$. By induction, $g_{b,1}$, ..., $g_{b,s}$ are all irrational, implying that $g_b$ is irrational.
\end{proof}

Proposition \ref{replacing a linear sequence by a polynomial one} is vaguely reminiscent of Proposition 3.1 of \cite{frantzikinakis_kra_2005} in that we replace a linear sequence by a polynomial object on a simpler space. These two results are not equivalent, however, in that in Proposition \ref{replacing a linear sequence by a polynomial one}, we end up with a polynomial sequence on a nilmanifold of a connected group where in Proposition 3.1 of \cite{frantzikinakis_kra_2005}, one obtains a unipotent affine transformation on a torus. 

\begin{lemma}\label{preserving factors when passing to connected component}
Let $G_\bullet$ and $G^0_\bullet$ be as given in Proposition \ref{replacing a linear sequence by a polynomial one}. Then $Z_i(G/\Gamma)=Z_i(G^0/\Gamma^0)=G^0/(G^0_{i+1}\Gamma^0)$ for each $i\in\NN$.
\end{lemma}
\begin{proof}
We do the cases $i=0$ and $i>0$ separately. For $i>0$, we recall from (\ref{Host-Kra factors of nilsystems}) that $Z_i(G/\Gamma) = G/G_{i+1}\Gamma$. Since $G/\Gamma = G^0/\Gamma^0$ by connectedness of $G/\Gamma$, and $G_j = G^0_j$ for $j\geqslant 2$, it follows that $$Z_i(G^0/\Gamma^0) = Z_i(G/\Gamma) = G/G_{i+1}\Gamma = G^0/G^0_{i+1}\Gamma^0.$$

For $i=0$, we have $Z_i(G/\Gamma) = G/G^0\Gamma = 1 = G^0/G^0\Gamma^0 = Z_i(G^0/\Gamma^0)$. 
%for each $s\in\NN_+$, and the same also holds for $s=0$ by the connectedness of $G/\Gamma$. 
\end{proof}

%%%
%%%
%%% Homogeneous versus inhomogeneous polynomial progressions
%%%
%%%

\section{Homogeneous and inhomogeneous polynomial progressions}\label{section on homogeneity}
The central message of this paper is that homogeneous polynomial progressions satisfy certain linear algebraic properties that make them pliable for our analysis. In this section, we explicitly describe these properties.

Let $\vec{P}\in\RR[x,y]^{t+1}$ be an integral polynomial progression. Let $V_k$ be the subspace of $\RR[x,y]$ given by
\begin{align*}
    V_k &= \Span_\RR\{ (x+P_i(y))^j:\; 0\leqslant i\leqslant t,\; 1\leqslant j\leqslant k \}\\
    &=\Span_\RR\left\{ {{x+P_i(y)}\choose{j}}:\; 0\leqslant i\leqslant t,\; 1\leqslant j\leqslant k \right\},
\end{align*}
and similarly let 
\begin{align*}
    W_k =\Span_\RR \left\{ {{x+P_i(y)}\choose{k}}:\; 0\leqslant i\leqslant t \right\}.
\end{align*}
Thus, the space $V_k$ consists of all the polynomial in $x,\; x+P_1(y),\; \ldots,\; x+P_t(y)$ of degree up to $k$ while the space $W_k$ is the span of ``Taylor monomials'' ${{x}\choose{k}}, {{x+P_1(y)}\choose{k}}, \ldots, {{x+P_t(y)}\choose{k}}$ of degree $k$.
We also set
\begin{align*}
V^* = \Span_\RR\{(Q_0, \ldots, Q_{t})\in\RR[u]^{t+1}:
 Q_0(x) + Q_1(x+P_1(y)) + \ldots + Q_{t}(x+P_{t}(y)) = 0 \}
\end{align*}
to be the space of all algebraic relations satisfied by $\vec{P}$. 
%The space $V^*$ is finitely dimensional whenever $P_1$, \ldots, $P_t$ are all distinct. 
We recall that an algebraic relation $(Q_0, \ldots, Q_{t})$ is homogeneous if there exists $d\in\NN$ and $a_0, \ldots, a_d\in\RR$ not all zero such that $Q_i(u) = a_i u^d$ for each $0\leqslant i\leqslant t$. We call $\vec{P}$ {homogeneous} if $V^*$ is spanned by homogeneous algebraic relations, and {inhomogeneous} otherwise. 

The concepts of integral polynomial progression and homogeneity, as well as our results in this paper, could likely be extended to multiparameter polynomial progressions of the form
\begin{align*}
(x,\; x+P_1(y_1, \ldots, y_r),\; \ldots,\; x+P_t(y_1, \ldots, y_r));
\end{align*}
however, we do not pursue this generalisation so as not to obfuscate the notation. 

Some important examples of homogeneous progressions include:
\begin{enumerate}
\item linear progressions $(x,\; x+a_1 y,\; \ldots,\; x+ a_t y)$ for distinct nonzero integers  $a_1, \ldots, a_t$, and more generally linear progressions of the form $(x,\; x + \psi_1(y_1, \ldots, y_r),\; \ldots,\; x + \psi_t(y_1, \ldots, y_r)$ for some linear forms $\psi_1, \ldots, \psi_t:\ZZ^r\to\ZZ$;
\item progressions of algebraic complexity 0, i.e. progressions where the polynomials $P_1, \ldots, P_t$ are integral and linearly independent;
\item progressions of algebraic complexity 1, such as $(x,\; x+y,\; x+y^2,\; x+y+y^2)$, which satisfy no quadratic or higher-order algebraic relation. 
\end{enumerate}

Another, less obvious example of a homogeneous progression is $(x, \; x+y,\; x+2y,\; x+y^3)$, already mentioned in the introduction, which only satisfies the homogeneous relation
\begin{align}\label{relation in 3APs, 2}
    x - 2(x+y) + (x+2y) = 0.
\end{align}
This progression should be contrasted with $(x, \; x+y,\; x+2y,\; x+y^2)$, which is inhomogeneous because it satisfies both (\ref{relation in 3APs, 2}) and the inhomogeneous relation
\begin{align}\label{algebraic relation for x, x+y, x+2y, x+y^2, 2}
    x^2 + 2x - 2(x+y)^2 + (x+2y)^2 - 2(x+y^2) = 0
\end{align}
that cannot be written down as a sum of homogeneous relations. More generally, progressions of the form
\begin{align*}
(x,\; x+y,\; \ldots,\; x+(t-1)y,\; x+ P_t(y))
\end{align*}
are all inhomogeneous whenever $1<\deg P_t < t$ because there exist polynomials $Q_0, \ldots, Q_{t-1}$ of degree $\deg P_t$ for which
\begin{align*}
    Q_0(x) + Q_1(x+y) + \ldots + Q_{t-1}(x+(t-1)y) + (x+P_t(y)) = 0.
\end{align*}

When discussing algebraic relations $(Q_0, \ldots, Q_t)$, we want to move freely between expressing the polynomials $Q_i$ in terms of the standard basis $\{u^k: k\in\NN\}$ on one hand and the Taylor basis $\{{{u}\choose{k}}: k\in\NN\}$ on the other hand. The next lemma allows us to make this transition for homogeneous polynomials.
\begin{lemma}\label{equivalent descriptions of algebraic relations}
Let $t\in\NN_+$ and $\vec{P}\in\RR[x,y]^{t+1}$ be an integral polynomial progression. Let $(Q_0, \ldots, Q_t)$ be an algebraic relation of degree $d$ satisfied by $\vec{P}$, and set $Q_i(u) = \sum_{k=0}^d b_{ik}{{u}\choose{k}}$. Then the following conditions are equivalent:
\begin{enumerate}
    \item The relation $(Q_0, \ldots, Q_t)$ is a sum of homogeneous algebraic relations.
    \item For every $0\leq k\leq d$  and $0\leq j\leq k$, we have 
    \begin{align}\label{homogeneous relations in Taylor monomials}
        b_{0k} {{x}\choose{j}} + b_{1k} {{x+P_1(y)}\choose{j}} + \ldots + b_{tk} {{x+P_t(y)}\choose{j}} = 0.
    \end{align}
    \item For every $0\leq k\leq d$ and $0\leq j\leq k$, we have 
    \begin{align}\label{equation in equivalent descriptions}
        b_{0k} x^j + b_{1k} (x+P_1(y))^j + \ldots + b_{tk} (x+P_t(y))^j = 0.
    \end{align}
\end{enumerate}
\end{lemma}

In particular, the condition (ii) implies that homogeneous relations can equivalently be defined as relations of the form $(Q_0, ..., Q_t) = (a_0 {{u}\choose{d}}, \ldots, a_t {{u}\choose{d}})$. 

\begin{proof}
We first show the equivalence of (ii) and (iii), to be followed by the equivalence of (i) and (iii). 

The implication (iii) $\implies$ (ii) follows from the fact that the polynomial ${{u}\choose{j}}$ is a sum of the polynomials $1, u, \ldots, u^j$. For the converse, we similarly note that $u^j$ is a sum of the polynomials $1, u, \ldots, {{u}\choose{j}}$.

To prove the equivalence of (i) and (iii), we set ${{u}\choose{k}} = \sum_{j=0}^k c_{jk} u^j$ for each $k\in\NN$, so that $Q_i(u) = \sum_{k=0}^d b_{ik}\sum_{j=0}^k c_{jk} u^j$. Importantly, $c_{jk}\neq 0$ for any $0\leq j\leq k$. This allows us to rewrite
\begin{align*}
    0 &= Q_0(x) + Q_1(x+P_1(y)) + \cdots + Q_t(x+P_t(y))\\
    &= \sum_{i=0}^t \sum_{k=0}^d b_{ik}\sum_{j=0}^k c_{jk} (x+P_i(y))^j\\
    &= \sum_{j=0}^k \sum_{i=0}^t \sum_{k=j}^d b_{ik} c_{jk} (x+P_i(y))^j.
\end{align*}
The relation $(Q_0, \ldots, Q_t)$ is a sum of homogeneous algebraic relations if and only if for every $0\leq j\leq d$, we have
\begin{align}\label{condition in the equivalent descriptions}
     \sum_{k=j}^d  c_{jk}\sum_{i=0}^t b_{ik} (x+P_i(y))^j = 0,
\end{align}
and it is immediate from this expression that (iii) implies (i). To prove the implication (i) $\implies$ (iii), we assume that the relation $(Q_0, \ldots, Q_t)$ is indeed a sum of homogeneous algebraic relations, and so (\ref{condition in the equivalent descriptions}) holds for $0\leq j\leq d$. Taking $j = d$ implies (\ref{equation in equivalent descriptions}) for $j=d$, $k=d$, i.e. 
\begin{align*}
    b_{0d} x^d + b_{1d} (x+P_1(y))^d + \ldots + b_{td} (x+P_t(y))^d = 0.    
\end{align*}
Taking partial derivatives with respect to $x$ of the expression above $d-j$ times implies (\ref{equation in equivalent descriptions}) for $k = d$ and $0\leq j \leq d$. 

Thus, (\ref{condition in the equivalent descriptions}) equals
\begin{align}
     \sum_{k=j}^{d-1}  c_{jk}\sum_{i=0}^t b_{ik} (x+P_i(y))^j = 0.
\end{align}
Running the same argument as above, we prove (\ref{equation in equivalent descriptions}) for $k = d-1$ and $0\leq j\leq d-1$. Inducting downwardly on $k$ proves (iii) for all required values of $k$ and $j$.
\end{proof}
We observe that the argument in Lemma \ref{equivalent descriptions of algebraic relations} relied on the fact that the polynomial progression takes the form
\begin{align}\label{form of the progression}
    (x,\; x+P_1(y),\; \ldots,\; x+P_t(y))
\end{align}
(taking $P_1, \ldots, P_t$ to be polynomials of several variables would also do) rather than the more general form
\begin{align}\label{more general progression}
    (P_1(x, y),\; \ldots,\; P_t(x, y)).
\end{align}
This is because when the progression takes the form (\ref{form of the progression}), we can use partial differentiating with respect to $x$ to lower the degree of algebraic relations. Without this, Lemma \ref{equivalent descriptions of algebraic relations} need not hold, and so we would not have the same correspondence between relations of the form $(Q_0, ..., Q_t) = (a_0 u^d, \ldots, a_t u^d)$ and $(b_0 {{u}\choose{d}}, \ldots, b_t {{u}\choose{d}})$. 

The special form (\ref{form of the progression}) of our progression also ensures that we do not encounter issues similar to what has been discovered by Altman with regards to the original proof and statement of Theorem 1.13 from \cite{green_tao_2010a} (see \cite{tao_2020} for the explanation of the problem). Similar issues would have quite plausibly appeared, however, if we dealt with more general progressions like (\ref{more general progression}).

We define several more families of polynomial vector spaces. 
For $k\in\NN_+$, we let 
\begin{align*}
    W_k^c = W_k\cap\sum_{j\neq k}W_j \quad {\rm{and}}\quad W^c = \sum_k W_k^c,
\end{align*}
as well as the family of quotient spaces
\begin{align*}
    W'_k = W_k/W_k^c = W_k/{\left(W_k\cap\sum_{j\neq k}W_j\right)}.
\end{align*}

The space $W_k^c$ captures all the polynomials in $W_k$ that ``participate'' in inhomogeneous algebraic relations, an intuition made more precise by the result below and the examples discussed below Proposition \ref{bound on algebraic complexity}. The notation $W_k^c$ is supposed to signify the fact that $W_k^c$ is a complement of the subspace $W_k'$ inside $W_k$.

\begin{proposition}[Equivalent conditions for homogeneity]\label{equivalent conditions for homogeneity}
Let $t\in\NN_+$ and $\vec{P}\in\RR[x,y]^{t+1}$ be an integral polynomial progression. The following are equivalent:
\begin{enumerate}
\item $\vec{P}$ is homogeneous;
\item $W_k^c$ is trivial for each $k\in\NN_+$;
\item $W'_k = W_k$ for each $k\in\NN_+$.
\end{enumerate}
\end{proposition}

Intuitively, Proposition \ref{equivalent conditions for homogeneity} states that homogeneity is equivalent to the fact that for every $k\in\NN_+$, there are no polynomials in $W_k$ that could be used in constructing an inhomogeneous algebraic relation (condition (ii)). When proving Theorem \ref{finitary equidistribution on nilmanifolds}, our key equidistribution result on nilmanifolds, we will use the condition (iii) of Proposition \ref{equivalent conditions for homogeneity}.

\begin{proof}
The equivalence of (ii) and (iii) follows trivially from the definition of $W'_k$, and we focus on showing the equivalence of (i) and (ii) instead. The inhomogeneity of $\vec{P}$ implies the existence of a nontrivial algebraic relation
\begin{align*}
(Q_0(u), \ldots, Q_t(u)) %= \left(\sum_k a_{0k} u^k, \ldots, \sum_k a_{tk} u^k\right) 
= \left(\sum_k b_{0k} {{u}\choose{k}}, \ldots, \sum_k b_{tk} {{u}\choose{k}} \right)    
\end{align*}
 that is not a sum of homogeneous algebraic relations. By Lemma \ref{equivalent descriptions of algebraic relations}, this means that there exists $k\in\NN_+$ and $0\leq j \leq k$ for which
 \begin{align}\label{homogeneous relations in Taylor monomials, 2}
    b_{0k} {{x}\choose{j}} + b_{1k} {{x+P_1(y)}\choose{j}} + \ldots + b_{tk} {{x+P_t(y)}\choose{j}} \neq 0.
 \end{align}
 We claim that in fact we can take $j=k$. We define the discrete derivative of $Q\in\RR[u]$ to be $\partial Q(u) = Q(u+1) - Q(u)$, and the partial discrete derivative of $R\in\RR[x,y]$ with respect to $x$ to be $\partial_x R(x,y) = R(x+1, y) - R(x,y)$. Observing that $\partial {{u}\choose{k}} = {{u+1}\choose{k}}-{{u}\choose{k}} = {{u}\choose{k-1}}$, we deduce that $\partial_x {{x+P_i(y)}\choose{k}} = {{x+P_i(y)}\choose{k-1}}$. It follows that if 
\begin{align*}%\label{homogeneous relations in Taylor monomials}
    b_{0k} {{x}\choose{j}} + b_{1k} {{x+P_1(y)}\choose{j}} + \ldots + b_{tk} {{x+P_t(y)}\choose{j}} = 0,
\end{align*}
then applying the partial discrete derivative with respect to $x$ to the expression above $k-j$ times would imply
\begin{align*}
    b_{0k} {{x}\choose{j}} + b_{1k} {{x+P_1(y)}\choose{j}} + \ldots + b_{tk} {{x+P_t(y)}\choose{j}} = 0,
 \end{align*}
 contradicting (\ref{homogeneous relations in Taylor monomials, 2}). We can thus assume that $j=k$ in \eqref{homogeneous relations in Taylor monomials, 2}. 
 
Since
\begin{align*}
Q_0(x) + Q_1(x+P_1(y)) + \ldots + Q_t(x+P_t(y)) = 0,
\end{align*}
we have
\begin{align*}
R(x,y) = -\sum_{j\neq k} \sum_{i=0}^t b_{ij} {{x+P_i(y))}\choose{j}} \in \sum_{j\neq k} W_j,
\end{align*}
and so $W^c_k = W_k\cap\sum_{j\neq k} W_j$ is nontrivial. Thus (ii) implies (i) by contrapositive. The argument can be reversed, and so (i) and (ii) are in fact equivalent. 
\end{proof}

For homogeneous progressions, it is quite straightforward to obtain an upper bound on algebraic complexity.
\begin{proposition}\label{bound on algebraic complexity}
Let $t\in\NN_+$ and $\vec{P}\in\RR[x,y]^{t+1}$ be a homogeneous polynomial progression. Then $\A_i(\vec{P})\leqslant t-1$ for each $0\leqslant i\leqslant t$.
\end{proposition}
This bound is sharp, as evidenced by the example of arithmetic progressions.

\begin{proof}
Suppose that $\A_i(\vec{P}) \geq  t$, and let 
\begin{align*}
Q_0(x) + Q_1(x+P_1(y)) + \ldots + Q_t(x+P_t(y)) = 0
\end{align*}
be an algebraic relation of degree $d\geq t$. Taking partial derivative with respect to $x$ of the expression above $d-t$ times, we can assume that the relation has degree $t$. The homogeneity of $\vec{P}$ and Lemma \ref{equivalent descriptions of algebraic relations} impliy the existence of a nontrivial algebraic relations of the form
\begin{align}\label{homogeneous relation}
a_0 {{x}\choose{t}} + a_1 {{x+P_i(y)}\choose{t}} +  \ldots + a_t {{x+P_t(y)}\choose{t}} = 0.
\end{align}
The relation (\ref{homogeneous relation}) and the formula
\begin{align*}
    {{x+P_i(y)}\choose{t}}={{x}\choose{t}}+{{x}\choose{t-1}}P_i(y) + {{x}\choose{t-2}}{{P_i(y)}\choose{2}}+ \ldots +{{P_i(y)}\choose{t}},
\end{align*}
imply
\begin{align*}
a_1 {{P_i(y)}\choose{k}} +  \ldots + a_t {{P_t(y)}\choose{k}} = 0  
\end{align*}
for $1\leq k\leq t$. This gives us $t$ equations
\begin{alignat*}{2}
    &a_1 P_1(y) + \ldots + &&a_t P_t(y) = 0\\
    &a_1 P_1(y)^2 + \ldots + &&a_t P_t(y)^2 = 0\\
    &\vdots  &&\vdots \\
    &a_1 P_1(y)^t + \ldots + &&a_t P_t(y)^t = 0.
\end{alignat*}
The invertibility of the Vandermonde matrix and the distinctness of the polynomials $P_1, \ldots, P_t$ imply that these $t$ equations can only be satisfied if $a_1 = \ldots = a_t = 0$, which also implies $a_0 = 0$. This contradicts the nontriviality of (\ref{homogeneous relation}).

\end{proof}

Proposition \ref{equivalent conditions for homogeneity} implies that homogeneous progressions satisfy
\begin{align}\label{direct sum for homogeneous configurations}
    V_k = \bigoplus_{i=1}^k W_i = \bigoplus_{i=1}^k W'_i.
\end{align}
In the inhomogeneous case, we instead have
\begin{align}\label{direct sum for inhomogeneous configurations}
    V_k = \sum_{i=1}^k W_i = \left(\bigoplus_{i=1}^k W'_i\right)\oplus(W^c\cap V_k) 
\end{align}
for some nontrivial subspace $W^c\cap V_k$. The nontriviality of this subspace is the main source of difficulty preventing us from generalising Theorem \ref{Main theorem} to inhomogeneous progressions.

Given the rather abstract nature of the spaces $W_k, W'_k$ and $W^c_k$, we illustrate their definitions with concrete examples. For the homogeneous progression $(x, \; x+y,\; x+2y,\; x+y^3)$, we have
\begin{align*}
    W_1' = W_1 = \Span_\RR\{x, y, y^3\}\quad {\rm{and}} \quad W'_2 = W_2 = \Span_\RR\left\{{{x}\choose{2}}, xy+ {{y}\choose{2}}, y^2, xy^3 + {{y^3}\choose{2}}\right\}, 
\end{align*}
while for the inhomogeneous progression $(x, \; x+y,\; x+2y,\; x+y^2)$, we have
\begin{align*}
    W_1 = \Span_\RR\{x, y, y^2\}\quad {\rm{and}} \quad W_2 = \Span_\RR\left\{{{x}\choose{2}}, xy+ {{y}\choose{2}}, y^2, xy^2 + {{y^2}\choose{2}}\right\} 
\end{align*}
but 
\begin{align*}
    W_1' = \Span_\RR\{x, y\}, \quad W_2' = \Span_\RR\left\{{{x}\choose{2}}, xy+ {{y}\choose{2}}, xy^2 + {{y^2}\choose{2}}\right\} \quad {\rm{and}} \quad W^c = \Span_\RR\{y^2\}.
\end{align*}
The nontriviality of $W^c$ for the latter progression is intrinsically related to the algebraic relation (\ref{algebraic relation for x, x+y, x+2y, x+y^2, 2}).

The spaces $V_k$ and $W_k$ are subspaces of $\RR[x,y]$.
%(the space $W'_k$ has been defined as a quotient space, but there is a natural embedding into $\QQ[x,y]$). 
We also need an analogous family of subspaces of $\RR^{t+1}$. For a polynomial progression $\vec{P}\in\RR[x,y]^{t+1}$, we let
\begin{align*}
\vec{P}^k(x,y) = (x^k, (x+P_1(y))^k, \ldots, (x+P_{t}(y))^k) \; {\rm{and}} \; {{\vec{P}(x,y)}\choose{k}} = \left({{x}\choose{k}}, {{x+P_1(y)}\choose{k}}, \ldots, {{x+P_{t}(y)}\choose{k}}\right).
\end{align*}
We then define 
\begin{align*}
    \P_k &=\Span_{\RR}\left\{\vec{P}^k(x,y): x, y\in\RR\right\} =\Span_{\RR}\left\{(x^k, (x+P_1(y))^k, \ldots, (x+P_{t}(y))^k): x,y\in\RR\right\}
\end{align*}
for each $k\in\NN_+$. The following lemma gives equivalent formulas for the spaces $\P_k$.
\begin{lemma}\label{equivalent conditions for P_k}
Let $t, k\in\NN_+$ and $\vec{P}\in\RR[x,y]^{t+1}$ be an integral polynomial progression. Then
\begin{align*}
    \P_k &= \Span_{\RR}\left\{\vec{P}^k(x,y): x, y\in\RR\right\} &&= \Span_{\RR}\left\{\vec{P}^j(x,y): x, y\in\RR,\; 1\leq j\leq k\right\}\\
    &= \Span_{\RR}\left\{{{\vec{P}(x,y)}\choose{k}}: x, y\in\RR\right\} &&= \Span_{\RR}\left\{{{\vec{P}(x,y)}\choose{j}}: x, y\in\RR,\; 1\leq j\leq k\right\}.
\end{align*}
\end{lemma}
\begin{proof}
We fix $k\in\NN_+$ and denote
\begin{align*}
    A_1 &= \Span_{\RR}\left\{\vec{P}^k(x,y): x, y\in\RR\right\},\quad &&A_2 = \Span_{\RR}\left\{\vec{P}^j(x,y): x, y\in\RR,\; 1\leq j\leq k\right\},\\
    A_3 &= \Span_{\RR}\left\{{{\vec{P}(x,y)}\choose{k}}: x, y\in\RR\right\},\quad &&A_4 = \Span_{\RR}\left\{{{\vec{P}(x,y)}\choose{j}}: x, y\in\RR,\; 1\leq j\leq k\right\} 
\end{align*}
for the four spaces mentioned in the statement of the lemma. It is clear that $A_1\subseteq A_2$ and $A_3 \subseteq A_4$. To prove the converse inclusions, we note that $\frac{\partial}{\partial_x}\vec{P}(x,y)^k = k \vec{P}(x,y)^{k-1}$ and $\partial_x{{\vec{P}(x,y)}\choose{k}} = {{\vec{P}(x,y)}\choose{k-1}}$, where $\frac{\partial}{\partial_x}$ is the usual partial derivative with respect to $x$ and $\partial_x$ is the partial discrete derivative with respect to $x$ defined in the proof of Lemma \ref{equivalent descriptions of algebraic relations}. For every $R(x,y)\in A_1$ and $h\neq 0$, the expression $$\cfrac{R(x+h, y)-R(x, y)}{h}$$ is still in $A_1$, and so $\frac{\partial}{\partial_x}R(x,y)\in A_1$ by the closeness of $A_1$. Applying $\frac{\partial}{\partial_x}$ to $\vec{P}(x,y)^k$ exactly $k-j$ times with the observations above, we deduce that $\vec{P}(x,y)^j$ is still in $A_1$. Hence $A_2\subseteq A_1$. Analogously, applying $\partial_x$ exactly $k-j$ times to ${{\vec{P}(x,y)}\choose{k}}$, we deduce that ${{\vec{P}(x,y)}\choose{j}}\in A_3$, and so $A_4\subseteq A_3$.

It remains to show that $A_2 = A_4$. For this, we note that $\vec{P}(x,y)^k$ is a linear combination of ${{\vec{P}(x,y)}\choose{1}}, \ldots, {{\vec{P}(x,y)}\choose{k}}$, and conversely ${{\vec{P}(x,y)}\choose{k}}$ is a linear combination of $\vec{P}(x,y), \ldots, \vec{P}(x,y)^k$, from which the equality $A_2 = A_4$ follows. 
\end{proof}

Henceforth, we treat $\RR^{t+1}$ as an $\RR$-algebra with coordinatewise multiplication $\vec{v}\cdot\vec{w}=(v(0) w(0), \ldots, v(t) w(t))$ for $\vec{v}=(v(0), \ldots, v(t))$ and $\vec{w}=(w(0), \ldots, w(t))$. We similarly let $A\cdot B = \{\vec{a}\cdot\vec{b}: \vec{a}\in A, \vec{b}\in B\}$ be the product set of $A$ and $B$ for any $A,B\subseteq\RR^{t+1}$. With these definitions, we observe that $\P_{i+j}\leqslant\P_i\cdot\P_j$, but the converse is in general not true. We also set $\vec{e}_i$ to be the coordinate vector with 1 in the $i$-th place and 0 elsewhere.

We conclude this section by relating the spaces $W_k$ and $W'_k$ to $\P_k$.
Let $t_k = \dim W_k$ and $t'_k = \dim W'_k$ for each $k\in\NN$. The spaces $W_k$ and $\P_k$ are connected as follows. Let $\{Q_{k,1}, \ldots, Q_{k,t_k}\}$ be a basis for $W_k$. Then
\begin{align*}
    \left({{x}\choose{k}}, {{x+P_1(y))}\choose{k}}, \ldots, {{x+P_t(y)}\choose{k}}\right) = \sum_{j=1}^{t_k} \vec{v}_{k,j} Q_{k,j}(x,y)
\end{align*}
for some linearly independent vectors $\vec{v}_{k,1}, \ldots, \vec{v}_{k,t_k}\in\RR^{t+1}$.
We let $\tau_k(Q_{k,j})=\vec{v}_{k,j}$, and extend this map to all of $W_k$ by linearity. This map depends on the choice of the basis for $W_k$. It is surjective by the definition of $\P_k$ and injective by the linear independence of $\vec{v}_{k,1}, \ldots, \vec{v}_{k,t_k}$. Hence it is a vector space isomorphism. In particular, Proposition \ref{equivalent conditions for homogeneity} implies that $W'_k\cong \P_k$ whenever $\vec{P}$ is homogeneous, a fact that we shall use a lot in the proof of Theorem \ref{finitary equidistribution on nilmanifolds}. 

To illustrate the aforementioned correspondence between $W_k$ and $\P_k$, consider the progression $(x,\; x+y,\; x+2y,\; x+y^3)$. The isomorphisms $\tau_1$ and $\tau_2$ are given by
\begin{align*}
    \tau_1(x) = (1,1,1,1),\quad \tau_1(y) = (0,1,2,0),\quad \tau_1(y^3) = (0,0,0,1)
\end{align*}
and
\begin{align*}
    \tau_2\left({{x}\choose{2}}\right) &= (1,1,1,1),\quad \tau_2\left(xy+ {{y}\choose{2}}\right) = (0,1,2,0),\\
    \tau_2(y^2) &= (0,0,1,0),\quad \tau_2\left(xy^3 + {{y^3}\choose{2}}\right) = (0,0,0,1).
\end{align*}

%%%
%%%
%%% Relating Host-Kra complexity to algebraic complexity
%%%
%%%

\section{Relating Host-Kra complexity to algebraic complexity}\label{section on Host-Kra complexity}
Having introduced the notation for the spaces $\P_i$, we are ready to show precisely how determining Host-Kra complexity for homogeneous progressions can be reduced to a certain equidistribution problem on nilmanifolds. We start by defining a group which contains the orbit (\ref{polynomial orbit 2}). Groups of this form have previously been defined in \cite{leibman_2009, green_tao_2010a, candela_sisask_2012, kuca_2020b}, among others.

\begin{definition}[Leibman group]\label{Leibman group}
Let $t\in\NN_+$ and $G$ be a connected group with a filtration $G_\bullet$ of degree $s$. For an integral polynomial progression $\vec{P}\in\RR[x,y]^{t+1}$, we define the associated \emph{Leibman group} to be 
\begin{align*}
G^P = \langle g_i^{\vec{v}_i}: g_i\in G_i, \vec{v}_i\in\P_i, 1\leqslant i\leqslant s \rangle,
\end{align*}
where $h^{\vec{v}}=(h^{v(0)}, \ldots,h^{v(t)})$ for any $h\in G$ and $\vec{v} = (v(0), \ldots, v(t))\in\RR^{t+1}$. We also set $\Gamma^P=G^P\cap G^{t+1}$. If $g\in\poly(\ZZ, G_\bullet)$, then we denote
\begin{align*}
g^P(x,y) = (g(x), g(x+P_1(y)), \ldots, g(x+P_t(y)))
\end{align*}
and observe that $g^P$ takes values in $G^P$. %Lastly, we let $\mu^P$ be the Haar measure on $G^P/\Gamma^P$.
\end{definition}

\begin{lemma}\label{Leibman group contains subgroups}
Let $t\in\NN_+$ and $G$ be a connected group with a filtration $G_\bullet$ of degree $s$. Suppose that $\vec{P}\in\RR[x,y]^{t+1}$ is an integral polynomial progression with $\A_i(\vec{P}) = s'$ for some $s'\in\NN$ and some $0\leq i\leq t$. Then 
$G^P$ contains $1^{i}\times G_{s'+1} \times 1^{t-i}$.
\end{lemma}
\begin{proof}
The assumption $\A_i(\vec{P}) = s'$ implies that $(x+P_i(y))^{s'+1}$ is linearly independent from $(x+P_k(y))^{s'+1}$ for $k\neq i$, hence  $\P_{s'+1}$ contains $\vec{e}_i$. The Lemma then follows by the definition of $G^P$.
\end{proof}

We are now ready to state an infinitary version of the main technical result in the paper. This result constitutes the first part of Theorem \ref{dichotomy}.

\begin{theorem}\label{infinitary equidistribution on nilmanifolds}
Let $t\in\NN_+$ and $G$ be a connected group with filtration $G_\bullet$. Suppose that $g\in\poly(\ZZ, G_\bullet)$ is irrational and that $\vec{P}\in\RR[x,y]^{t+1}$ is a homogeneous polynomial progression. Then $g^P$ is equidistributed on the nilmanifold $G^P/\Gamma^P$.
\end{theorem}

Importantly, Theorem \ref{infinitary equidistribution on nilmanifolds} fails for inhomogeneous progressions in that for each inhomogeneous progression $\vec{P}$, we can find a nilmanifold $G/\Gamma$, a filtration $G_\bullet$, and an irrational sequence $g\in\poly(\ZZ, G_\bullet)$ for which the orbit of $g^P$ is contained in a proper subnilmanifold of $G^P/\Gamma^P$. An example of this is given in Section \ref{section on failure in the inhomogeneous case}.

We have all the tools to prove Theorem \ref{infinitary equidistribution on nilmanifolds} by now. However, we will later need a finitary version of Theorem \ref{infinitary equidistribution on nilmanifolds}, and so instead of proving twice what is essentially the same result, we shall only give the finitary proof later on and deduce Theorem \ref{infinitary equidistribution on nilmanifolds} from it. For now, however, we can show how the $\HK_i(\vec{P}) \leqslant \A_i(\vec{P})$ part of Theorem \ref{Main theorem} follows from Theorem \ref{infinitary equidistribution on nilmanifolds}.

\begin{corollary}\label{Host-Kra complexity equals algebraic complexity}
Let $t\in\NN_+$ and $\vec{P}\in\RR[x,y]^{t+1}$ be a homogeneous polynomial progression. For any $0\leqslant i\leqslant t$, we have
\begin{align*}
    \HK_i(\vec{P}) \leqslant \A_i(\vec{P}).
\end{align*}
\end{corollary}
The converse inequality will follow from showing that algebraic complexity equals Weyl complexity, and that Weyl complexity is less than or equal to Host-Kra complexity, both of which are done in Section \ref{section on Weyl complexity}.

\begin{proof}[Proof of Corollary \ref{Host-Kra complexity equals algebraic complexity} using Theorem \ref{infinitary equidistribution on nilmanifolds}]
Let $\A_i(\vec{P}) = s$. Let $(X,\X,\mu,T)$ be a totally ergodic system, $f_0, \ldots, f_t\in L^\infty(\mu)$, and suppose that $\EE(f_i|\Z_s)=0$. By Theorem \ref{Host-Kra factors are characteristic}, the expression
\begin{align}\label{weak limit in the proof}
    \lim_{N\to\infty}\EE_{n\in [N]}\int_X f_0 \cdot T^{P_1(n)} f_1 \cdots T^{P_{t}(n)} f_{t} d\mu
\end{align} 
remains unchanged if we project the functions $f_0, \ldots, f_t$ onto the factor $\Z_{s_0}$ for some $s_0\in\NN$. If $s_0<s$, then $\EE(f_i|\Z_{s_0}) = 0$ and the limit (\ref{weak limit in the proof}) is 0, so we can assume that $s_0\geqslant s$. Since the factor $\Z_{s_0}$ is an inverse limit of $s_0$-step nilsystems, we can approximate $X$ by totally ergodic nilsystems.

Let $(G/\Gamma, \G/\Gamma, \nu, T_a)$ be a totally ergodic nilsystem, and $G_\bullet$ be the lower central series filtration on $G$. Using (\ref{Host-Kra factors of nilsystems}), it suffices to show that if $f_0, \ldots, f_t\in L^\infty(\nu)$ and $f_i$ vanishes on each coset of $G_{s+1}\Gamma$, then 
\begin{align*} %\label{multiple averages on nilmanifolds}
\lim_{N\to\infty}\EE_{n\in[N]} \int_{G/\Gamma} f_0(b\Gamma)\cdot f_1(a^{P_1(n)}b\Gamma) \cdots f_t(a^{P_t(n)} b\Gamma) d\nu(b\Gamma) = 0.
\end{align*}

Let $G^0_\bullet$ be the filtration on $G^0$ given by $G^0_\bullet = G_\bullet\cap G^0$, and let $g_b\in\poly(\ZZ, G^0_\bullet)$ be the irrational sequence defined in Proposition \ref{replacing a linear sequence by a polynomial one} for which $a^n b\Gamma = g_b(n)\Gamma$. The irrationality of $g_b$, Lemma \ref{introducing a parameter to the integral} and Theorem \ref{infinitary equidistribution on nilmanifolds} imply that
\begin{align*}
&\lim_{N\to\infty}\EE_{n\in[N]} \int_{G/\Gamma} f_0(b\Gamma)\cdot f_1(a^{P_1(n)}b\Gamma) \cdots f_t(a^{P_t(n)} b\Gamma) d\nu(b\Gamma)\\
&= \int_{G^0/\Gamma^0} \lim_{N\to\infty}\EE_{m,n\in[N]}  f_0(g_b(m)\Gamma^0)\cdot f_1(g_b(m+P_1(n))\Gamma^0) \cdots f_t(g_b(m+P_t(n))\Gamma^0) d\nu(b\Gamma^0)\\
& = \int_{(G^0)^P/(\Gamma^0)^P} f_0 \otimes \cdots \otimes f_t d\nu^P,
\end{align*}
where $(G^0)^P$ is the Leibman group for $\vec{P}$ and $\nu^P$ is the Haar measure on $(G^0)^P/(\Gamma^0)^P$.

The assumption that $f_i$ vanishes on each coset of $G_{s+1}\Gamma$ in $G/\Gamma$ together with Lemma \ref{preserving factors when passing to connected component} imply that $f_i$ vanishes on each coset of $G^0_{s+1}\Gamma^0$ inside $G^0/\Gamma^0$. By Lemma \ref{Leibman group contains subgroups}, the group $(G^0)^P$ contains $H=1^{i}\times G^0_{s+1} \times 1^{t-i}$; therefore
\begin{align*}
&\left|\int_{(G^0)^P/(\Gamma^0)^P} f_0 \otimes \cdots \otimes f_t \right| \leqslant \int_{(G^0)^P/H(\Gamma^0)^P} \left|\int_{x H(\Gamma^0)^P} f_0 \otimes \cdots \otimes f_t \right|\\
&\leqslant\left(\prod_{j\neq i}||f_j||_\infty\right)\int_{(G^0)^P/H(\Gamma^0)^P} \left|\int_{x_i G^0_{s+1}\Gamma^0} f_i\right| = 0,
\end{align*}
implying that $\Z_s$ is characteristic for the weak convergence of $\vec{P}$ at $i$. 
\end{proof}

Corollary \ref{Host-Kra complexity equals algebraic complexity} implies that if a progression $\vec{P}$ satisfies $\A_i(\vec{P})=s$, then $\Z_s$ is characteristic for the weak or $L^2$ convergence of $\vec{P}$ at $i$ for any totally ergodic system. We now prove Corollary \ref{factors for ergodic systems}, which extends this result to ergodic systems for eligible progressions, with a slight modification in the $s=0$ case. The proof is almost identical to the proof of Proposition 4.1 in \cite{frantzikinakis_2008}.

\begin{proof}[Proof of Corollary \ref{factors for ergodic systems}]
Let $\vec{P}\in\RR[x,y]^{t+1}$ be an eligible homogeneous progression with $\A_i(\vec{P}) = s$ and $(X,\X,\mu,T)$ be ergodic. By Theorem \ref{Host-Kra factors are characteristic}, there exists a Host-Kra factor that is characteristic for the weak and $L^2$ convergence of $\vec{P}$. Since each Host-Kra factor is an inverse limit of nilsequences, we can approximate $X$ by an ergodic nilsystem $(G/\Gamma, \G/\Gamma, \nu, T_a)$. The compactness of $G/\Gamma$ and the assumption that $G$ is generated by the connected component $G^o$ and $a$ imply that $a^r \in G^o$ for some $r\in\NN_+$; and hence

\begin{align}\label{multiple average of a dilate}
\EE_{n\in[rN]}\prod_{i=1}^t T_a^{P_i(n)}f_i &= \EE_{j\in[r]} \EE_{n\in[N]} \prod_{i=1}^t T_a^{P_i(r(n-1)+j)} f_i \\
\nonumber &= \EE_{j\in[r]} \EE_{n\in[N]} \prod_{i=1}^t (T_a^r)^{\tilde{P}_{i,j}(n)} (T_a^{P_i(j)} f_i),
\end{align}
where $\tilde{P}_{i,j}(n) = \frac{P_i(r(n-1)+j) - P_i(j)}{r}$. This is where we use the fact that $\vec{P}$ is eligible. The definition of eligibility implies that for any $0\leqslant j < r$, the progression 
\begin{align*}
\vec{\tilde{P}}_j(x,y) = (x,\; x+\tilde{P}_{1,j}(y),\; ...,\; x+\tilde{P}_{t,j}(y))
\end{align*}
is homogeneous and that $\A_i(\vec{\tilde{P}}_j) = \A_i(\vec{P})$ for every $0\leq i < r$.

If $s>0$, suppose that $\EE(f_i|\Z_s(T_a)) = 0$. Then the equality $\Z_s(T_a) = \Z_s(T_a^r)$ and the $T_a$-invariance of $\Z_s$ imply that $\EE(T_a^{P_i(j)} f_i|\Z_s(T_a^r)) = 0$. We deduce from  Corollary \ref{Host-Kra complexity equals algebraic complexity} and the total ergodicity of $T_a^r$ on each connected components of $G/\Gamma$ that the expression in (\ref{multiple average of a dilate}) converges to 0 as $N\to\infty$.

If $s=0$, suppose that $\EE(f_i|\K_{rat}(T_a)) = 0$. The total ergodicity of $T_a^r$ implies that $\K_{rat}(T_a) = \Z_0(T_a^r)$, and so $\EE(T_a^{P_i(j)} f_i|\Z_0(T_a^r)) = 0$. Again, it follows from Corollary \ref{Host-Kra complexity equals algebraic complexity} and the total ergodicity of $T_a^r$ on each connected components of $G/\Gamma$ that the expression in (\ref{multiple average of a dilate}) converges to 0 as $N\to\infty$.

% $T_a^r$ is totally ergodic on each of the disjoint systems $a^i G^o/\Gamma_0$ for $0\leqslant i\leqslant r-1$. 
\end{proof}

We now show that progressions of algebraic complexity at most 1 are eligible, which together with Corollary \ref{factors for ergodic systems} immediately implies Corollary \ref{factors for ergodic systems for progressions of complexity 1}.

\begin{lemma}\label{progressions of complexity 1 are homogeneous and eligible}
Let $\vec{P}\in\RR[x,y]^{t+1}$ be an algebraic progression with $\max_i \A_i(\vec{P})\leq 1$. Then $\vec{P}$ is homogeneous and eligible. 
\end{lemma}
\begin{proof}
From the definition of inhomogeneous relations it follows that each inhomogeneous relations must have degree at least 2. Thus, the fact that $\vec{P}$ has algebraic complexity at most 1 immediately implies that it is homogeneous.

To prove that $\vec{P}$ is eligible, we fix $r\in\NN_+$ and $0\leq j < r$. We show that the progression
\begin{align*}
\vec{\tilde{P}}(x,y) = (x,\; x+\tilde{P}_{1,j}(y),\; \ldots,\; x+\tilde{P}_{t,j}(y)),
\end{align*}
where $\tilde{P}_{i,j}(y) = \frac{P_i(r(y-1)-j) - P_i(y)}{r}$, also has algebraic complexity at most 1, from which the eligibility of $\vec{P}$ will follow easily. Indeed, suppose first that $\vec{\tilde{P}}$ satisfies an algebraic relation of degree 2:
\begin{align*}
    \sum_{i=0}^t a_{i2}\left(x+\frac{P_i(r(y-1) - j)-P_i(j)}{r}\right)^2 + a_{i1}\left(x+\frac{P_i(r(y-1) - j)-P_i(j)}{r}\right) = 0. 
\end{align*}
Setting $y' = r(y-1) - j,\; x' = x r,\; a_{i2}' = a_{i2}/r^2,\; a_{i1}' = a_{i1}/r$ for brevity and rearranging, we deduce that 
\begin{align*}
    \sum_{i=0}^t\left(a_{i2}'(x'+P_i(y'))^2 + (a'_{i1} - 2a'_{i2}P_i(j)) (x'+P_i(y')) + a'_{i2}P_i(j)^2 - a'_{i1} P_i(j)\right) = 0.
\end{align*}
The homogeneity of $\vec{P}$ implies that
\begin{align*}
    \sum_{i=0}^t a'_{i2}(x'+P_i(y'))^2 = 0,
\end{align*}
and the fact that $\vec{P}$ has algebraic complexity at most 1 further implies that $a'_{02} = \ldots = a_{t2}' = 0$. The claim $a_{02} = \ldots = a_{t2} = 0$ follows by rescaling. Thus, $\vec{\tilde{P}}$ satisfies no algebraic relation of degree 2. It follows by induction that $\vec{\tilde{P}}$ satisfies no algebraic relation of degree $d > 2$ since each such relation
\begin{align}\label{algebraic relation in the proof of eligibility}
    Q_0(x) + Q_1(x+\tilde{P}_{1,j}(y)) + \ldots + Q_t(x+\tilde{P}_{t,j}(y)) = 0
\end{align}
would induce an algebraic relation of degree $d-1$ by partially differentiating (\ref{algebraic relation in the proof of eligibility}) with respect to $x$. This establishes the claim that $\vec{\tilde{P}}$ has algebraic complexity at most 1. Thus, every algebraic relation satisfied by $\vec{\tilde{P}}$ is of the form
\begin{align*}
 a_0 x + a_1 \left(x+\frac{P_1(r(y-1)+j) - P_1(j)}{r}\right) + \ldots + a_t \left(x+\frac{P_t(r(y-1)+j) - P_t(j)}{r}\right) = 0
\end{align*}
and corresponds to an algebraic relation
\begin{align*}
 a_0 x + a_1 (x+P_1(y)) + \ldots + a_t (x+P_t(y)) = 0
\end{align*}
satisfied by $\vec{P}$. This one-to-one correspondence between the algebraic relations satisfied by $\vec{\tilde{P}}$ and $\vec{P}$ implies the eligibility of $\vec{P}$.
\end{proof}

Theorem \ref{infinitary equidistribution on nilmanifolds} also allows us to prove the second part of Corollary \ref{counting result}.
\begin{proof}[Proof of  Corollary \ref{counting result}(ii)]
Let $(X, \X, \mu, T)$ be a totally ergodic system, and suppose that $\vec{P}\in\RR[x,y]^{t+1}$ is an integral progression with algebraic complexity at most 1. This implies that $\vec{P}$ is homogeneous since each inhomogeneous  algebraic relation must have degree at least 2. For each $0\leqslant i\leqslant t$, let $P_i(y) = \sum_{j=1}^d a_{i,j} Q_j(y)$ and $L_i(y_1, \ldots y_d) = \sum_{j=1}^d a_{i,j} y_j$ for some $a_{i,j}\in\ZZ$ and integral polynomials $Q_1, \ldots, Q_d$. Letting 
\begin{align*}
\vec{L}(x,y_1, \ldots, y_d) = (x,\; x+L_1(y_1, \ldots, y_d), \; \ldots, \; x+L_t(y_1, \ldots, y_d)), 
\end{align*}
we observe that $\vec{P}(x,y) = \vec{L}(x, Q_1(y), \ldots, Q_d(y))$. It follows that $\vec{L}$ also has an algebraic complexity at most 1, since each algebraic relation of degree $(j_0, \ldots, j_t)$ between terms of $\vec{L}$ would immediately imply an algebraic relation of the same degree between terms of $\vec{P}$ after substituting $y_i = Q_i(y)$.

Using the same argument as in the proof of Corollary \ref{Host-Kra complexity equals algebraic complexity}, we reduce the question of understanding
\begin{align}\label{multiple average for g^P}
\lim_{N\to\infty}\EE_{n\in [N]} \int_X \prod_{i=0}^t T^{P_i(n)} f_i d\mu
\end{align}
to understanding
\begin{align}\label{multiple average for g^P, 2}
\lim_{N\to\infty}\EE_{x,y\in[N]} F(g^P(x,y))
\end{align}
for each essentially bounded function $F: (G/\Gamma)^{t+1}\to\CC$ and an irrational sequence $g\in\poly(\ZZ, G_\bullet)$ for some filtration $G_\bullet$ on $G$. Following the same method to analyse
\begin{align}\label{multiple average for g^L}
\lim_{N\to\infty}\EE_{y_1, ..., y_d\in [N]} \int_X \prod_{i=0}^t T^{L_i(y_1, \ldots, y_d)} f_i d\mu,
\end{align}
we deduce that understanding (\ref{multiple average for g^L}) comes down to estimating
\begin{align}\label{multiple average for g^L, 2}
\lim_{N\to\infty}\EE_{x,y_1, ..., y_d\in[N]} F(g^L(x,y_1, \ldots, y_d)),
\end{align}
where 
$$g^L(x,y_1, ..., y_d) = (g(x), g(x+L_1(y_1, \ldots, y_d)), \ldots, g(x+L_t(y_1, \ldots, y_d))).$$

By Theorem \ref{infinitary equidistribution on nilmanifolds}, the limit in (\ref{multiple average for g^P, 2}) equals $\int_{G^P/\Gamma^P} F$; by Theorem 11 of \cite{green_tao_2010a}\footnote{While Theorem 11 of \cite{green_tao_2010a} has been shown to fail in general, its corrected version, to be found in the arXiv version of the paper at https://arxiv.org/abs/1002.2028, still holds in this case, as the system under consideration is translation invariant.}, the limit in (\ref{multiple average for g^L, 2}) is $\int_{G^L/\Gamma^L} F$ for some subgroup $G^L\leqslant G^{t+1}$. From the fact that $\max_i \A_i(\vec{P})\leqslant 1$ we deduce that $G^P = \langle h_1^{\vec{v}_1}, G_2^{t+1}: h_1\in G_1, \vec{v}_1\in\P_1 \rangle$; similarly, the construction of the group $G^L$ in \cite{green_tao_2010a} and the fact that $\vec{L}$ has algebraic complexity at most 1 reveals that $G^L = \langle h_1^{\vec{v}_1}, G_2^{t+1}: h_1\in G_1, \vec{v}_1\in\L_1 \rangle$, where
\begin{align*}
\L_1 = \Span_{\RR}\{(x, x+L_1(y_1, \ldots, y_d), \ldots, x+L_t(y_1, \ldots, y_d)): x,y_1, \ldots, y_d\in\RR\}.
 \end{align*}
We observe that $\P_1 = \L_1$; from this it follows that $G^P = G^L$, and so the limits in (\ref{multiple average for g^P, 2}) and (\ref{multiple average for g^L, 2}) are equal. This implies that (\ref{multiple average for g^P}) and (\ref{multiple average for g^L}) equal as well.

\end{proof}

%%%
%%%
%%% Finitary nilmanifold theory
%%% 
%%%

\section{Finitary nilmanifold theory}\label{section on finitary nilmanifold theory}

Before we can prove a finitary version of Theorem \ref{infinitary equidistribution on nilmanifolds}, we need to introduce necessary finitary concepts required for this task. Most concepts and definitions in this and next section are taken from \cite{green_tao_2010a, green_tao_2012, candela_sisask_2012}. Throughout this section, we assume that $G$ is connected, and that each nilmanifold $G/\Gamma$ comes with a filtration $G_\bullet$ and a Mal'cev basis $\chi$ adapted to $G_\bullet$. We call a nilmanifold endowed with filtration and a Mal'cev basis \emph{filtered}. A \emph{Mal'cev basis} is a basis for the Lie algebra of $G$ with some special properties; since we do not explicitly work with the notion of Mal'cev basis or its rationality in the paper, we refer the reader to \cite{green_tao_2012} for definitions of these concepts. What matters for us is that each Mal'cev basis induces a diffemomorphism $\psi: G\to\RR^m$, called \emph{Mal'cev coordinate map}, which satisfies the following properties:
\begin{enumerate}
\item $\psi(\Gamma) = \ZZ^m$;
\item $\psi(G_i) = \{0\}^{m-m_i}\times \RR^{m_i}$, where $m_i = \dim G_i$.
\end{enumerate}
Thus, $\psi$ provides a natural coordinate system on $G$ that respects the filtration $G_\bullet$ and the lattice $\Gamma$.  Similarly to $\psi$, we define maps $\psi_i:G_i\to\RR^{m_i-m_{i+1}}$ by assigning to each element of $G_i$ its Mal'cev coordinates indexed by $m-m_i + 1$, ..., $m-m_{i+1}$. With this definition, we have $\psi_i(x)=0$ if and only if $x\in G_{i+1}$, and $\psi_i(x)\in\ZZ^{m_i-m_{i+1}}$ if and only if $x\in\Gamma_i$.

\begin{definition}[Complexity of nilmanifolds]\label{complexity of nilmanifolds}
A filtered nilmanifold $G/\Gamma$ has complexity $M$ if the degree $s$ of the filtration $G_\bullet$, the dimension $m$ of the group $G$, and the rationality of the Mal'cev basis $\chi$ are all bounded by $M$. 
\end{definition}

We remark that complexity of nilmanifolds has nothing to do with the four notions of complexity of polynomial progressions that we examine. Neither does complexity of nilsequences defined below.

\begin{definition}[Nilsequences]\label{nilsequences}
A function $f:\ZZ\to\CC$ is a \emph{nilsequence} of {degree} $s$ and {complexity} $M$ if $f(n)=F(g(n)\Gamma)$, where $F:G/\Gamma\to\RR$ is an $M$-Lipschitz function on a filtered nilmanifold $G/\Gamma$ of degree $s$ and complexity $M$, and $g\in\poly(\ZZ,G_\bullet)$. 
\end{definition}

\begin{definition}[Quantitative equidistribution]\label{quantitative equidistribution}
Let $D\in\NN_+$ and $\delta>0$. A sequence $g\in\poly(\ZZ^D,G)$ is $(\delta,N)$-\emph{equidistributed} on $G/\Gamma$ if
\begin{align*}
    \left|\EE_{n\in[N]^D}F(g(n)\Gamma)-\int_{G/\Gamma}F \right|\leqslant\delta\norm{f}_\Lip
\end{align*}
for all Lipschitz functions $F:G/\Gamma\to\CC$, where $\norm{f}_\Lip$ is the Lipschitz norm on $F$ with respect to a metric defined in \cite{green_tao_2012}.
\end{definition}

It has been shown in Theorem \ref{Leibman's equidistribution theorem} that equidistribution is related to horizontal characters. Given the Mal'cev coordinate map $\psi: G\to\RR^m$, each horizontal character can be written in the form $\eta(x) = k\cdot\psi(x)$ for some $k\in\ZZ^m$. We call $|\eta|=|k|=|k_1|+\ldots+|k_m|$ the \emph{modulus} of $\eta$. Similarly, each $i$-th level character $\eta_i:G_i\to\RR$ is of the form $\eta_i(x) = k\cdot\psi_i(x)$ for some $k\in\ZZ^{m_i-m_{i+1}}$, and we define its modulus to be $|\eta_i|=|k| = |k_1|+\ldots+|k_{m_i-m_{i+1}}|$. 

We shall also need to quantify the notion of polynomials that are ``almost constant" mod $\ZZ$, using a definition from \cite{green_tao_2012}. In what follows, $||x||_{\RR/\ZZ} = \min\{|x-n|: n\in\ZZ\}$ is the \emph{circle norm} of $x\in\RR$.
\begin{definition}[Smoothness norm]\label{smoothness norm}
Let $$Q(n_1, \ldots, n_D) = \sum_{i=0}^d \sum_{i_1+\ldots+i_D = i} a_{i_1, \ldots, i_D} {{n}\choose{i_1}}\cdots{{n}\choose{i_d}}$$ be a polynomial in $\RR[n_1, \ldots, n_D]$. For $N\in\NN_+$, we define the \emph{smoothness norm} of $Q$ to be
\begin{align*}
||Q||_{C^\infty[N]} = \max\{N^{i_1+\ldots+i_D} ||a_{i_1, \ldots, i_D}||_{\RR/\ZZ}: i_1, \ldots, i_D\in\NN,\; 1\leq i_1+\cdots + i_D \leq d\}.
\end{align*}
\end{definition}
In particular, $||Q||_{C^\infty[N]}$ is bounded from above as $N\to\infty$ if and only if $Q$ is constant mod $\ZZ$.

With these definitions, we are ready to state a quantitative version of Theorem \ref{Leibman's equidistribution theorem}

\begin{theorem}[Quantitative Leibman's equidistribution theorem, Theorem 2.9 of \cite{green_tao_2012}]\label{quantitative Leibman's equidistribution theorem}
Let $\delta>0$, $M\geqslant 2$ and $D, N\in\NN_+$ with $D\leqslant M$. Let $G/\Gamma$ be a filtered nilmanifold of complexity $M$ and $g\in\poly(\ZZ^D, G_\bullet)$. Then there exists $C_M>0$ such that at least one of the following is true:
\begin{enumerate}
\item $g$ is $(\delta,N)$-equidistributed in $G/\Gamma$;
\item there exists a nontrivial horizontal character $\eta$ of modulus $|\eta|\ll\delta^{-C_{M}}$ for which $||\eta\circ g||_{C^\infty[N]}\ll \delta^{-C_{M}}$.
\end{enumerate}
\end{theorem}

We now need to quantify the notion of irrationality. 
\begin{definition}[Quantitative irrationality]\label{quantitative irrational sequences}
Let $G/\Gamma$ be a filtered nilmanifold of degree $s$, and suppose $A,N>0$. An element $g_i\in G_i$ is \emph{$(A,N)$-irrational} if for every nontrivial $i$-th level character $\eta:G_i\to\RR$ of modulus $|\eta|\leqslant A$, we have $||\eta(g_i)||_{\RR/\ZZ}\geqslant A/N^i$. It is $A$-irrational if for every nontrivial $i$-th level character $\eta:G_i\to\RR$ of modulus $|\eta|\leqslant A$, we have $\eta\circ g_i\notin\ZZ$. We say that a sequence $g\in\poly(\ZZ, G_\bullet)$ is $(A,N)$-irrational (respectively $A$-irrational) if $g_i$ is $(A,N)$-irrational (respectively $A$-irrational) for each $1\leqslant i\leqslant s$. Similarly, we say that the nilsequence $n\mapsto F(g(n)\Gamma)$ is $(A,N)$- or $A$-irrational if the polynomial sequence $g$ is.
\end{definition}
Clearly, $(A,N)$-irrationality is stronger than $A$-rationality, but for some of our applications the latter notion will be sufficient. 

We are now ready to state the finitary version of Theorem \ref{infinitary equidistribution on nilmanifolds}, which is the main technical result of this paper, and derive Theorem \ref{infinitary equidistribution on nilmanifolds} from it.

\begin{theorem}\label{finitary equidistribution on nilmanifolds}
Let $t\in\NN_+$ and $A,M,N\geqslant 2$. Let $G/\Gamma$ be a filtered nilmanifold of complexity $M$. Suppose that $g\in\poly(\ZZ, G_\bullet)$ is $(A,N)$-irrational, $F:(G/\Gamma)^{t+1}\to\CC$ is $M$-Lipschitz, and $\vec{P}\in\RR[x,y]^{t+1}$ is a homogeneous polynomial progression. Then
\begin{align*}
\EE_{x,y\in[N]} F(g^P(x,y)\Gamma^{t+1})  = \int_{G^P/\Gamma^P} F + O_M(A^{-c_M})
%\EE_{x,y\in[N]} F(g^P(x,y)\Gamma^{t+1})  = \int_{G^P/\Gamma^P} F + o_{A\to\infty,M}(1) + o_{N\to\infty, M}(1).
\end{align*}
for some $c_M>0$.
\end{theorem}
\begin{proof}[Proof of Theorem \ref{infinitary equidistribution on nilmanifolds} using Theorem \ref{finitary equidistribution on nilmanifolds}]
Let $F:(G/\Gamma)^{t+1}\to\RR$ be a continuous function.  By the Stone-Weierstrass theorem, Lipschitz functions on a compact set form a dense subset of the algebra of continuous functions. Approximating $F$ by a sequence of Lipschitz functions if necessary, we can assume without loss of generality that $F$ is Lipschitz. We let $M$ be the maximum of the complexity of $G/\Gamma$ and the Lipschitz norm of $F$.

Let $g\in\poly(\ZZ, G_\bullet)$ be an irrational sequence. For each $N\in\NN_+$, we let $A_N$ be the maximal real number $A$ for which $g$ is $(A_N, N)$-irrational. We claim that $A_N\to\infty$ as $N\to\infty$. If not, then there exists some number $A>0$ and an index $i\in\NN_+$ with the property that $g_i$ is not $(A,N)$-irrational for all $N\in\NN_+$. We fix this $i$. It follows that there exists a sequence of nontrivial $i$-th level characters $\eta_{N}: G_i\to\RR$ of modulus at most $A$ such that $||\eta_N(g_i)||_{\RR/\ZZ}<A/N^i$. Since there are only finitely many $i$-th level characters of modulus bounded by $A$, we conclude that there exists a nontrivial $i$-th level character $\eta$ of modulus at most $A$ such that $||\eta(g_i)||_{\RR/\ZZ} < A/N^i$ for all $N\in\NN_+$. Taking $N\to\infty$, we see that $\eta(g_i)\in\ZZ$, contradicting the irrationality of $g_i$. 

It therefore follows from Theorem \ref{finitary equidistribution on nilmanifolds} that
\begin{align*}
\EE_{x,y\in[N]} F(g^P(x,y)\Gamma^{t+1})  = \int_{G^P/\Gamma^P} F + O_M(A_N^{-c_M})
%\EE_{x,y\in[N]} F(g^P(x,y)\Gamma^{t+1})  = \int_{G^P/\Gamma^P} F + o_{A_N\to\infty,M}(1) + o_{N\to\infty, M}(1).
\end{align*}
Since $M$ is constant, letting $N\to\infty$ sends the error term to 0, implying that $g^P$ is equidistributed on $G^P/\Gamma^P$ as claimed. 
\end{proof}

%%%
%%%
%%% Reducing true complexity to an equidistribution question
%%% 
%%%

\section{Reducing true complexity to an equidistribution question}\label{section on true complexity}

In Sections \ref{section on reducing to connected groups}-\ref{section on finitary nilmanifold theory}, we have shown how the question of determining Host-Kra complexity for homogeneous progressions can be reduced to showing that $g^P$ is equidistributed on $G^P/\Gamma^P$. Determining true complexity for homogeneous progression comes down to the exact same equidistribution question. All the arguments in this section can be viewed as finitary analogues of arguments in previous sections.

Since we are now primarily concerned with functions from $\ZZ/N\ZZ$ to $\CC$, we shall need an $N$-periodic version of certain previously defined concepts. In this section, $N$ is always a prime, and the group $G$ is connected. A function $f:\ZZ/N\ZZ\to\CC$ is called \emph{1-bounded} whenever $\norm{f}_\infty\leqslant 1$.
\begin{definition}[Periodic sequences]\label{$N$-periodicity}
Let $G_\bullet$ be a filtration on $G$. A sequence $g\in\poly(\ZZ, G_\bullet)$ is \emph{$N$-periodic} if $g(n+N)g(n)^{-1}\in\Gamma$ for each $n\in\ZZ$, and it is periodic if it is $N$-periodic for some $N>0$. A nilsequence $n\mapsto F(g(n)\Gamma)$ is $N$-periodic (resp. periodic) if $g$ is. 
\end{definition}

Given a homogeneous polynomial progression $\vec{P}\in\RR[x,y]^{t+1}$, we want to show that $\A_i(\vec{P}) = \T_i(\vec{P})$ for each $0\leqslant i\leqslant t$. The forward inequality $\A_i(\vec{P}) \leq \T_i(\vec{P})$ is straightforward to derive (see Theorem 1.13 in \cite{kuca_2020b}); it is the reverse inequality that poses a challenge. We thus want to prove the following.
\begin{theorem}\label{true complexity is bounded by algebraic complexity}
Let $t\in\NN_+$, $\vec{P}\in\RR[x,y]^{t+1}$ be a homogeneous polynomial progression, $0\leqslant i\leqslant t$, and suppose that $\A_i(\vec{P}) = s$. For every $\epsilon>0$, there exist $\delta>0$ and $N_0\in\NN$ such that for all primes $N>N_0$ and all 1-bounded functions $f_0, \ldots, f_t:\ZZ/N\ZZ\to\CC$, we have
\begin{align*}
    \left|\EE_{x,y\in\ZZ/N\ZZ}f_0(x) f_1(x+P_1(y))\cdots f_t(x+P_t(y))\right| < \epsilon
\end{align*}
whenever $\norm{f_i}_{U^{s+1}}<\delta$. 
\end{theorem}

We know that each progression is controlled by \emph{some} Gowers norm. The result below plays the same role in deriving Theorem \ref{true complexity is bounded by algebraic complexity} as Theorem \ref{Host-Kra factors are characteristic} plays in the proof of Corollary \ref{Host-Kra complexity equals algebraic complexity}.
\begin{proposition}[Proposition 2.2 of \cite{peluse_2019}]\label{control by some Gowers norm}
Let $\vec{P}\in\RR[x,y]^{t+1}$ be an integral polynomial progression. There exists $s\in\NN_+$ with the following property: for every $\epsilon>0$, there exist $\delta>0$ and $N_0\in\NN$ such that for all primes $N>N_0$ and all 1-bounded functions $f_0, \ldots, f_t:\ZZ/N\ZZ\to\CC$, we have
\begin{align*}
    \left|\EE_{x,y\in\ZZ/N\ZZ}f_0(x) f_1(x+P_1(y))\cdots f_t(x+P_t(y))\right| < \epsilon
\end{align*}
whenever $\norm{f_i}_{U^{s+1}}<\delta$ for some $0\leqslant i\leqslant t$.
\end{proposition}

Next, we want to perform a finitary analogue of the approximation-by-nilsystems argument. This can be achieved with the help of a periodic version of celebrated arithmetic regularity lemma from \cite{green_tao_2010a} in which the same polynomial sequence $g$ is used in the decomposition of several functions.

\begin{lemma}[Lemma 2.13 of \cite{kuca_2020b}]\label{regularity lemma}
Let $s, t\in\NN_+ $, $\epsilon>0$, and $\mathcal{F}:\RR_+\to\RR_+$ be a growth function. There exists $M=O_{\epsilon,\mathcal{F}}(1)$, a filtered nilmanifold $G/\Gamma$ of degree $s$ and complexity at most $M$, and an $N$-periodic, $\mathcal{F}(M)$-irrational sequence $g\in\poly(\ZZ,G_\bullet)$ satisfying $g(0)=1$ such that for all 1-bounded functions $f_0, \ldots, f_t:\ZZ/N\ZZ\to\CC$, there exist decompositions
\begin{align*}
    f_i = f_{i, nil} + f_{i,sml} + f_{i,unf}
\end{align*}
where
\begin{enumerate}
    \item $f_{i,nil}(n)=F_i(g(n)\Gamma)$ for $M$-Lipschitz function $F_i: G/\Gamma\to\CC$,
    \item $||f_{i,sml}||_2\leqslant \epsilon$,
    \item $||f_{i,unf}||_{U^{s+1}}\leqslant \frac{1}{\mathcal{F}(M)}$,
    \item the functions $f_{i,nil}$, $f_{i,sml}$ and $f_{i,unf}$ are 4-bounded,
    %\item if $f_i$ takes values in $[0,1]$, then so do $f_{i,nil}$ and $f_{i,nil}+f_{i,sml}$.
\end{enumerate}
\end{lemma}

The last piece that we need is a finitary, periodic version of Theorem \ref{finitary equidistribution on nilmanifolds}.
\begin{proposition}\label{periodic equidistribution on nilmanifolds}
Let $t\in\NN_+$ and $A,M,N\geqslant 2$. Let $G/\Gamma$ be a filtered nilmanifold and complexity $M$. Suppose that $g\in\poly(\ZZ, G_\bullet)$ is an $A$-irrational, $N$-periodic polynomial sequence, $F:(G/\Gamma)^{t+1}\to\CC$ is $M$-Lipschitz and 1-bounded, and $\vec{P}\in\RR[x,y]^{t+1}$ is a homogeneous polynomial progression. Then
\begin{align*}
\EE_{x,y\in\ZZ/N\ZZ} F(g^P(x,y)\Gamma^{t+1})  = \int_{G^P/\Gamma^P} F + O_M(A^{-c_M})
\end{align*}
for some $c_M>0$.
\end{proposition}

\begin{proof}[Proof of Proposition \ref{periodic equidistribution on nilmanifolds} using Theorem \ref{finitary equidistribution on nilmanifolds}]
Let $g\in\poly(\ZZ, G_\bullet)$ be $A$-irrational and $N$-periodic. We claim that $g$ is $(A,Nk)$-irrational for all sufficiently large $k\in\NN_+$. If not, then there exists $1\leqslant i\leqslant s$ such that for each $k\in\NN_+$ there exists an $i$-th level character $\eta_{i,k}:G_i\to\RR$ of complexity at most $A$ satisfying $||\eta_{i,k}(g_i)||_{\RR/\ZZ}< A/(Nk)^i$. The $N$-periodicity of $g_i$ implies that $g_i^{N^i}\in\Gamma_i$ mod $G_{i+1}^\nabla$ (Lemma 5.3 of \cite{candela_sisask_2012}); hence $\eta_{i,k}(g_i)\in\frac{1}{N^i}\ZZ$. Thus, $\eta_{i,k}(g_i)\in\ZZ$ whenever $k^i>A$. In particular, since we can take $k$ arbitrarily large, there exists a nontrivial $i$-th level character $\eta_{i,k}$ of complexity at most $A$ for which $\eta_{i,k}(g_i)\in\ZZ$, contradicting the $A$-irrationality of $g$. Hence  $g$ is $(A,Nk)$-irrational for all sufficiently large $k\in\NN_+$.

Applying Theorem \ref{finitary equidistribution on nilmanifolds}, we deduce that
\begin{align*}
&\EE_{x,y\in\ZZ/N\ZZ} F(g^P(x,y)\Gamma^{t+1})  = \EE_{x,y\in[Nk]} F(g^P(x,y)\Gamma^{t+1}) + O(1/k)\\
&  = \int_{G^P/\Gamma^P} F + O_M(A^{-c_M}) + O(1/k)
\end{align*}
for all sufficiently large $k\in\NN_+$. Taking $k\to\infty$ finishes the proof. 
\end{proof}

Theorem \ref{true complexity is bounded by algebraic complexity} is a special case of Theorem 8.1 of \cite{kuca_2020b}, the proof of which is analogous to the derivation of Corollary \ref{Host-Kra complexity equals algebraic complexity} from Theorem \ref{infinitary equidistribution on nilmanifolds}. Here, we only sketch the steps taken in the derivation of Theorem 8.1 of \cite{kuca_2020b}, and we refer the reader to \cite{kuca_2020b} for all the details. First, we use Proposition \ref{control by some Gowers norm} and Lemma \ref{regularity lemma} to replace the functions $f_0, \ldots, f_t$ by irrational, periodic nilsequences. Second, we use Proposition \ref{periodic equidistribution on nilmanifolds} to approximate the sum by an integral of some Lipschitz function $F$ over $G^P/\Gamma^P$. Third, we use the fact that $\A_i(\vec{P})=s$ to conclude that $1^i \times G_{s+1}\times 1^{t_i}$ is a subgroup of $G^P$. Fourth, we use disintegration theorem to bound $\int_{G^P/\Gamma^P}$ by averages of some Lipschitz function $F_i$ over cosets of $G_{s+1}\Gamma$. Fifth, we use the assumption that $f_i$ has a small $U^{s+1}$ norm to conclude that averages of $F_i$ over cosets of $G_{s+1}\Gamma$ are small. From this follows the smallness of 
\begin{align*}
    \EE_{x,y\in\ZZ/N\ZZ}f_0(x) f_1(x+P_1(y))\cdots f_t(x+P_t(y)).
\end{align*}
The proof of Theorem 8.1 of \cite{kuca_2020b} makes this argument precise and illustrates how all the error quantities are taken care of. 

Finally, Proposition \ref{periodic equidistribution on nilmanifolds} together with Theorem 9.1 of \cite{kuca_2020b} imply part (i) of Corollary \ref{counting result}.

%%%
%%%
%%% The proof of Theorem \ref{finitary equidistribution on nilmanifolds}
%%%
%%%

\section{The proof of Theorem \ref{finitary equidistribution on nilmanifolds}}\label{section on the full homogeneous case}

To complete the proofs of Corollary \ref{Host-Kra complexity equals algebraic complexity} and Theorem \ref{true complexity is bounded by algebraic complexity}, it remains to derive Theorem \ref{finitary equidistribution on nilmanifolds}. Before we prove Theorem \ref{finitary equidistribution on nilmanifolds} for an arbitrary homogeneous progression, we want to deduce the theorem in the special case of $\vec{P} = (x, \; x+y,\; x+2y,\; x+y^3)$. This will help illustrate the method, and we will later compare this progression with $(x, \; x+y,\; x+2y,\; x+y^2)$ to see what is failing in the inhomogeneous case. The method is an adaptation of the proof of Theorem 1.11 from \cite{green_tao_2010a}, however the linear algebraic component coming from the fact that we are dealing with polynomial progressions is much more involved. The method used here is somewhat similar to the methods used in \cite{kuca_2020b}; here, however, we perform downward induction on the degree of subgroups $G_i$ whereas in \cite{kuca_2020b}, we induct downwardly on the degree of monomials in $\eta\circ g^P$. 

\begin{proposition}\label{equidistribution for x, x+y, x+2y, x+y^3}
Let $A,M,N\geqslant 2$. Let $G/\Gamma$ be a filtered nilmanifold of degree 2 and complexity $M$. Suppose that $g\in\poly(\ZZ, G_\bullet)$ is an $(A,N)$-irrational sequence satisfying $g(0)=1$, $F:(G/\Gamma)^{t+1}\to\CC$ is $M$-Lipschitz, and $\vec{P} = (x, \; x+y,\; x+2y,\; x+y^3)$. Then
\begin{align*}
\EE_{x,y\in[N]} F(g^P(x,y)\Gamma^4)  = \int_{G^P/\Gamma^P} F + O_M(A^{-c_M})
\end{align*}
for some $c_M>0$.
\end{proposition}
The assumption that $G$ has a filtration of degree 2 is made to simplify the exposition, and because all the difficulties that emerge in higher-step cases are already present here.

%For an integral polynomial $Q(x,y) = \sum\limits_{j,k} c_{jk}{{x}\choose{j}} {{y}\choose{k}}$ and $m\in\NN$, we set $Q^{(m)}(x,y) = \sum\limits_{j+k\geqslant m} c_{jk}{{x}\choose{j}} {{y}\choose{k}}$. 
We shall need the following lemma.

\begin{lemma}\label{breaking an integral polynomial into a sum of integral polynomials}
Let $t\in\NN_+$ and $\vec{P}\in\RR[x,y]^{t+1}$ be a homogeneous polynomial progression, $\epsilon>0$, and $s, N\in\NN_+$. Let $W_i\leqslant\RR[x,y]$ be as defined in Section \ref{section on homogeneity}, and for each $1\leqslant i\leqslant s$, let $Q_{i,1}, ..., Q_{i, t_i}$ be a basis for $W_i$ composed of integral polynomials. Suppose that $a_{ij}$ are real numbers such that the polynomial
\begin{align*}
Q(x,y) = \sum_{i=1}^s \sum_{j=1}^{t_i} a_{ij} Q_{i,j}(x,y)
\end{align*}
satisfies $||Q||_{C^\infty[N]}\leqslant\epsilon$. Then there exists a positive integer $q=O(1)$ with the property that $||q a_{sj} ||_{\RR/\ZZ}\ll\epsilon N^{-s}$ for all $1\leqslant j\leqslant t_s$.
\end{lemma}
%The property (\ref{direct sum for homogeneous configurations}) satisfied by homogeneous progressions ensures that the family $\{Q_{i,j}: i\in\NN_+, 1\leq j\leq t_i\}$ is linearly independent, which would not be true for inhomogeneous progressions.
\begin{proof}
For $s\in\NN_+$, we let $W_s, V_s$ be as in Section \ref{section on homogeneity}. We also define 
\begin{align*}
\tilde{W}_s = \Span_\RR\{(x+P_i(y))^s: 0\leqslant i\leqslant t\}\quad {\rm{and}}\quad U_s = \Span_\RR\left\{{{x}\choose{i}}{{y}\choose{j}}: i+j < s\right\}.
\end{align*}
We want to show first that $\dim W_s/U_s = \dim W_s = t_s$, i.e. that the polynomials $Q_{s, 1}, ..., Q_{s, t_s}$ remain linearly independent when we subtract from them the monomials in the Taylor basis of degree less than $s$. While this claim may plausibly hold for any polynomial progression, we prove it for homogeneous progressions since this is the only case in which we need this result. The homogeneity of $\vec{P}$ implies that $W_s\cong V_s/V_{s-1}\cong \tilde{W}_s$. Therefore $W_s/U_s\cong V_s/U_s V_{s-1} \cong \tilde{W}_s/ U_s\cong \tilde{W}_s,$ where the last isomorphism follows from the fact no polynomial in $\tilde{W}_s$ has a nonzero monomial of degree less than $s$. The claim $\dim W_s/U_s = t_s$ follows.

Let $Q(x,y) = \sum\limits_{k,l} c_{kl} {{x}\choose{k}}{{y}\choose{l}}$ and $\tilde{Q}(x,y) = \sum\limits_{k+l\geqslant s} c_{kl} {{x}\choose{k}}{{y}\choose{l}}$. Thus, $\tilde{Q} = Q$ mod $U_s$, and it satisfies $||\tilde{Q}||_{C^\infty[N]}\leqslant\epsilon$. Setting $Q_{i,j}(x,y) = \sum\limits_{k,l} b_{klij} {{x}\choose{k}}{{y}\choose{l}}$, we similarly let $\tilde{Q}_{i,j}(x,y) = \sum\limits_{k+l\geqslant s} b_{klij} {{x}\choose{k}}{{y}\choose{l}}$. We deduce from $\dim W_k/U_k = t_k = \dim W_k$ that $\tilde{Q}_{s,1}$, ..., $\tilde{Q}_{s,t_s}$ are linearly independent.

From the definitions of $Q$ and $b_{klij}$ it follows that $c_{kl} = \sum\limits_{i,j} b_{klij} a_{ij} $, and that $||c_{kl}||_{\RR/\ZZ}\leqslant \epsilon N^{-(k+l)}\leqslant \epsilon N^{-s}$ whenever $k+l\geqslant s$.

Let $u$ be the number of pairs $(k,l)$ with $k+l\geqslant s$ for which $c_{kl}\neq 0$. The fact that $\dim W_s/U_s = t_s$ implies that $u\geqslant t_s$. Indexing these pairs as $(k_1, l_1), ..., (k_u, l_u)$ in some arbitrary fashion, we obtain an $u\times s$ matrix $B = (b_{k_r l_r i j})_{r}$ as well as a $t_s$-dimensional column vector $a = (a_{sj})_{j}$ and a $u$-dimensional column vector $c = (c_{k_r l_r})_{r}$ such that $Ba = c$. The linear independence of $\tilde{Q}_{s,1}, ..., \tilde{Q}_{s,t_s}$ implies that there exists an invertible $t_s\times t_s$ submatrix $\tilde{B}$ of $B$ and a $t_s$-dimensional column vector $\tilde{c}$ such that $\tilde{B}a = \tilde{c}$. Since the entries of $\tilde{B}$ are integers of size $O(1)$, the entries of $\tilde{B}^{-1}$ are rational numbers of height $O(1)$. Therefore, there exists a positive integer $q=O(1)$ for which the entries of the matrix $q\tilde{B}^{-1}$ are integers of size $O(1)$. The equality $a = \tilde{B}^{-1}\tilde{c}$ and the condition $||c_{kl}||_{\RR/\ZZ}\leqslant \epsilon N^{-s}$ whenever $k+l\geqslant s$ imply that $||q a_{sj}||_{\RR/\ZZ}\ll \epsilon N^{-s}$ for $1\leqslant j\leqslant t_s$, as claimed. 
\end{proof}

\begin{proof}[Proof of Proposition \ref{equidistribution for x, x+y, x+2y, x+y^3}]
Let $\vec{P} = (x, \; x+y,\; x+2y,\; x+y^3)$. We set 
\begin{align*}
\vec{v}_1 = (1,1,1,1),\quad \vec{v}_2 = (0,1,2,0),\quad \vec{v}_3 = (0,0,0,1) \quad {\rm{and}}\quad \vec{v}_4 = (0,0,1,0)
\end{align*}
and observe that
\begin{align*}
\vec{P}(x,y) &= \vec{v}_1 x + \vec{v}_2 y + \vec{v}_3 y^3\\ 
{{\vec{P}(x,y)}\choose{2}} &= \vec{v}_1 {{x}\choose{2}}+ \vec{v}_2 \left(xy+{{y}\choose{2}}\right)  + \vec{v}_3 \left(xy^3 + {{y^3}\choose{2}}\right) + \vec{v}_4 y^2.
 %{{\vec{P}(x,y)}\choose{3}} &= \vec{v}_1{{x}\choose{3}} + \vec{v}_2 \left({{x}\choose{2}}y +x{{y}\choose{2}}+{{y}\choose{3}}\right)\\
 %& + \vec{v}_3 \left({{x}\choose{2}}y^3 + x{{y^3}\choose{2}}+{{y^3}\choose{3}}\right) + \vec{v}_4 \left(xy^2 + 2y{{y}\choose{2}}\right).
\end{align*}

Thus, we have 
\begin{align*}
\P_1 = \Span_\RR\{\vec{v}_1, \vec{v}_2, \vec{v}_3\} \quad {\rm{and}}\quad \P_2 = \P_3 = \ldots = \Span_\RR\{\vec{v}_1, \vec{v}_2, \vec{v}_3, \vec{v}_4\} = \RR^4
\end{align*}
as well as
\begin{align*}
G^P = G^{\vec{v}_1} G^{\vec{v}_2} G^{\vec{v}_3} G_2^4,
\end{align*}
where $H^{\vec{w}} = \langle h^{\vec{w}}: h\in H\rangle$ for any subgroup $H\leq G$.

We shall prove Proposition \ref{equidistribution for x, x+y, x+2y, x+y^3} by applying Theorem \ref{quantitative Leibman's equidistribution theorem}.
Suppose that $g^P$ is not $(c_M A^{-C_M},N)$-equidistributed on $G^P/\Gamma^P$ for some constants $0<c_M<1<C_M$. By Theorem \ref{quantitative Leibman's equidistribution theorem}, there exists a nontrivial horizontal character ${\eta: G^P\to\RR}$ of modulus at most $cA$, for which $||\eta\circ g^P||_{C^\infty[N]}\leqslant cA$ for some constant $c>0$ that depends on $c_M$ and $C_M$. The constant $C_M$ is chosen in such a way as to match the exponents in the case (ii) of Theorem \ref{quantitative Leibman's equidistribution theorem}. We however have control over how we choose the constant $c_M$, and we shall pick it small enough to show that $g^P$ not being $(c_M A^{-C_M},N)$-equidistributed contradicts the $(A,N)$-irrationality of $g$.

Rewriting the expression for $\eta\circ g^P$, we see that
\begin{align*}
    \eta\circ g^P(x,y) &= \eta(g_1^{\vec{v}_1}) x + \eta(g_1^{\vec{v}_2}) y + \eta(g_1^{\vec{v}_3}) y^3\\
    &+ \eta(g_2^{\vec{v}_1}) {{x}\choose{2}}+ \eta(g_2^{\vec{v}_2}) \left(xy+{{y}\choose{2}}\right)  + \eta(g_2^{\vec{v}_3}) \left(xy^3 + {{y^3}\choose{2}}\right) + \vec{v}_4 y^2.
\end{align*}
Applying Lemma \ref{breaking an integral polynomial into a sum of integral polynomials} and the assumption $||\eta\circ g^P||_{C^\infty[N]}\leqslant cA$, and choosing $c_M$ in such a way that $c>0$ is sufficiently small, we deduce that there exists a positive integer $q=O(1)$ such that $||q \eta(g_i^{\vec{v}_{j}})||_{\RR/\ZZ} < A N^{-i}$ for all pairs
\begin{align*}
(i,j)\in\{(1,1), (1,2), (1,3), (2,1), (2,2), (2,3), (2,4)\}. 
\end{align*}

We aim to show that $\eta$ is trivial by showing that it vanishes on all of $G^P$. First, we want to show that $\eta$ vanishes on $G_2^4$. Suppose that $\eta|_{G_2^4}\neq 0$,  and define $\xi_{2,1}:G_2\to\RR$ by $\xi_{2,1}(h_2) =q \eta(h_2^{(1,1,1,1)})$. We claim that $\xi_{2,1}$ is a 2-nd level character. To prove this, we need to show that $\xi_{2,1}$ is a continuous group homomorphism, it vanishes on $G_3$,  it sends $(\Gamma_2)$ to $\ZZ$, and it vanishes on $[G_1,G_1]$. The first statement follows from the fact that $\eta$ is a continuous group homomorphism, the second is true since $G_3$ is trivial, and the third follows from the fact that $q\in\ZZ$, $\eta(\Gamma^P)\leqslant\ZZ$ and $(1,1,1,1)\in\ZZ^4$. To see the last statement, we note from $\vec{v}_1 \cdot \vec{v}_1 = \vec{v}_1$, the formula (C.2) in \cite{green_tao_2010a}, and the 2-step nilpotence of $G$ that for any $h_1, h'_1\in G_1$, we have
\begin{align*}
    [h_1^{\vec{v}_1}, {h_1'}^{\vec{v}_1}] = [h_1, h_1']^{\vec{v}_1}.
\end{align*}
Since $h_1^{\vec{v}_1}, {h_1'}^{\vec{v}_1}$ are both elements of $G^P$, we have
\begin{align*}
    \xi_{2,1}([h_1, h_1']) = \eta([h_1, h_1']^{\vec{v}_1}) = \eta([h_1^{\vec{v}_1}, {h_1'}^{\vec{v}_1}]) = 0,
\end{align*}
implying that $\xi_{2,1}$ vanishes on $[G_1, G_1]$. Thus, $\xi_{2,1}$ is a 2-rd level character.

Performing a similar analysis while looking at the coefficients of ${{x}\choose{2}}, xy+{{y}\choose{2}}, xy^3 + {{y^3}\choose{2}}$ and $y^2$ respectively, we conclude that for all $1\leqslant j\leqslant 4$, the maps $\xi_{2,j}(h_2) = q \eta(h_2^{\vec{v}_j})$ from $G_2$ to $\RR$ are 2-nd level characters. The nontriviality of $\eta$ on $G_2^4$ and the fact that $\vec{v}_1$, $\vec{v}_2$, $\vec{v}_3$ and $\vec{v}_4$ span $\P_2 = \RR^4$ imply that for at least one value $1\leqslant i\leqslant 4$, the character $\eta$ does not vanish on $G_2^{\vec{v}_i}$. We fix this $i$. From $||\xi_{2,i}(g_i)||_{\RR/\ZZ} = ||q \eta(g_i^{\vec{v}_{j}})||_{\RR/\ZZ} < A N^{-i}$ and the $(A,N)$-irrationality of $g_2$ we deduce that $|\xi_{2,i}|>A$. Together with the bounds $q=O(1)$ and $|\vec{v}_1|=O(1)$, this implies that $|\eta| > c' A$ for some constant $c'>0$. Choosing $c_M$ in such a way that $c<c'$ gives the desired contradiction. Hence $\eta$ vanishes on $G_2^4$.

 This leaves us with
\begin{align*}
    \eta\circ g^P(x,y) &= \eta(g_1^{\vec{v}_1})x + \eta(g_1^{\vec{v}_2})y + \eta(g_1^{\vec{v}_3})y^3.
\end{align*}
By analysing the coefficients of $x, y$ and $y^3$ as above, we see that $\eta$ vanishes on elements of the form $h_1^{\Vec{v}_i}$ with $h_1\in G_1$ and $1\leqslant i\leqslant 3$. Thus, $\eta$ vanishes on all of $G^P$. This contradicts the nontriviality of $\eta$, and so $g^P$ is $(c_M A^{-C_M},N)$-equidistributed on $G^P/\Gamma^P$. 
\end{proof}

We now prove Theorem \ref{finitary equidistribution on nilmanifolds} in full generality.

\begin{proof}[Proof of Theorem \ref{finitary equidistribution on nilmanifolds}]
Let $\vec{P}\in\RR[x,y]^{t+1}$ be an integral polynomial progression, $G_\bullet$ be a filtration of degree $s$ and $g\in\poly(\ZZ, G_\bullet)$. By (\ref{direct sum for homogeneous configurations}), we can find a family $\{Q_{i,j}:\; 1\leqslant i\leqslant s,\; 1\leqslant j\leqslant t_i\}$ of linearly independent integral polynomials such that $Q_{i,1}$, \ldots, $Q_{i, t_i}$ is a basis for $W_i = W'_i$ for $1\leqslant i\leqslant s$. It is crucial that these polynomials are linearly independent, which follows from homogeneity of $\vec{P}$. For each $i$, let $\tau_i: W_i\to\P_i$ be the map associated with $Q_{i,1}$, \ldots, $Q_{i, t_i}$ as defined in Section \ref{section on homogeneity}. We also let $\vec{v}_{i,j}\in\ZZ^{t+1}$ be the vectors such that $\tau_i(Q_{i,j}) = \vec{v}_{i,j}$. 

As in the proof of Proposition \ref{equidistribution for x, x+y, x+2y, x+y^3}, suppose that $g^P$ is not $(c_M A^{-C_M},N)$-equidistributed on $G^P/\Gamma^P$ for some constants $0<c_M<1<C_M$. We apply Theorem \ref{quantitative Leibman's equidistribution theorem} again to conclude that there exists a nontrivial horizontal character ${\eta: G^P\to\RR}$ of modulus at most $cA$, satisfying $||\eta\circ g^P||_{C^\infty[N]}\leqslant cA$ for some constant $c>0$ that depends on $c_M$ and $C_M$. The constant $C_M$ is chosen in such a way as to match the exponents in the case (ii) of Theorem \ref{quantitative Leibman's equidistribution theorem}, but the choice of $c_M$ is up to us again. We shall pick it small enough to show that the failure of $g^P$ to be $(c_M A^{-C_M},N)$-equidistributed contradicts the $(A,N)$-irrationality of $g$.

Thus,
\begin{align*}
    \eta\circ g^P(x,y) = \sum_{i=1}^s\sum_{j=1}^{t_i}\eta(g_i^{\vec{v}_{i,j}}) Q_{i,j}(x,y).
\end{align*}
Using Lemma \ref{breaking an integral polynomial into a sum of integral polynomials} and the assumption $||\eta\circ g^P||_{C^\infty[N]}\leqslant cA$, and choosing $c_M$ in such a way that $c>0$ is sufficiently small, we deduce that there exists a positive integer $q=O(1)$ such that $||q \eta(g_i^{\vec{v}_{i,j}})||_{\RR/\ZZ} < A N^{-i}$ for all $1\leqslant i\leqslant s$ and $1\leqslant j\leqslant t_i$.

Our goal now is to show by downward induction on $i$ that $\eta$ vanishes on the group $$H_i = \langle h_i^{\vec{v}_{i,j}}: h_i\in G_i, 1\leqslant j\leqslant t_i\rangle$$ for all $i\in\NN_+$. This is trivially true for $i\geqslant s+1$. Suppose that $\eta$ vanishes on $H_{i+1}$ for some $1\leqslant i\leqslant s$ but that it does not vanish on $H_i$. We define the maps $\xi_{i,j}:G_i\to\RR$ by $\xi_{i,j}(h_i) = \eta(q h_i^{\vec{v}_{i,j}})$ and claim that they are $i$-th level characters. They are continuous group homomorphisms because $\eta$ is, and they vanish on $G_{i+1}$ by induction hypothesis. Since $q\in\ZZ$ and $\vec{v}_{i,j}$ have integer entries, we also have $\xi_{i,j}(\Gamma_i)\subseteq\ZZ$. It remains to show that $\xi_{i,j}$ vanishes on $[G_l, G_{i-l}]$ for all $1\leqslant l < i$.
The fact that $\P_i\subseteq \P_l\cdot\P_{i-l}$ implies the existence of $\vec{u}_l\in\P_l$ and $\vec{u}_{i-l}\in\P_{i-l}$ for which $\vec{v}_{i,j}=\vec{u}_l\cdot\vec{u}_{i-l}$, and so we have
\begin{align*}
    [G_l^{\vec{u}_l}, G_{i-l}^{\vec{u}_{i-l}}] = [G_l, G_{i-l}]^{\vec{u}_l\cdot \vec{u}_{i-l}}\; {\rm{mod}}\; G_{i+1}^{t+1},
\end{align*}
from which it follows that $\xi_{i,j}|_{[G_l,G_{i-l}]}=0$. Therefore each $\xi_{i,j}$ is an $i$-th level character.

The nontriviality of $\eta$ on $H_i$ and the fact that $\P_i$ is spanned by the vectors $\vec{v}_{i,1}$, \ldots, $\vec{v}_{i,t_i}$ imply that for at least one value $1\leqslant j\leqslant t_i$, the character $\eta$ does not vanish on $G_i^{\vec{v}_{i,j}}$, and so $\xi_{i,j}$ is nontrivial. From $||\xi_{i,j}(g_i)||_{\RR/\ZZ} = ||q \eta(g_i^{\vec{v}_{i,j}})||_{\RR/\ZZ} < A N^{-i}$ and the $(A,N)$-irrationality of $g_i$ we deduce that $|\xi_{i,j}|>A$. Together with the bounds $q=O(1)$ and $|\vec{v}_{i,j}|=O(1)$, this implies that $|\eta| > c' A$ for some constant $c'>0$. We choose $c_M$ in such a way that $c<c'$; this contradicts the nontriviality of $\eta$ on $H_i$. This proves the inductive step; hence $\eta$ vanishes on all of $G^P$, contradicting the nontriviality of $\eta$. It follows that $g^P$ is $(c_M A^{-C_M},N)$-equidistributed on $G^P/\Gamma^P$. 
\end{proof}
%%%
%%%
%%% Failure in the inhomogeneous case
%%%
%%%

\section{The failure of Theorem \ref{finitary equidistribution on nilmanifolds} in the inhomogeneous case}\label{section on failure in the inhomogeneous case}
Having derived Theorem \ref{finitary equidistribution on nilmanifolds}, we want to show why an analogous statement fails in the inhomogeneous case. We let 
\begin{align}\label{x, x+y, x+2y, x+y^2}
    \Vec{P}(x,y) = (x, \; x+y,\; x+2y,\; x+y^2),
\end{align}
with a square instead of a cube in the last position. It is an inhomogeneous progression because of the inhomogeneous relation (\ref{algebraic relation for x, x+y, x+2y, x+y^2}). Suppose that $g\in\poly(\ZZ, G_\bullet)$ is an irrational polynomial sequence with $g(0)=1$ on a connected group $G$ with a filtration $G_\bullet$ of degree 2. We shall try to show that $g^P$ is equidistributed on $G^P/\Gamma^P$ the same way as we argued in Proposition \ref{equidistribution for x, x+y, x+2y, x+y^3}, and we indicate where and why the argument fails.

Once again, we let
\begin{align*}
\vec{v}_1 = (1,1,1,1),\quad \vec{v}_2 = (0,1,2,0),\quad \vec{v}_3 = (0,0,0,1) \quad {\rm{and}}\quad \vec{v}_4 = (0,0,1,0),
\end{align*}
and we observe that $\P_1 = \Span_\RR\{\vec{v}_1, \vec{v}_2, \vec{v}_3\}$ and $\P_2 = \Span_\RR\{\vec{v}_1, \vec{v}_2, \vec{v}_3, \vec{v}_4\}$. Hence $G^P = G^{\vec{v}_1} G^{\vec{v}_2} G^{\vec{v}_3} G_2^4.$ Suppose that $g^P$ is not $(c_M A^{-C_M},N)$-equidistributed on $G^P/\Gamma^P$ for some constants $0<c_M<1<C_M$. Theorem \ref{quantitative Leibman's equidistribution theorem} once again implies the existence of a nontrivial horizontal character ${\eta: G^P\to\RR}$ of modulus at most $cA$, for which $||\eta\circ g^P||_{C^\infty[N]}\leqslant cA$ for some constant $c>0$ that depends on $c_M$ and $C_M$. 

%The constant $C_M$ is chosen in such a way as to match the exponents in the case (ii) of Theorem \ref{quantitative Leibman's equidistribution theorem}. We however have control over how we choose the constant $c_M$, and we shall pick it small enough to show that $g^P$ not being $(c_M A^{-C_M},N)$-equidistributed contradicts the $(A,N)$-irrationality of $g$.

Rewriting the expression for $\eta\circ g^P$, we see that
\begin{align*}
    \eta\circ g^P(x,y) &= \eta(g_1^{\vec{v}_1}) x + \eta(g_1^{\vec{v}_2}) y + \eta(g_1^{\vec{v}_3}) y^2\\
    &+ \eta(g_2^{\vec{v}_1}) {{x}\choose{2}}+ \eta(g_2^{\vec{v}_2}) \left(xy+{{y}\choose{2}}\right)  + \eta(g_2^{\vec{v}_3}) \left(xy^2 + {{y^2}\choose{2}}\right) + \vec{v}_4 y^2\\
    &= \eta(g_1^{\vec{v}_1}) x + \eta(g_1^{\vec{v}_2}) y + (\eta(g_1^{\vec{v}_3}) + \eta(g_2^{\vec{v}_4}))y^2\\
    &+ \eta(g_2^{\vec{v}_1}) {{x}\choose{2}}+ \eta(g_2^{\vec{v}_2}) \left(xy+{{y}\choose{2}}\right)  + \eta(g_2^{\vec{v}_3}) \left(xy^2 + {{y^2}\choose{2}}\right).
\end{align*}

Applying Lemma \ref{breaking an integral polynomial into a sum of integral polynomials} and the assumption $||\eta\circ g^P||_{C^\infty[N]}\leqslant cA$, and choosing $c_M$ in such a way that $c>0$ is sufficiently small, we deduce that there exists a positive integer $q=O(1)$ such that
\begin{align}\label{small coefficient}
||q \eta(g_i^{\vec{v}_{j}})||_{\RR/\ZZ} < A N^{-i}
\end{align}
 for all pairs
\begin{align*}
(i,j)\in\{(1,1), (1,2), (2,1), (2,2), (2,3)\}.
\end{align*}

By looking at the coefficient of ${{x}\choose{2}}$, $xy+{{y}\choose{2}}$ and $xy^2 + {{y^2}\choose{2}}$, we deduce that the maps
\begin{align*}
    h_2\mapsto q\eta(h_2^{\vec{v}_1}),\; q\eta(h_2^{\vec{v}_2}),\; q\eta(h_2^{\vec{v}_3})
\end{align*}
are trivial 2-nd level characters; the argument goes the exact same way as in the proof of Proposition \ref{equidistribution for x, x+y, x+2y, x+y^3}. Thus, $\eta$ vanishes on all elements of the form $h_2^{\Vec{w}_2}$ with $h_2\in G_2$ and
\begin{align*}
    \vec{w}_2 \in \P'_2 = \Span_\RR\{\vec{v}_1, \vec{v}_2, \vec{v}_3\}.
\end{align*}
By looking at the coefficients of $x$ and $y$, we similarly show that $\eta$ vanishes on all elements of the form $h_1^{\Vec{w}_1}$ with $h_1\in G_1$ and
\begin{align*}
    \vec{w}_1 \in \P'_1 = \Span_\RR\{\vec{v}_1, \vec{v}_2\}.
\end{align*}
We are left with
\begin{align*}
    \eta\circ g^P(x,y) = \left(\eta(g_1^{\vec{v}_3})+ \eta(g_2^{\vec{v}_4})\right)y^2.
\end{align*}

We would like to be able to say that $\eta$ vanishes on all elements of the form $h_1^{\Vec{w}_1}$ and $h_2^{\Vec{w}_2}$ with $h_i \in G_i$ and $\Vec{w}_i \in \P_i$; this would imply that $\eta$ is trivial. For this to be case, it would suffice to show that both $\eta(g_1^{\vec{v}_3})$ and $\eta(g_2^{\vec{v}_4})$ satisfy an estimate (\ref{small coefficient}), and then use $(A,N)$-irrationality of $g_1$ and $g_2$ to conclude that the characters $h_1\mapsto q \eta(h_1^{\vec{v}_3})$ and $h_2\mapsto q\eta(h_2^{\vec{v}_4})$ are trivial. Alas, this need not be true. In Proposition \ref{equidistribution for x, x+y, x+2y, x+y^3}, the number $\eta(h_1^{\vec{v}_3})$ was the coefficient of $y^3$ while $\eta(h_2^{\vec{v}_4})$ was the coefficient of $y^2$, from which it followed that they both satisfied  (\ref{small coefficient}). Now, however, all we can show is that
\begin{align}\label{small coefficient 2}
||q (\eta(g_1^{\vec{v}_3}) + \eta(g_2^{\vec{v}_4}))||_{\RR/\ZZ} < A N^{-1}
\end{align}
 because $\eta(g_1^{\vec{v}_3}) + \eta(g_2^{\vec{v}_4})$ is the coefficient of $y^2$. But it need not follow that either of $\eta(g_1^{\vec{v}_3})$ and $\eta(g_2^{\vec{v}_4})$ satisfies (\ref{small coefficient}); in particular, $g^P$ may take values in a proper rational subgroup of $G^P$.
 
 We illustrate this with a specific example, akin to the example in Section 11 of \cite{kuca_2020b}. Suppose that $G = G_1 = \RR^2$, $G_2 = 0\times \RR$, $G_3 = 0\times 0$. The sequence $g(n) = (a n, b {{n}\choose{2}})$ is adapted to the filtration $G_\bullet$, and it is irrational if and only if $a$ and $b$ are irrational. We identify $G^4$ with $\RR^8$ via the map
\begin{align*}
G^4 &\to\RR^8\\
((x_1, y_1), (x_2, y_2), (x_3, y_3), (x_4, y_4)) &\mapsto (x_1, x_2, x_3, x_4, y_1, y_2, y_3, y_4).
\end{align*}
Setting
\begin{align*}
\vec{v}_{11}  &= \vec{e}_1 + \vec{e}_2 + \vec{e}_3 + \vec{e}_4, \quad \vec{v}_{12} = \vec{e}_2 + 2\vec{e}_3, \quad \vec{v}_{13} = \vec{e}_4,\\ 
\vec{v}_{21} &= \vec{e}_5+\vec{e}_6+\vec{e}_7+\vec{e}_8,\quad \vec{v}_{22} = \vec{e}_6 + 2\vec{e}_7,\quad \vec{v}_{23} = \vec{e}_8,\quad \vec{v}_{24} = \vec{e}_7,
\end{align*}
we observe that $G^P = \Span_\RR\{\vec{v}_{11}, \vec{v}_{12}, \vec{v}_{13}, \vec{v}_{21}, \vec{v}_{22}, \vec{v}_{23}, \vec{v}_{24}\}$.

With these definitions, the coefficient of $y^2$ in $g^P$ becomes $ a\vec{v}_{13} + b\vec{v}_{24} = a\vec{e}_4 + b\vec{e}_7$. If $a, b, 1$ are rationally independent, then the closure of $g^P$ is the image of the 7-dimensional subspace $G^P$ in $(\RR/\ZZ)^8$. If $a$ and $b$ are rationally dependent, then the closure of $g^P$ is the image in $(\RR/\ZZ)^8$ of the 6-dimensional subspace 

$$\tilde{G} = \Span_\RR\{\vec{v}_{11}, \vec{v}_{12}, a\vec{v}_{13} + b\vec{v}_{24}, \vec{v}_{21}, \vec{v}_{22}, \vec{v}_{23}\}.$$ 
Finally, if some rational linear combination of $a$ and $b$ is a rational number $q/r$ in its lower terms with $r>1$, then the closure of $g^P$ is a union of at most $r$ translates of a 6-dimensional subtorus of $G^P/\Gamma^P$. For instance, if $a = \sqrt{2}$ and $b = \sqrt{2}+\frac{1}{3}$, then we define
\begin{align}\label{6-dimensional closure, 2}
    \tilde{G} = \Span_\RR\{\vec{v}_{11}, \vec{v}_{12}, \vec{v}_{13} + \vec{v}_{24}, \vec{v}_{21}, \vec{v}_{22}, \vec{v}_{23}\},
\end{align}
and observe that the sequences $g^P_0, g^P_1, g^P_2$ defined by $g^P_i(x,y) = g^P(x, 3y + i)$ are equidistributed on $\tilde{G}/\tilde{\Gamma}$, $\frac{1}{3}\vec{v}_{24} + \tilde{G}/\tilde{\Gamma}$ and $\frac{1}{3}\vec{v}_{24} + \tilde{G}/\tilde{\Gamma}$ respectively. In particular, for inhomogeneous progressions it is not true that the group $\tilde{G}$ depends only on the filtration $G_\bullet$ and the progression $\vec{P}$.

%When dealing with inhomogeneous progressions, it is therefore not enough to know that the coefficients $g_1$, \ldots, $g_s$ are irrational; they also have to be ``rationally independent". 

While annihilating the coefficients of $\eta\circ g^P$, we were able to deal with the coefficients of $x$ and $y$ as well as ${{x}\choose{2}}$, $xy+{{y}\choose{2}}$ and $xy^2 + {{y^2}\choose{2}}$, which span the spaces $W_1'$ and $W_2'$ respectively. The problematic coefficient was that of $y^2$, belonging to the space $W^c$. We have remarked below (\ref{direct sum for inhomogeneous configurations}) in Section \ref{section on homogeneity} that the nontriviality of the subspace $W^c$ prevents us from running the same argument as in Proposition \ref{equidistribution for x, x+y, x+2y, x+y^3} and Theorem \ref{finitary equidistribution on nilmanifolds} for inhomogeneous progressions; the problem with the coefficient of $y^2$ that we have encountered here illustrates this point. The reader should see from here how to generalise the aforementioned example to other inhomogeneous progression; this generalised construction proves part (ii) of Theorem \ref{dichotomy}.
%find an example of $g$ and $G$ for an inhomogeneous progressions $\vec{P}$ such that the sequence $g^P$ is contained in a proper rational subgroup of $G^P$. 

%%%
%%%
%%% The inhomogeneous case
%%%
%%%
 
\section{Finding closure in the inhomogeneous case}\label{section on inhomogeneous case}
Section \ref{section on failure in the inhomogeneous case} shows that we cannot always hope for the sequence $g^P$ to equidistribute in $G^P/\Gamma^P$ for an inhomogeneous progression $\vec{P}$. Here, we provide an inductive recipe for finding the closure of $g^P$ in the case of $\vec{P}(x,y) = (x,\; x+y,\; x+2y,\; x+y^2)$. We believe that this argument could be generalised to an arbitrary inhomogeneous progressions; while trying to do so, however, we have encountered significant technical issues of linear algebraic nature that we have not been able to overcome. 

Since the argument that we present here is already complicated enough, we prove it in an infinitary setting so as to avoid confusion coming from various quantitative parameters. In effect, we show the following. 

\begin{proposition}\label{closure for inhomogeneous}
Let $G$ be a connected group with filtration $G_\bullet$ of degree $s$, and $\vec{P}(x,y) = (x,\; x+y,\; x+2y,\; x+y^2)$. Suppose that $g\in\poly(\ZZ, G_\bullet)$ is irrational. There exists 
a subgroup $\tilde{G}\leqslant G^P$ and a decomposition $g^P = \tilde{g}\gamma$, where $\tilde{g}$ takes values in $\tilde{G}$ and is equidistributed on $\tilde{G}/\tilde{\Gamma}$ whereas $\gamma$ is periodic. Moreover, the group $\tilde{G}$ contains the subgroup
\begin{align*}
K = \langle h_i^{\vec{w}_i}: h_i\in G_i, \vec{w}_i\in \P'_i, 1\leqslant i\leqslant s\rangle,
\end{align*}
where
\begin{align*}
\P_1' &= \Span_\RR\{(1,1,1,1), (0,1,2,0)\}, \\
\P_2' &= \Span_\RR\{(1,1,1,1), (0,1,2,0), (0,0,0,1)\},\\
\P_3' &= \P_4' = \ldots = \RR^4.
\end{align*}
\end{proposition}

We will need the following lemma, which is similar in spirit to Lemma \ref{breaking an integral polynomial into a sum of integral polynomials}.
\begin{lemma}\label{breaking an integral polynomial into a sum of integral polynomials, 2}
Let $a_1, \ldots, a_s$ be nonzero real numbers. Let $Q_1, \ldots, Q_s\in\QQ[x,y]$ be linearly independent polynomials, and suppose that $Q = a_1 Q_1 + \ldots + a_s Q_s$ takes values in $\QQ$. Then $a_i\in\QQ$ for all $1\leqslant i\leqslant s$.
\end{lemma}
\begin{proof}
 Let $b_{kl i}$ be the coefficient of ${{x}\choose{k}}{{y}\choose{l}}$ in $Q_i$. Then $$Q(x, y) = \sum_{k, l} \left(\sum_{i=1}^s a_i b_{kli}\right) {{x}\choose{k}}{{y}\choose{l}}.$$
 The coefficient
 \begin{align*}
     c_{kl} = a_1 b_{kl1} + \ldots + a_s b_{kls}
 \end{align*}
  of ${{x}\choose{k}}{{y}\choose{l}}$ in $Q$ is rational, which can be seen as follows: there exists an integer $q>0$ such that $q Q\in \ZZ[x,y]$, and hence $q c_{kl}\in\ZZ$ by the classical fact that each integral polynomial is an integral linear combination of the Taylor monomials ${{x}\choose{k}}{{y}\choose{l}}$. Indexing the pairs $(k_1, l_1), \ldots, (k_u, l_u)$ in some arbitrary fashion, we obtain an $u\times s$ matrix $B = (b_{k_r l_r i})_{r i}$ as well as an $s$-dimensional column vector $a = (a_{i})_{i}$ and a $u$-dimensional column vector $c = (c_{j_l k_l})_{l}$ such that $Ba = c$. The linear independence of $Q_1, \ldots, Q_r$ implies that $B$ has full rank, and so there exists an invertible $s\times s$ submatrix $\tilde{B}$ of $B$ and an $s$-dimensional column vector $\tilde{c}$ such that $\tilde{B}a = \tilde{c}$. Since the entries of $\tilde{B}$ are integers, the entries of $\tilde{B}^{-1}$ are rational numbers. The equality $a = \tilde{B}^{-1}\tilde{c}$ then implies that $a_i\in\QQ$ for each $1\leqslant i\leqslant s$.
\end{proof}

\begin{proof}[Proof of Proposition \ref{closure for inhomogeneous}]
For each $i\geqslant 3$, we find a basis $\{Q_{i,1}, Q_{i,2}, Q_{i,3}, Q_{i,4}\}$ for $W_i$. The absence of an inhomogeneous algebraic relation of degree 3 or higher implies that $$\sum_{i=3}^s W_i = \bigoplus_{i=3}^s W_i,$$ from which it follows that the set $\{Q_{i,j}: 3\leqslant i\leqslant s,\; 1\leqslant j\leqslant 4\}$ is linearly independent.
For $3\leqslant i\leqslant s$ and $1\leqslant j\leqslant 4$, we let $\vec{v}_{i,j}=\tau_i(Q_{i,j})$. We also set
\begin{align*}
\vec{v}_1 = (1,1,1,1),\quad \vec{v}_2 = (0,1,2,0),\quad \vec{v}_3 = (0,0,0,1) \quad {\rm{and}}\quad \vec{v}_4 = (0,0,1,0).
\end{align*}

We want to find a subgroup $\Tilde{G}$ of $G^P$ on which we can guarantee equidistribution. Starting with $$H^{(1)} = \langle h_1^{\vec{v}_3},\; h_2^{\vec{v}_4}: h_1\in G_1, h_2\in G_2\rangle,$$ we inductively define a chain of subgroups 
\begin{align*}
    H^{(1)}\geqslant H^{(2)}\geqslant H^{(3)}\geqslant \ldots
\end{align*}
as well as groups $ G^{(k)} = \langle K, H^{(k)}\rangle$
%\begin{align*}
%    G^{(k)} = \langle K, H^{(k)}\rangle
    %= \langle h_i^{\vec{v}_i}, H^{(k)}:\; h_i\in G_i, \vec{v}_i\in \P'_i, 1\leqslant i\leqslant s \rangle
%\end{align*}
and $\Gamma^{(k)}=\Gamma^P\cap G^{(k)}$. We note that $G^{(1)}=G^P$.

We also inductively define sequences $g^{(k)}$ and $h^{(k)}$, starting with $h^{(1)}(y) = {g_1^{\vec{v}_3}}^{y^2}{g_2^{\vec{v}_4}}^{y^2}$ and $g^{(1)} = g^P$. If $g^{(k)}$ is equidistributed in $G^{(k)}/\Gamma^{(k)}$, then we terminate the procedure. Otherwise Theorem \ref{Leibman's equidistribution theorem} implies the existence of a nontrivial horizontal character $\eta^{(k)}:G^{(k)}\to\RR$ that vanishes on all of $G^{(k)}$ except $H^{(k)}$, and for which $\eta^{(k)}\circ g^{(k)} =\eta^{(k)}\circ h^{(k)}$ takes values in $\ZZ$. We then take $G^{(k+1)}=\ker\eta^{(k)}$ and $H^{(k+1)}=\ker\eta^{(k)}|_{H^{(k)}}$, and we factorize $h^{(k)} = h^{(k+1)}\gamma^{(k+1)}$ using an infinitary version of Proposition 9.2 of \cite{green_tao_2012}, where $\eta^{k+1}\circ h^{(k+1)} = 0$ and $\gamma^{(k+1)}$ is periodic. We define
\begin{align*}
    g^{(k+1)}(x,y) = g^{(k)}(x,y)(\gamma^{(k+1)}(y))^{-1}
\end{align*}
and observe that 

\begin{align*}
    g^{(k+1)}(x,y) = g_1^{{\vec{v}_1}x+\vec{v}_2 y} h^{(k+1)}(y) g_2^{{\vec{v}_1}{{x}\choose{2}} + \vec{v}_2(xy+{{y}\choose{2}}) + \vec{v}_3(xy^2 +{{y^2}\choose{2}})}
    \prod_{i=3}^s \prod_{j=1}^4 g_i^{\vec{v}_{i,j} Q_{i,j}} \; {\rm{mod}} \; [G_1, G_2]^4.
\end{align*}

The sequence $g^{(k+1)}$ takes values in $G^{(k+1)}$. We also write 
\begin{align*}
h^{(k)}(y) = a^{(k)}(y)^{\vec{v}_4}b^{(k)}(y)^{\vec{v}_3},
\end{align*}
with $a^{(k)}$ being $G_2$-valued and $b^{(k)}$ being $G_1$-valued. Letting $a^{(k)}(y) = \prod\limits_{i=1}^s {a^{(k)}_i}^{{y}\choose{i}}$ and similarly for $b^{(k)}$, we claim that $a^{(k)}_2$ and $b^{(k)}_2$ are irrational elements of $G_2$ and $G_1$ respectively with regard to the filtration $G_\bullet$ on $G$.  
Finally, we claim that
\begin{align*}
    H^{(k)} = G_2^{\vec{v}_4} \mod G_1^{\vec{v}_3}\quad {\rm{and}}\quad    H^{(k)} = G_1^{\vec{v}_3} \mod G_2^{\vec{v}_4}
\end{align*}

First, we observe that all these properties hold at $k=1$. We assume that they hold for some $k\geqslant 1$, from which we aim to deduce that they also hold at $(k+1)$-th level.

If $g^{(k)}$ is equidistributed in $G^{(k)}/\Gamma^{(k)}$, then we are done. Otherwise there exists a nontrivial horizontal character $\eta^{(k)}:G^{(k)}\to\RR$ for which $\eta^{(k)}\circ g^{(k)}$ is $\ZZ$-valued. We have
\begin{align*}
    \eta^{(k)}\circ g^{(k)}(x,y) &= \eta^{(k)}(g_1^{{\vec{v}_1}})x + \eta^{(k)}(g_1^{\vec{v}_2})y + \eta^{(k)}(h^{(k)}(y)) \\
    &+ \eta^{(k)}(g_2^{{\vec{v}_1}}){{x}\choose{2}} + 2\eta^{(k)}(g_2^{\vec{v}_2}) \left(xy+{{y}\choose{2}}\right) + \eta^{(k)}(g_2^{\vec{v}_3})\left(xy^2 +{{y^2}\choose{2}}\right)\\
    &+\sum_{i=3}^k \sum_{j=1}^4 \eta^{(k)}(g_i^{\vec{v}_{i,j}}) Q_{i,j}(x,y).
\end{align*}
By looking at the coefficients of $Q_{i,j}$ for $3\leqslant i\leqslant s$, applying Lemma \ref{breaking an integral polynomial into a sum of integral polynomials, 2}, and following the same method as in the proof of Theorem \ref{finitary equidistribution on nilmanifolds}, we see that $\eta^{(k)}$ vanishes on elements of the form $h_i^{\vec{v}_{i,j}}$ for $h_i\in G_i$, $3\leqslant i\leqslant s$ and $1\leqslant j\leqslant 4$, and so $\eta^{(k)}$ vanishes on all of $G_3\times G_3\times G_3\times G_3$. This leaves us with
\begin{align*}
    \eta^{(k)}\circ g^{(k)}(x,y)
    &= \eta^{(k)}(g_1^{{\vec{v}_1}})x + \eta^{(k)}(g_1^{\vec{v}_2})y + \eta^{(k)}(h^{(k)}(y)) \\
    &+ \eta^{(k)}(g_2^{{\vec{v}_1}}){{x}\choose{2}} + 2\eta^{(k)}(g_2^{\vec{v}_2}) \left(xy+{{y}\choose{2}}\right) + \eta^{(k)}(g_2^{\vec{v}_3})\left(xy^2 +{{y^2}\choose{2}}\right).
\end{align*}
We now carry on. By looking at the coefficient of ${{x}\choose{2}}$ and $xy+{{y}\choose{2}}$, we see that $\eta^{(k)}(g_2^{{\vec{v}_1}})$ and $\eta^{(k)}(g_2^{\vec{v}_2})$ are both integers, and so $\eta^{(k)}$ vanishes on all elements of the form $h_2^{\Vec{v}_1}$ and $h_2^{\Vec{v}_2}$ with $h_2\in G_2$. By looking at the coefficients of $x$ and $y$, we can similarly show that $\eta^{(k)}$ vanishes on all elements of the form $h_1^{\Vec{v}_1}$ and $h_1^{\Vec{v}_2}$ with $h_1\in G_1$.
We are thus left with
\begin{align*}
    \eta^{(k)}\circ g^{(k)}(x,y) = \eta^{(k)}(h^{(k)}(y)) + \eta^{(k)}(g_2^{\vec{v}_3})\left(xy^2 +{{y^2}\choose{2}}\right).
\end{align*}

We first deal with the last term. Since $H^{(k)} = G_1^{\vec{v}_3}$ mod $G_2^{\vec{v}_3}$, we have $[H^{(k)}, H^{(k)}] = [G_1^{\vec{v}_3}, G_1^{\vec{v}_3}]$ mod $G_3^4$. Using the fact that $\eta^{(k)}$ vanishes on both $G_3^4$ and $[H^{(k)}, H^{(k)}]$, we deduce that it also vanishes on $[G_1^{\vec{v}_3}, G_1^{\vec{v}_3}]$. Hence the function ${\xi_{2,3}: G_2\to\RR}$ given by $\xi_{2,3}(h) = \eta^{(k)}(h^{\vec{v}_3})$ is a 2-nd level character. By irrationality of $g_2$, it follows that $\xi_{2,3}$ is trivial, and so $\eta^{(k)}$ vanishes on $G_2^{\vec{v}_3}$. We have thus proved that $\eta^{(k)}$ vanishes on all of $G^{(k)}$ except $H^{(k)}$, and consequently that $\eta^{(k)}\circ g^{(k)} = \eta^{(k)}\circ h^{(k)}$.

We now show that
\begin{align}\label{projection is as it should be}
    H^{(k+1)} = G_2^{\vec{v}_4}\! \mod G_1^{\vec{v}_3}
\end{align}
Suppose not; let $U$ be a proper rational subgroup of $G_2^{\vec{v}_4}$ such that 
\begin{align*}
    H^{(k+1)} = U\! \mod G_1^{\vec{v}_3}. 
\end{align*}
Then
\begin{align*}
    H^{(k+1)}\leqslant UG_1^{\vec{v}_3} \cap H^{(k)} \leqslant H^{(k)}.
\end{align*}
We know from the rank-nullity theorem that $\dim H^{(k+1)} = \dim H^{(k)} - 1$, and we have $H^{(k)} = G_2^{\vec{v}_4} \mod G_1^{\vec{v}_3}$ from the inductive hypothesis. These two facts, together with the assumption that $U$ is a proper rational subgroup of $G_2^{(0,0,1,0)}$, imply that $H^{(k+1)} = UG_1^{\vec{v}_3} \cap H^{(k)}$. It follows that

\begin{align*}
    \eta^{(k)}\circ g^{(k)}(x,y) = \eta^{(k)}(a^{(k)}(y)^{\vec{v}_4})+\eta^{(k)}(b^{(k)}(y)^{\vec{v}_3})= \eta^{(k)}(a^{(k)}(y)^{\vec{v}_4})
\end{align*}

We have already shown that $\eta^{(k)}$ vanishes on $G_3^4$. From the facts that $a^{(k)}(y) = \prod_{i=1}^s {a^{(k)}_i}^{{{y}\choose{i}}}$ with $a^{(k)}_i\in G_i$, we deduce that $\eta^{(k)}({a^{(k)}(y)}^{\vec{v}_4}) = \eta^{(k)}(a^{(k)}_1) y + \eta^{(k)}(a^{(k)}_2){{y}\choose{2}}$. The map $\xi_{2,4}(h_2) = \eta^{(k)}(h_2^{\vec{v}_4})$ is a continuous group homomorphism on $G_2$ that vanishes on $G_3$ and sends $\Gamma_2$ to $\ZZ$. Since $\vec{v}_4 = (\vec{v}_2\cdot\vec{v}_2 - \vec{v}_2)/2$, we also have
\begin{align*}
\xi_{2,4}([h_1, h_1']) = \frac{1}{2}\eta^{(k)}([h_1^{\vec{v}_2},{h'_1}^{\vec{v}_2}]) - \frac{1}{2}\eta^{(k)}([h_1, h'_1]^{\vec{v}_2}),
\end{align*}
for any $h_1, h'_1\in G_1$, and so $\xi_{2,4}$ vanishes on $[G_1, G_1]$. Thus $\xi_{2,4}$ is a 2-nd level character on $G_2$ with respect to the filtration $G_\bullet$ on $G$, and since $a^{(k)}_2$ is an irrational element of $G_2$ with respect to this filtration, it follows that $\eta^{(k)}$ is trivial, a contradiction; hence (\ref{projection is as it should be}) holds. The argument that 
\begin{align*}
    H^{(k+1)} = G_1^{\vec{v}_3} \mod G_2^{\vec{v}_4}
\end{align*}
is similar.

Finally, we factorize $h^{(k)} = h^{(k+1)}\gamma^{(k+1)}$, where $\gamma^{(k+1)}$ is periodic and $h^{(k+1)}$ takes values in $H^{(k+1)} = \ker \eta^{(k+1)}$. It remains to show that $a^{(k+1)}_2$ and $b^{(k+1)}_2$ are irrational elements of $G_2$ and $G_1$ with respect to the filtration $G_\bullet$ on $G$. We observe that
\begin{align*}
    a^{(k)}  = a^{(k+1)} \gamma_a^{(k+1)} \quad {\rm{and}} \quad b^{(k)}  = b^{(k+1)} \gamma_b^{(k+1)}
\end{align*}
for some periodic sequences $\gamma_a$ and $\gamma_b$ taking values in $G_2$ and $G_1$ respectively.
%\footnote{The first of these statement follows from the fact that both $a^{(k+1)}$ and $\gamma_a^{(k+1)}$ take values in $G_2$, hence $(a^{(k)})^{\vec{v}_4}  = (a^{(k+1)})^{\vec{v}_4} (\gamma_a^{(k+1)})^{\vec{v}_4}=(a^{(k+1)}\gamma_a^{(k+1)})^{\vec{v}_4}$ mod $G_5^4$ (the last equality is a consequence of the Baker-Campbell-Hausdorff formula). The second statement follows similarly except that $b^{(k+1)}$ and $\gamma_b^{(k+1)}$ take values in $G_1$ as opposed to $G_2$, which is why we have mod $G_3$ instead of mod $G_5$.}. 
Suppose that $\xi: G_2\to\RR$ is a 2-nd level character with respect to the filtration $G_\bullet$, for which $\xi(a^{(k+1)}_2)\in\ZZ$. The sequence $\gamma_a^{(k+1)}$ is periodic, hence $\xi\circ\gamma_a^{(k+1)}$ is $\QQ$-valued, and so it follows that $\xi(a^{(k)}_2)\in\QQ$ as well. Therefore there exists an integer $l>0$ such that $l\xi(a^{(k)}_2)\in\ZZ$. Since $\xi':= l\cdot \xi$ is also a 2-nd level character, it follows from the irrationality of $a^{(k)}_2$ that $\xi'$ is trivial. This implies that $\xi$ is trivial as well, and hence $a^{(k+1)}_2$ is irrational. The argument showing that $b^{(k+1)}_2$ is irrational is identical. 

We have thus shows inductively that  $g^{(k)}$, $h^{(k)}$, $G^{(k)}$ and $H^{(k)}$ satisfy all the properties we want them to satisfy for all $k\geqslant 1$. Since $0\leqslant\dim G^{(k+1)}<\dim G^{(k)}$, the procedure eventually terminates, at which point the sequence $g^{(k)}$ takes values in $G^{(k)}$ and is equidistributed on $G^{(k)}/\Gamma^{(k)}$. Letting $\tilde{G} = G^{(k)}$ for this value of $k$ and $\gamma = \gamma^{(k)} \ldots \gamma^{(1)}$, and observing that a product of periodic sequences is periodic, we finish the proof.
\end{proof}

%%%
%%%
%%% Weyl complexity equals algebraic complexity
%%%
%%%

\section{The equivalence of Weyl and algebraic complexity}\label{section on Weyl complexity}
While we are not able to show that Host-Kra and true complexities equal algebraic complexity for inhomogeneous progression, we can show the equivalence of Weyl and algebraic complexities for all integral progressions. 

\begin{definition}[Weyl system]\label{Weyl system}
A \emph{Weyl system} is an ergodic system $(X, \X, \mu, T)$, where $X$ is a compact abelian Lie group and $T$ is a unipotent affine transformation on $X$, i.e. $Tx = \phi(x) + a$ for $a\in X$ and an automorphism $\phi$ of $X$ satisfying $(\phi - \rm{Id}_X)^s = 0$ for some $s\in\NN_+$.
\end{definition}

We recall that an integral polynomial progression $\vec{P}\in\RR[x,y]^{t+1}$ has Weyl complexity $s$ at $0\leqslant i\leqslant t$ if $s$ the smallest natural number for which the factor $\Z_s$ is characteristic for the weak convergence of $\vec{P}$ at $i$ for any Weyl system. 

Every disconnected Weyl system can be written as a finite union of isomorphic tori that are cyclically permuted by the transformation $T$, much the same way as each disconnected nilsystem is a union of connected nilsystems (cf. Proposition \ref{totally ergodic nilsystems} and the remark below Theorem 3.5 of \cite{bergelson_leibman_lesigne_2007}). Therefore we can restrict our attention to connected Weyl systems. These can in turn be reduced to standard Weyl systems, which are totally ergodic by Proposition \ref{totally ergodic nilsystems}. Throughout this section, we let $\TT = \RR/\ZZ$.
\begin{definition}[Standard Weyl system of order $s$]\label{standard Weyl system}
Let $s\in\NN_+$ and $X = \TT^s$. A \emph{standard Weyl system of order $s$} is a system $(X,\X,\mu,T)$, where $\X$ is the Borel $\sigma$-algebra on $X$, $\mu$ is the Lebesgue measure, and
\begin{align*}
T(a_1, \ldots, a_s) = (a_1 + a_0, a_2 + a_1, \ldots, a_s + a_{s-1})
\end{align*}
for some irrational $a_0$.
\end{definition}

\begin{proposition}[Lemma 4.1 of \cite{frantzikinakis_kra_2005}]
Each connected Weyl system is a factor of a product of several standard Weyl systems.
\end{proposition}

Determining Weyl complexity therefore amounts to analysing standard Weyl systems. Since each standard Weyl system is totally ergodic, we immediately deduce the following.

\begin{proposition}\label{Weyl complexity smaller than Host-Kra complexity}
Let $t\in\NN_+$ and  $\vec{P}\in\RR[x,y]^{t+1}$ be an integral polynomial progression. Then ${\W_i(\vec{P})\leqslant\HK_i(\vec{P})}$ for all $0\leqslant i\leqslant t$.
\end{proposition}

We now fix a standard Weyl system $(X,\X,\mu,T)$ of order $s$ with some irrational $a_0$. Then
\begin{align}\label{Weyl map}
T^n(a_1, \ldots, a_s) &= \left(a_1 + n a_0, a_2 + n a_1 + {{n}\choose{2}}a_0, \ldots, a_s + n a_{s-1} + \ldots + {{n}\choose{s}}a_0\right)\\
\nonumber &= g_0 + g_1 n + \ldots + g_s{{n}\choose{s}},
\end{align}
where $g_i = (a_{1-i}, \ldots, a_{s-i})$ and $a_{-k}=0$ for $k>0$. For almost all points $a=(a_1, \ldots, a_s)\in\RR^s$, the numbers $1, a_0, \ldots, a_s$ are rationally independent, and we fix a point $a\in\RR^s$ for which this is the case. The sequence $g(n) = T^n a$ is adapted to the filtration $G_i = \{0\}^{i-1} \times \RR^{s-i+1}$ for $1\leqslant i\leqslant s$ and $G_i = 0$ for $i>s$ on $G = G_0 = \RR^s$, and it is irrational due to the irrationality of $a_0$. Since the $\Z_i$ factor of $X$ consists of all the functions whose values depend only on the first $i$ coordinates, we have $Z_i = G/G_{i+1}\Gamma = \TT^i \times\{0\}^{s-i}$, where $\Gamma = \ZZ^s$.

What we aim to show is therefore the following.
\begin{proposition}\label{Weyl complexity equals algebraic complexity}
Let $t\in\NN_+$, $(X, \X, \mu, T)$ be a standard Weyl system of order $s$ and $\vec{P}\in\RR[x,y]^{t+1}$ be an integral polynomial progression. Fix $0\leqslant i\leqslant t$ and suppose that $\A_i(\vec{P}) = s'$. Then the image of the group $\{0\}^{i}\times G_{s'+1}\times \{0\}^{t-i}$ is contained in the closure of $g^P$ inside $(G/\Gamma)^{t+1}$.
\end{proposition}

If $\vec{P}\in\RR[x,y]^{t+1}$ is a homogeneous progression, then the sequence $g^P$ is equidistributed in $G^P/\Gamma^P$ by Theorem \ref{infinitary equidistribution on nilmanifolds}, and Proposition \ref{Weyl complexity equals algebraic complexity} follows immediately; we want to say something about the closure of $g^P$ in the general case. We fix an integral progression $\vec{P}$ for the rest of this section. For each $1\leqslant i\leqslant s$, we pick linearly independent integral polynomials $Q_{i,1}, \ldots, Q_{i, t'_i}$ that form a basis for $W'_i$. We also let $\{R_1, \ldots, R_r\}$ be a basis for $W^c$ consisting of integral polynomials. Thus,
\begin{align*}
{{\vec{P}}\choose{i}} = \sum_{j=1}^{t'_i} \vec{v}_{i,j} Q_{i,j} + \sum_{j=1}^r \vec{w}_{i,j} R_j
\end{align*}
for some vectors $\vec{v}_{i,j}, \vec{w}_{i,j}\in\ZZ^{t+1}$, which follows from (\ref{direct sum for inhomogeneous configurations}).
Consequently,
\begin{align}\label{g^P in Weyl systems}
g^P = g_0 \vec{1} + \sum_{i=1}^s g_i \sum_{j=1}^{t'_i} \vec{v}_{i,j} Q_{i,j} + \sum_{j=1}^r \left(\sum_{i=1}^s g_i \vec{w}_{i,j}\right) R_j
\end{align}

We should explain the notation used in (\ref{g^P in Weyl systems}). For $h\in G$ and $\vec{v}\in\RR^{t+1}$, we interpret $h \vec{v}$ as the element of $(\RR^s)^{t+1}$ of the form $(h {v}(0), \ldots, h{v}(t))$, where $h {v}(i) = (h_1 {v}(i), \ldots, h_s {v}(i))$ is an element of $\RR^s$ for each $h=(h_1, \ldots, h_s)\in\RR^s$ and $\vec{v} = (v(0), \ldots, v(t))$. Thus, $h \vec{v}$ is the same as what we previously called $h^{\vec{v}}$. We use the additive notation $h \vec{v}$ now since we are working in an abelian setting. We also denote $\vec{1} = (1, \ldots, 1)$.

We let $A_{i,j} = \Span_\RR\{ g_i \vec{v}_{i,j} \}$ and $B_{j} = \Span_\RR\{\sum_{i=1}^s g_i \vec{w}_{i,j}\}$, and we denote the closure of their images in $(G/\Gamma)^{t+1}$ as $\overline{A}_{i,j}$ and $\overline{B}_{j}$ respectively. From the rational independence of $a_i$ and the rationality of the entries of $\vec{v}_{i,j}$ and $\vec{w}_{i,j}$, we deduce that nonzero entries of $g_i \vec{v}_{i,j}$ and $\sum_{i=1}^s g_i \vec{w}_{i,j}$ are irrational; therefore the sequences $(x,y)\mapsto g_i \vec{v}_{i,j} Q_{i,j}(x,y)$ and $(x,y)\mapsto \sum_{i=1}^s g_i \vec{w}_{i,j} R_j(x,y)$ are equidistributed on $\overline{A}_{i,j}$ and $\overline{B}_{j}$ respectively. The linear independence of $Q_{i,j}, R_j$ then implies the following.

%We let $V_{i,j} = \overline{\langle g_i \vec{v}_{i,j} \rangle}$ and $W_{j} = \overline{\langle \sum_{i=1}^s g_i \vec{w}_{i,j} \rangle}$ be subgroups of $G^{t+1}$, and we denote the closures of their images in $(G/\Gamma)^{t+1}$ as $\overline{V}_{i,j}$ and $\overline{W}_{j}$ respectively. From the rational independence of $a_i$ and the rationality of the entries of $\vec{v}_{i,j}$ and $\vec{w}_{i,j}$ we deduce that nonzero entries of $g_i \vec{v}_{i,j}$ and $\sum_{i=1}^s g_i \vec{w}_{i,j}$ are irrational; therefore $\overline{V}_{i,j}$ and $\overline{W}_{j}$ are closed, connected subgroups of $(G/\Gamma)^{t+1}$, and the sequences $A_{i,j}(x,y)\mapsto g_i \vec{v}_{i,j} Q_{i,j}(x,y)$ and $B_{i,j}(x,y)\mapsto \sum_{i=1}^s g_i \vec{w}_{i,j} R_j(x,y)$ are equidistributed on them. The linear independence of $Q_{i,j}, R_j$ then implies the following.

\begin{proposition}\label{closure of orbits on Weyl systems}
The closure of $g^P$ is the image of $g_0 \vec{1} + \tilde{G}$ inside $(G/\Gamma)^{t+1}$, where $$\tilde{G} = \sum_{i=1}^s \sum_{j=1}^{t'_i}  {A}_{i,j} + \sum_{j=1}^r {B}_j.$$ In particular, the group $\tilde{G}$ contains
\begin{align*}
K =  \sum_{i=1}^s \sum_{j=1}^{t'_i}  {A}_{i,j} = \Span_\RR\{h_i \vec{v}_{i,j}: h_i\in G_i, 1\leqslant i\leqslant s, 1\leqslant j\leqslant t'_i\}. 
\end{align*}
\end{proposition}
We observe that $K = \tilde{G} = G^P$ whenever $\vec{P}$ is homogeneous. 

\begin{corollary}
Fix $0\leqslant i\leqslant t$ and let $\A_i(\vec{P}) < s$. For $k\leqslant s$, we have $\{0\}^{i}\times G_{k}\times \{0\}^{t-i}\leqslant K$ if and only if $k > \A_i(\vec{P})$. 
\end{corollary}

\begin{proof}
For each $1\leqslant k\leqslant s$, we let $\P'_k = \Span_\RR\{\vec{v}_{k,1}, \ldots, \vec{v}_{k, t'_k}\}$. Thus
\begin{align*}
K = \Span_\RR\{h_k \vec{u}_k:\; 1\leqslant k \leqslant s,\; h_k\in G_k,\; \vec{u}_k\in\P'_k\},
\end{align*}
and so for $k\leqslant s$, we have the inclusion $\{0\}^{i}\times G_{k}\times \{0\}^{t-i}\leqslant K$ if and only if the vector $\vec{e}_i$ with 1 in the $i$-th position and 0 elsewhere is contained in $\P'_k$. The statement $\vec{e}_i\in\P'_k$ is equivalent to the inclusion ${{x+P_i(y)}\choose{k}}\in W'_k$. This is in turn equivalent to the statement that there are no algebraic relations of the form (\ref{algebraic relation}) with $\deg Q_i = k$, which is precisely the condition that $k> \A_i(\vec{P})$.
\end{proof}

\begin{corollary}
Let $t\in\NN_+$ and $\vec{P}\in\RR[x,y]^{t+1}$ be an integral polynomial progression. Then $\W_i(\vec{P})\leqslant\A_i(\vec{P})$ for each $0\leqslant i\leqslant t$.
\end{corollary}

We finish this section by showing the converse.
\begin{proposition}\label{Weyl complexity no smaller than algebraic complexity}
Let $t\in\NN_+$ and $\vec{P}\in\RR[x,y]^{t+1}$ be an integral polynomial progression for which $\A_i(\vec{P}) = s$ for some $0\leqslant i\leqslant t$. Then for any standard Weyl system $(X,\X,\mu,T)$  of order $s$ there exist smooth functions $f_0, \ldots, f_t: X\to\CC$ such that $\EE(f_i|\Z_{s-1}) = 0$ but the expression (\ref{weak limit in the proof}) is 1. In particular, $\W_i(\vec{P})\geqslant s$.
\end{proposition}

Before we prove Proposition \ref{Weyl complexity no smaller than algebraic complexity}, we define $\partial Q(x) = Q(x+1) - Q(x)$ for $Q\in\RR[x]$. From the identity $\partial{{x}\choose{k}} = {{x+1}\choose{k}} - {{x}\choose{k}} = {{x}\choose{k-1}}$ we deduce that
\begin{align*}
\partial \left(a_0 + a_1 {{x}\choose{1}} + \ldots + a_d {{x}\choose{d}}\right) = a_1 + a_2 {{x}\choose{1}} + \ldots + a_d {{x}\choose{d-1}}.
\end{align*}

\begin{proof}[Proof of Proposition \ref{Weyl complexity no smaller than algebraic complexity}]
Let $T$ be as in (\ref{Weyl map}) for some irrational $a_0$. From  $\A_i(\vec{P}) = s$ it follows that $\vec{P}$ satisfies an algebraic relation (\ref{algebraic relation}) with $\deg Q_i = s$. For each $0\leqslant k\leqslant t$, we let $Q_k(u) = b_{k,1} u + \ldots + b_{k,s}{{u}\choose{s}}$. We define $\xi(u) = e(\alpha u)$ for some irrational $\alpha$, and we let $$f_k(a_1, \ldots, a_s) = \xi(b_{k,1} a_1 + \ldots + b_{k,s} a_s).$$
Thus, we have
\begin{align*}
f_k(T^{x+P_k(y)}a) = \xi(a_0 Q_k(x+P_k(y)) + a_1 \partial Q_k(x+P_k(y)) + \ldots + a_s \partial^s Q_k(x+P_k(y))),
\end{align*}
and so
\begin{align*}
\prod_{i=0}^t f_i(T^{x+P_i(y)} a) = \xi \left(\sum_{j=0}^s a_j \partial^j \sum_{k=0}^t Q_k(x+P_k(y)) \right) = 1.
\end{align*}
On the other hand, we have $$|\EE(f_i|\Z_{s-1})(a_1, \ldots, a_s)| = \left|\int_\TT f_i(a_1, \ldots, a_s) da_s \right| = \left|\int_\TT \xi(b_{i,s} a_s) da_s \right| = 0$$ for a.e. $a_s$. 
\end{proof}

%%%
%%%
%%% Multiple recurrence
%%%
%%%

\section{The proof of Theorem \ref{popular common differences}}\label{section on multiple recurrence}

We conclude the paper with the proof of Theorem \ref{popular common differences}. Throughout this section, we let $t\in\NN_+$ and $\vec{P}\in\RR[x,y]^{t+1}$ be an integral progression of algebraic complexity at most 1. We also let $Q_1, \ldots, Q_k$ be integral polynomials as in the statement of Theorem \ref{popular common differences}. Thus, $P_i = \sum_j a_{ij} Q_j$ and $Q_i = \sum_j a'_{ij} P_j$ for $a_{ij}, a'_{ij}\in\ZZ$. The second part of the theorem follows from the first part and the Furstenberg Correspondence Principle. We therefore proceed to prove part (i), followed by part (iii). Our argument for part (i) follows closely the proof of Theorem C of \cite{frantzikinakis_2008}.

%%%
%%% New proof
%%% 

\begin{proof}[Proof of Theorem \ref{popular common differences}(i)]
We first prove part (i) of Theorem \ref{popular common differences} in the totally ergodic case. Suppose that $(X, \X, \mu, T)$ is a totally ergodic system with the Kronecker factor $(Z_1, \Z_1, \nu, S)$. The space $Z_1$ can be assumed to be a connected compact abelian group with an ergodic translation $Sx = x+b$. For each $\delta>0$, let $B_\delta$ be the $\delta$-neighbourhood of the identity in $Z_1$, and let 
\begin{align*}
    \tilde{B}_\delta = \{n\in\NN: Q_1(n) b, \ldots, Q_{k}(n) b\in B_\delta\}.
\end{align*}
It follows from the ergodicity of $S$ and linear independence of $Q_1, \ldots, Q_k$ that $$\lim_{N-M\to\infty}\frac{|\tilde{B}_\delta\cap [M,N)|}{N-M} = \nu(B_\delta)^k>0$$ for any $\delta>0$. In particular, $\tilde{B}_\delta$ is syndetic for any $\delta>0$, otherwise we would have $\liminf_{N-M\to\infty}\frac{|\tilde{B}_\delta\cap [M,N)|}{N-M} = 0$.

We aim to show that for any $A\in\X$ with $\mu(A)>0$ and any $\epsilon>0$, we have
\begin{align}\label{bound in multiple recurrence for B_delta}
    \lim_{N-M\to\infty}\EE_{n\in\tilde{B}_\delta\cap [M,N)} \mu(A\cap T^{P_1(n)}A \cap \cdots \cap T^{P_t(n)}A) \geqslant \mu(A)^{t+1}-\epsilon
\end{align}
for all sufficiently small $\delta > 0$. This implies part (i) of Theorem \ref{popular common differences} as follows: if there is a sequence $K_N$ of intervals in $\NN$ of length converging to infinity, with the property that
\begin{align}\label{failed lower bound for multiple recurrence}
    \mu(A\cap T^{P_1(n)}A \cap \cdots \cap T^{P_t(n)}A) < \mu(A)^{t+1}-\epsilon
\end{align}
 for all $n\in\bigcup_{N\in\NN}K_N$, then the sets $\tilde{K}_N = K_N\cap \tilde{B}_\delta$ are nonempty for all sufficiently large $N$ due to the syndecticity of $B_\delta$ (in fact, their cardinalities also converge to infinity). Since (\ref{failed lower bound for multiple recurrence}) holds for all $n\in\bigcup_{N\in\NN}\tilde{K}_N$, the inequality (\ref{bound in multiple recurrence for B_delta}) fails, leading to a contradiction. 

We first show that if $\EE(f_i|\Z_1) = 0$, then
\begin{align}\label{sum in multiple recurrence 0}
    \lim_{N-M\to\infty}\EE_{n\in[M,N)} 1_{\tilde{B}_\delta}(n) \prod_{i=1}^t T^{P_i(n)}f_i = 0
\end{align}
in $L^2$ for any $f_1, \ldots, f_t\in L^\infty(\mu)$. From the measurability of $B_\delta$ it follows that we can approximate $1_{\tilde{B}_\delta}(n) = \prod_{i=1}^k 1_{{B}_\delta}(Q_i(n)b)$ arbitrarily well by linear combinations of $\prod_{i=1}^k \xi_i(Q_i(n)b)$  for some characters $\xi_1, \ldots, \xi_k$ on $Z_1$. Using the fact that each $Q_i$ is an integral linear combination of $P_1, \ldots, P_t$, we can rewrite $\prod_{i=1}^k \xi_i(Q_i(n)b) = \prod_{i=1}^t \tilde{\xi}_i(P_i(n)b)$ for some characters $\tilde{\xi}_1, \ldots, \tilde{\xi}_t$.

In effect, it suffices to show that
\begin{align}\label{sum in multiple recurrence 1a}
    \lim_{N-M\to\infty}\EE_{n\in[M,N)} \prod_{i=1}^t \tilde{\xi}_i(P_i(n)b)\prod_{i=1}^t T^{P_i(n)}f_i = 0.
\end{align}
We can rephrase the limit in (\ref{sum in multiple recurrence 1a}) as
\begin{align}\label{sum in multiple recurrence 1}
    \lim_{N-M\to\infty} \prod_{i=1}^t \tilde{\xi}_i(-y) \EE_{n\in[M,N)} \prod_{i=1}^t R^{P_i(n)}(f_i(x)\tilde{\xi}_i(y)),
\end{align}
where $R = T\times S$. Let $(R_t)_t$ be the ergodic components of $R$ and $(f_i \otimes \xi_i) (x,y) = f_i(x)\xi_i(y)$; then $\EE(f_i \otimes \xi_i|\Z_1(R_t)) = 0$ whenever $\EE(f_i|\Z_1(T)) = 0$ for a.e. $t$. It thus follows from Corollary \ref{factors for ergodic systems} that if $\EE(f_i|\Z_1) = 0$ for some $i$, then the limit in (\ref{sum in multiple recurrence 1}) is 0, which proves the claim. 

We therefore deduce that
\begin{align}\label{sum in multiple recurrence 2}
    \nonumber\lim_{N-M\to\infty}\EE_{n\in\tilde{B}_\delta\cap[M,N)}\int_X \prod_{i=0}^t T^{P_i(n)}1_A d\mu
    &= \lim_{N-M\to\infty}\EE_{n\in\tilde{B}_\delta\cap[M,N)} \int_{Z_1} \prod_{i=0}^t S^{P_i(n)}\tilde{1}_A d\nu\\
    &= \lim_{N-M\to\infty}\EE_{n\in\tilde{B}_\delta\cap[M,N)} \int_{Z_1} \prod_{i=0}^t S^{\sum_j a_{ij} Q_j(n)}\tilde{1}_A d\nu,
\end{align}
where $\tilde{1}_A = \EE(1_A|\Z_1)$. 
Due to the ergodicity of $S$ and the linear independence of $Q_1, \ldots, Q_k$, the limit in (\ref{sum in multiple recurrence 2}) equals
\begin{align}\label{sum in multiple recurrence 3}
    \frac{1}{\nu({B}_\delta)^k}\int_{B_\delta^k}\int_{Z_1} \prod_{i=0}^t \tilde{1}_A(x+\sum_j a_{ij} y_j) d\nu(x) d\nu^k(y).
\end{align}
In the limit $\delta\to 0$, the expression in (\ref{sum in multiple recurrence 3}) converges to $\int_{Z_1}(\tilde{1}_A)^{t+1}$; hence for every $\epsilon>0$ and sufficiently small $\delta>0$, we have
\begin{align}\label{sum in multiple recurrence 4}
    \frac{1}{\nu({B}_\delta)^k}\int_{B_\delta^k}\int_{Z_1} \prod_{i=0}^t \tilde{1}_A(x+\sum_j a_{ij} y_j) d\nu(x) d\nu^k(y)\geqslant \int_{Z_1}(\tilde{1}_A)^{t+1} - \epsilon.
\end{align}
Using the H\"{o}lder inequality, we obtain that $\int_{Z_1}(\tilde{1}_A)^{t+1}\geqslant(\int_{Z_1}\tilde{1}_A)^{t+1}=\mu(A)^{t+1}$, which implies (\ref{bound in multiple recurrence for B_delta}). This finished the totally ergodic case; the derivation of the ergodic case from the totally ergodic one proceeds in the same way as in the proof of Theorem C of \cite{frantzikinakis_2008}.

\end{proof}

We now proceed to the proof of part (iii) of Theorem \ref{popular common differences}. The argument can below can be seen as a finitary version of the argument above, with all the necessary modifications coming from working in the finitary setting. It follows the proof of the 3-term arithmetic progression case in Theorem 1.12 of \cite{green_tao_2010a}.

\begin{proof}[Proof of Theorem \ref{popular common differences}(iii)]
Let $\alpha, \epsilon > 0$, and suppose that $A\subset\ZZ/N\ZZ$ has size $|A|\geqslant\alpha N$ for a prime $N>N_0(\alpha,\epsilon)$. Let $\F:\RR_+\to\RR_+$ be a growth function to be specified later. By Theorem 5.1 of \cite{candela_sisask_2012}, the irrational and periodic version of the celebrated arithmetic regularity lemma of Green and Tao (Theorem 1.2 of \cite{green_tao_2010a}), there exists a positive number $M = O_{\epsilon, \F}(1)$ and a decomposition
\begin{align}\label{ARL decomposition}
    1_A = f_{nil} + f_{sml} + f_{unf}
\end{align}
into 1-bounded functions such that
\begin{enumerate}
    \item $f_{nil} = F(g(n)\Gamma)$ is an $\F(M)$-irrational, $N$-periodic nilsequence of degree 1 and complexity $M$;
    \item $\norm{f_{sml}}_1\leqslant\epsilon$;
    \item $\norm{f_{unf}}_{U^2}\leqslant\frac{1}{\F(M)}$.
\end{enumerate}
Moreover, $f_{nil}$ takes values in $[0,1]$.
Unpacking the definition of $f_{nil}$, we see that $F:(\RR/\ZZ)^m\to [0,1]$ is $M$-Lipschitz, $1\leqslant m\leqslant M$, and $g(n) = b n$ for some $\F(M)$-irrational element $b\in (\frac{1}{N}\ZZ/\ZZ)^m$.

Our strategy is as follows. We shall define a weight $\tilde{\mu}:\ZZ/N\ZZ\to\RR_{\geqslant 0}$ which satisfies
\begin{align}\label{estimate for mu}
\EE_{y\in \ZZ/N\ZZ} \tilde{\mu}(y) = 1 + O(\epsilon)    
\end{align}
and
\begin{align}\label{equation in popular common differences, 1}
\EE_{x,y\in\ZZ/N\ZZ} \tilde{\mu}(y)\prod_{i=0}^t 1_A(x + P_i (y)) \geqslant \alpha^{t+1} - O(\epsilon).
\end{align}
Using the pigeonhole principle and (\ref{estimate for mu}), it can be deduced from (\ref{equation in popular common differences, 1}) that for $\Omega_{\alpha,\epsilon}(N)$ values of $y$, we have 
\begin{align*}
\EE_{x\in\ZZ/N\ZZ}\prod_{i=0}^t 1_A(x + P_i (y)) \geqslant \alpha^{t+1} - O(\epsilon),
\end{align*}
which proves part (iii) of Theorem \ref{popular common differences}.

We shall prove (\ref{equation in popular common differences, 1}) by splitting each $1_A$ using (\ref{ARL decomposition}) and showing that terms involving $f_{sml}$ or $f_{unf}$ have contributions at most $O(\epsilon)$ while the term
\begin{align}\label{equation in popular common differences, 2}
\EE_{x,y\in\ZZ/N\ZZ} \tilde{\mu}(y)\prod_{i=0}^t f_{nil}(x + P_i (y))
\end{align}
has size at least $\alpha^{t+1}-O(\epsilon)$. Showing that the terms involving $f_{sml}$ or $f_{unf}$ make negligible contributions to (\ref{equation in popular common differences, 1}) is akin to showing (\ref{sum in multiple recurrence 0}) for all functions with $\EE(f_i|\Z_1) = 0$ in the proof of part (i) of Theorem \ref{popular common differences}. In doing so, we shall use the idea that while we fix $\epsilon>0$, we have control over how fast we choose $\F$ to grow - and we choose it to grow fast enough depending on $\alpha$ and $\epsilon$ to ensure that all the estimates work.

Let $\delta>0$ be fixed later. 
%We once again let $B_\delta$ be the $\delta$-neighbourhood of 0 in $G/\Gamma = \RR^m/\ZZ^m$, and we let $\tilde{B}_{\delta} = \{n\in\ZZ/N\ZZ: Q_1(n) b, \ldots, Q_k(n) b \in B_\delta\}$. 
We define $\psi:(\RR/\ZZ)^m\to\RR_+$ to be a nonnegative, 1-bounded, $O_M(\delta^{-1})$-Lipschitz function that is 1 on $[-\frac{1}{4}\delta, \frac{1}{4}\delta]^m$ and 0 outside $[-\frac{1}{2}\delta, \frac{1}{2}\delta]^m$. We let $c = \int_{(\RR/\ZZ)^m}\psi$; thus $(\frac{1}{2}\delta)^m\leqslant c\leqslant \delta^m$. We then let $\mu(y) = \frac{\psi(b y)}{c}$. Since $b$ can be picked without the loss of generality from $[0,1]^m$, the function $\mu$ is $O_M(\delta^{-M-1})$-Lipschitz. 

We let $\tilde{\mu}(y) = \mu(Q_1(y))\ldots\mu(Q_{k}(y))$. It is a weight that picks out all the values $y$ for which $Q_1(y) b$, ..., $Q_k(y) b$ are close to being an integer, and it plays a similar role as the function $1_{\tilde{B}_\delta}$ in the proof of part (i) of Theorem \ref{popular common differences}, except that it is constructed using a Lipschitz function rather than an indicator function. To show (\ref{estimate for mu}), we observe that
\begin{align}\label{integral of the weight}
    \EE_{y\in \ZZ/N\ZZ}\tilde{\mu}(y) = \frac{1}{c^{k}}\EE_{y\in[N]}\prod_{i=1}^{k}\psi(b Q_i(y)).
\end{align}
Using the $\F(M)$-irrationality of $g$, linear independence of $Q_1$, ..., $Q_{k}$ as well as Theorem \ref{Leibman's equidistribution theorem}, we deduce that (\ref{integral of the weight}) equals 
\begin{align*}
    \frac{1}{c^{k}}\left(\left(\int \psi\right)^{k} + O_M(\delta^{-1} \F(M)^{-c_M})\right) = 1 + O_M(\delta^{-M-2} \F(M)^{-c_M})
\end{align*}
for some $c_M>0$. The estimate (\ref{estimate for mu}) follows from choosing $\F$ growing fast enough depending on $\delta$ and picking $\delta = c'_M \epsilon$ for an appropriately chosen $c'_M>0$.

We decompose each $1_A$ in (\ref{equation in popular common differences, 1}) using (\ref{ARL decomposition}) and split (\ref{equation in popular common differences, 1}) into $3^t$ terms accordingly using multilinearity. We first estimate (\ref{equation in popular common differences, 2}), and subsequently we bound contributions of $f_{sml}$ and $f_{unf}$.

Taking $\F$ growing fast enough, we assume that $\norm{f_{unf}}_{U^2}\leqslant\epsilon$, and thus $|\EE_{x\in\ZZ/N\ZZ} f_{unf}(x)|=\norm{f_{unf}}_{U^1}\leqslant\norm{f_{unf}}_{U^2}\leqslant \epsilon$. From the H\"{o}lder inequality and the bound on the $L^1$ norm of $f_{sml}$, we obtain a bound $|\EE_{x\in\ZZ/N\ZZ} f_{sml}|\leqslant\epsilon$. From these bounds and (\ref{ARL decomposition}) we deduce that $\EE_{x\in\ZZ/N\ZZ} f_{nil}(x) \geqslant \alpha - 2\epsilon$. 

We observe that by $M$-Lipschitzness of $F$ and the definitions of $\mu$, $\tilde{\mu}$ and $Q_j$, we have $f(x+P_i(y)) = f(x + \sum_j a_{ij} Q_j (y)) = f(x) + O_M(\delta) = f(x) + O(\epsilon)$ whenever $\tilde{\mu}(y)>0$. It follows from this that 
\begin{align}\label{equation in popular common differences, 2.5}
\nonumber \EE_{x,y\in\ZZ/N\ZZ} \tilde{\mu}(y)\prod_{i=0}^t f(x + \sum_j a_{ij} Q_j (y))\\
= \left(\EE_{x\in \ZZ/N\ZZ} f(x)^{t+1} + O(\epsilon)\right) \EE_{y\in \ZZ/N\ZZ} \tilde{\mu}(y).
\end{align}
Using the estimate for (\ref{estimate for mu}) and the H\"{o}lder inequality, we deduce that (\ref{equation in popular common differences, 2.5}) is bounded from below by
\begin{align*}
%\EE_{x,y\in[N]} \tilde{\mu}(y)\prod_{i=0}^t f(x + \sum_j a_{ij} Q_j (y))  &= \EE_{x\in [N]} f(x)^{t+1} - O_M(\epsilon^{c'_M})\\
\left(\EE_{x\in \ZZ/N\ZZ} f_{nil}(x)\right)^{t+1} - O(\epsilon)\geqslant \alpha^{t+1} - O(\epsilon),
\end{align*}
where the last inequality follows from the H\"{o}lder inequality.

We now bound terms involving $f_{sml}$. Suppose without loss of generality that $f_{sml}$ is in the $i=0$ position, and let ${f}_1, \ldots, {f}_t\in\{f_{nil}, f_{sml}, f_{unf}\}$. Then
\begin{align}
    \left|\EE_{x,y\in\ZZ/N\ZZ} \tilde{\mu}(y) f_{sml}(x)\prod_{i=1}^t {f}_i(x + P_i (y))\right|\leqslant \norm{f_{sml}}_1 \EE_{y\in\ZZ/N\ZZ}\tilde{\mu}(y)\leqslant\epsilon,
\end{align}
where the first inequality follows from the H\"{o}lder inequality, positivity of $\tilde{\mu}$ and 1-boundedness of ${f}_1, \ldots, {f}_t$. 

It remains to bound the contributions of $f_{unf}$. Using a standard argument (see e.g. the proof of Proposition 3.1 of \cite{green_tao_2012}), we want to approximate $f_{unf}$ by a trigonometric polynomial, which allows us to essentially replace $f_{unf}$ by additive characters.  Let $K\in\NN_+$ be fixed later. Since $\mu$ is an $O_M(\epsilon^{-M})$-Lipschitz function, there exists a trigonometric polynomial $\mu_1:\ZZ/N\ZZ\to\CC$ such that $||\mu-\mu_1||_\infty\ll_M \epsilon^{-C^{(1)}_M} K^{-c}$ for some $0<c, C^{(1)}_M$. Moreover, $\mu_1$ has degree at most $K^M$ and its coefficients satisfy $||\widehat{\mu_1}||_\infty\leqslant||\mu||_\infty\ll_M \epsilon^{-M}$. 

Let $f_0, \ldots, f_t\in\{f_{nil}, f_{sml}, f_{unf}\}$, with at least one of them being $f_{unf}$. We then bound
\begin{align}\label{equation in popular common differences, 3}
    \left|\EE_{y\in\ZZ/N\ZZ}\tilde{\mu}(y)\prod_{i=0}^t f_i(x+P_i(y))\right|&= \left|\EE_{y\in\ZZ/N\ZZ}\prod_{i=1}^k\mu(Q_i(y))\prod_{i=0}^t f_i(x+P_i(y))\right|\\
    \nonumber\leqslant k \max(\norm{\mu}_\infty, \norm{\mu_1}_\infty)^{k-1}\norm{\mu-\mu}_\infty
    &+\left|\EE_{y\in\ZZ/N\ZZ}\prod_{i=1}^k\mu_1(Q_i(y))\prod_{i=0}^t f_i(x+P_i(y))\right|. 
\end{align}
The first term has size at most $C^{(2)}_M \epsilon^{-C^{(2)}_M} K^{-c}$ for some $C^{(2)}_M>0$. The second term is bounded by
\begin{align}\label{equation in popular common differences, 4}
    K^M \norm{\widehat{\mu_1}}_\infty \left|\EE_{y\in\ZZ/N\ZZ}\prod_{i=1}^k\xi_i(Q_i(y))\prod_{i=1}^t f_i(x+P_i(y))\right|
\end{align}
for some characters $\xi_i$ on $\ZZ/N\ZZ$. Since each $Q_i$ is an integral linear combination of $P_i$'s, we can rewrite $\prod_{i=1}^k\xi_i(Q_i(y)) = \prod_{i=1}^t \tilde{\xi}_i(x+P_i(y))$. We let $\tilde{f}_i = f_i \tilde{\xi}_i$. Since each $\tilde{\xi}_i$ is a linear character, we have $\norm{f_i}_{U^2} = ||\tilde{f}_i||_{U^2}$ for each $i$.

We recall from Theorem \ref{Main theorem} that $\vec{P}$ has true complexity 1. Combining this fact with (\ref{equation in popular common differences, 3}), (\ref{equation in popular common differences, 4}) and the bound $||\tilde{f}_i||_{U^2}\leqslant 1/\F(M)$ for some $i$, we deduce that there is some decreasing function $\omega:\RR_+\to\RR_+$, depending only on $\vec{P}$, such that 
\begin{align}\label{equation in popular common differences, 5}
    \left|\EE_{y\in\ZZ/N\ZZ}\tilde{\mu}(y)\prod_{i=0}^t f_i(x+P_i(y))\right|&\leqslant C^{(2)}_M \epsilon^{-C^{(2)}_M} K^{-c} + C^{(2)}_M \epsilon^{-M} K^M \omega(1/\F(M)),
\end{align}
increasing the constant $C^{(2)}_M$ if necessary. We note that the existence of $\omega$ is equivalent to the statement that $\vec{P}$ is controlled by $U^2$ at $i$. We now show that we can choose $K$ large enough and $\F$ growing fast enough so that the right-hand side of (\ref{equation in popular common differences, 5}) is bounded by $O(\epsilon)$. 

For any given $M$, we find a constant $C^{(3)}_M$ such that $(C^{(3)}_M)^c \geqslant C^{(2)}_M$ and $c C^{(3)}_M - C^{(2)}_M \geqslant 1$. We then let $K_M = C^{(3)}_m \epsilon^{-C^{(3)}_M}$, so that
\begin{align*}
    C^{(2)}_M \epsilon^{-C^{(2)}_M} K_M^{-c} = C^{(2)}_M {C^{(3)}_M}^{-c} \epsilon^{c {{C^{(3)}_M} {-C^{(2)}_M}}}\leqslant \epsilon.
\end{align*}
 Picking $\F$ growing sufficiently fast depending on $\epsilon$, we can ensure that $C^{(2)}_M \epsilon^{-M} K_M^M \omega(1/\F(M))\leqslant \epsilon$. We thus set $K = K_M$ for the value of $M$ induced by $\epsilon$ and $\F$, and so 
 \begin{align*}
    \left|\EE_{y\in\ZZ/N\ZZ}\tilde{\mu}(y)\prod_{i=0}^t f_i(x+P_i(y))\right|\leqslant 2\epsilon.
\end{align*}
\end{proof}

\bibliography{library}
\bibliographystyle{alpha}
\end{document}